\documentclass[11pt]{amsart}

\pdfoutput=1

\usepackage{etex}
\usepackage{graphicx}
\usepackage{bbm}
\usepackage{latexsym,wasysym}
\usepackage{soul}
\usepackage[all]{xy}
\usepackage{booktabs}
\usepackage{multirow}
\usepackage{mathrsfs}
\usepackage[margin=1in,marginparwidth=0.8in, marginparsep=0.1in]{geometry}
\usepackage[backref=page, bookmarks=true, bookmarksopen=true,%
bookmarksdepth=3,bookmarksopenlevel=2,%
colorlinks=true,%
linkcolor=blue,%
citecolor=blue,%
filecolor=blue,%
menucolor=blue,%
urlcolor=blue]{hyperref}
\usepackage{amsmath,amsthm,amssymb,amsbsy,amstext,amsopn,amsfonts,psfrag,color,float}
\usepackage{verbatim}
\usepackage{microtype} 
\usepackage{times} 
\usepackage{tikz}
\usetikzlibrary{matrix,shapes,arrows,calc,topaths,intersections,positioning,decorations.pathreplacing,decorations.pathmorphing,fit}

\let\oldtocsection=\tocsection
\let\oldtocsubsection=\tocsubsection
\let\oldtocsubsubsection=\tocsubsubsection

\renewcommand{\tocsection}[2]{\hspace{0em}\oldtocsection{#1}{#2}}
\renewcommand{\tocsubsection}[2]{\hspace{3em}\oldtocsubsection{#1}{#2}}
\renewcommand{\tocsubsubsection}[2]{\hspace{6em}\oldtocsubsubsection{#1}{#2}}

\usepackage{cleveref}
\Crefname{construction}{Construction}{Constructions}

\newtheorem{lemma}{Lemma}[section]
\newtheorem{proposition}[lemma]{Proposition}
\newtheorem{definition/proposition}[lemma]{Definition/Proposition}
\newtheorem{corollary}[lemma]{Corollary}
\newtheorem{theorem}[lemma]{Theorem}

\theoremstyle{definition}
\newtheorem{definition}[lemma]{Definition}
\newtheorem{remark}[lemma]{Remark}
\newtheorem{construction}[lemma]{Construction}
\newtheorem{example}[lemma]{Example}

\newcommand{\B}{\mathbb{B}}
\newcommand{\C}{\mathbb{C}}
\newcommand{\D}{\mathbb{D}}

\newcommand{\F}{\mathbb{F}}
\newcommand{\G}{\mathbb{G}}

\newcommand{\N}{\mathbb{N}}

\renewcommand{\P}{\mathbb{P}}

\newcommand{\R}{\mathbb{R}}

\newcommand{\Z}{\mathbb{Z}}

\newcommand{\cA}{\mathcal{A}}

\newcommand{\cC}{\mathcal{C}}
\newcommand{\cD}{\mathcal{D}}

\newcommand{\cF}{\mathcal{F}}

\newcommand{\cH}{\mathcal{H}}
\newcommand{\cI}{\mathcal{I}}

\newcommand{\cL}{\mathcal{L}}
\newcommand{\cM}{\mathcal{M}}
\newcommand{\cN}{\mathcal{N}}

\newcommand{\cS}{\mathcal{S}}

\newcommand{\cX}{\mathcal{X}}

\newcommand{\bH}{\mathbf{H}}

\newcommand{\bP}{\mathbb{P}}

\newcommand{\bR}{\mathbf{R}}

\newcommand{\coeffs}{\mathbbm{k}}
\newcommand{\e}{\epsilon}

\newcommand{\GL}{\mathrm{GL}}

\newcommand{\Sh}{\mathit{Sh}}
\newcommand{\Loc}{\mathit{Loc}}

\DeclareMathOperator{\Gr}{Gr}
\DeclareMathOperator{\Meas}{\B}

\DeclareMathOperator{\im}{\F} 
\newcommand\onto{\twoheadrightarrow}
\newcommand\into{\hookrightarrow}

\newcommand{\La}{\Lambda}

\newcommand{\sheafhom}{\mathscr Hom}
\newcommand{\congto}{\xrightarrow{\sim}}
\newcommand\dmod{\textrm{-mod}}
\newcommand*{\hrlen}{10}
\newcommand*{\hramp}{3}
\newcommand*{\fcolor}{black!5}
\newcommand*{\bvertrad}{.09}
\newcommand*{\wvertrad}{.06}
\newcommand*{\ifac}{.26}

\tikzset{
asdstyle/.style={blue,thick},
righthairs/.style={postaction={decorate,draw,decoration={border,amplitude=\hramp,segment length=\hrlen,angle=-90,pre=moveto,pre length=\hrlen/2}}},
lefthairs/.style={postaction={decorate,draw,decoration={border,amplitude=\hramp,segment length=\hrlen,angle=90,pre=moveto,pre length=\hrlen/2}}},
righthairsnogap/.style={postaction={decorate,draw,decoration={border,amplitude=\hramp,segment length=\hrlen,angle=-90}}},
lefthairsnogap/.style={postaction={decorate,draw,decoration={border,amplitude=\hramp,segment length=\hrlen,angle=90}}},
graphstyle/.style={thick},
arrowstyle/.style={thick,decorate,decoration={snake,amplitude=1.7,segment length=10pt,post length=.5mm,pre length=0}},
genmapstyle/.style={thick,-stealth'},
arrhdstyle/.style={thick},
exceptarcstyle/.style={red, ultra thick},
dualquiverstyle/.style={thick,->}
}

\DeclareMathOperator{\Ext}{Ext}
\DeclareMathOperator{\End}{End}
\DeclareMathOperator{\Hom}{Hom}
\DeclareMathOperator{\Aut}{Aut}

\DeclareMathOperator{\Spec}{Spec}

\newcommand{\cunknot}{\mathbin{\text{\rotatebox[origin=c]{90}{$(\!)$}}}}

\author{Vivek Shende, David Treumann, Harold Williams, and Eric Zaslow}

\title{Cluster varieties from Legendrian knots}

\numberwithin{equation}{subsection}

\begin{document}

\begin{abstract}
Many interesting spaces --- including all positroid strata and wild character varieties --- 
are moduli of constructible sheaves on a surface with microsupport in a Legendrian link.
We show that the existence of cluster structures on these spaces may be deduced in a uniform, systematic fashion by constructing and taking the sheaf quantizations of a set of exact Lagrangian fillings in correspondence with isotopy representatives whose front projections have crossings with alternating orientations. 
It follows in turn that results in cluster algebra may be used to construct and distinguish exact Lagrangian 
fillings of Legendrian
links in the standard contact three space. 
\end{abstract}

\maketitle

{\small \tableofcontents}

\newpage

\section{Introduction}
The moduli space of decorated
local systems on a punctured surface admits a cluster structure --- which is to say, 
a space of  {\em nonabelian} representations of the fundamental group 
can be built out of algebraic tori, which are objects of an {\em abelian} nature.  
The original proof of this fact in \cite{FG1} is constructive and combinatorial: for each ideal triangulation of the surface one defines a toric coordinate system using invariants of configurations of flags. An orthogonal geometric perspective was provided in \cite{GMN2}: given a complex structure on the surface and a spectral curve of the associated Hitchin system, one defines a map from abelian local systems on the curve to nonabelian local systems on the base by studying trajectories of holomorphic differentials.

One could ask for an account which is complementary to these in the following sense: rather than start by prescribing a collection of toric coordinate systems by hand, one should be able to start from general geometric principles and deduce abstractly that the moduli space will be populated by toric charts -- the explicit form of these charts then becoming the result of a calculation rather than a definition. To provide a fully satisfactory foundation for the theory it should be clear in advance that the transition functions between charts will have a universal form. One would further want the standard package of cluster combinatorics, such as triangulations and bipartite graphs, to emerge as a natural byproduct. Finally, since the moduli of local systems only depends on the topology of the surface, the cluster structure should in principle be visible without referring to any intermediate choice of complex structure.

Clues to such an approach are provided by the following known connections between cluster theory and symplectic geometry. It was observed in \cite{Thu} that cluster algebras may be formally associated to connectivity classes of alternating triple point diagrams, or equivalently isotopy classes of certain Legendrian knots. These diagrams appeared later as encodings of bipartite graphs in the theory of positroids \cite{Po} and a number of other cluster-algebraic contexts \cite{Gon,GK}. In another direction, the spectral curves in \cite{GMN2} are in particular (holomorphic) Lagrangians in the cotangent bundle of the base. Finally, in symplectic geometry cluster transformations are known to 
appear in the context of wall crossing: in particular, given a family of exact Lagrangian surfaces, smooth away from
the appearance at one instant of a single double point, the families of objects defined
by rank one local systems before and after the critical moment are algebraic tori related by a cluster transformation (c.f. \cite{A} or \cite[Lec. 11]{Se}, for example).

\vspace{2mm}

Here we will give an account of the existence of cluster structures on moduli spaces of local systems 
which begins with the above geometric structures --- Legendrian knots and Lagrangian surfaces which fill them --- 
and arrives at explicit coordinate systems 
only as the result of calculations rather than prescriptions.   
The knots and fillings will live in the cosphere and cotangent bundles of the base surface. Decorated local systems, that is local systems with extra data at punctures, arise via sheaf quantization.

Recall that a basic form of quantization takes functions on Lagrangians in a cotangent bundle $T^*M$ to distributions on $M$, and symplectomorphisms to operators given by integral kernels \cite{BW}. By analogy, sheaf quantization takes local systems on Lagrangians in $T^*M$ to constructible sheaves on $M$,
and conic symplectomorphisms to autoequivalences of the sheaf category given by integral kernels.

The essential theorems of sheaf quantization may be obtained by passing through the Fukaya category following \cite{NZ,N} or independently of Floer theory as in \cite{Tam, GKS, Gui}. 
The key results we need are the following.  In \cite{KS}, the 
microsupport of a sheaf on a manifold is defined; it is a conical co-isotropic locus in the cotangent bundle 
measuring the failure of local propagation of sections.  For the sheaves of interest here, this locus is a stratifiable (generally singular) 
conical Lagrangian whose boundary is a Legendrian link in the cosphere bundle.  One can study the subcategory of sheaves with fixed microsupport.  In \cite{GKS}, it is shown
that `contact isotopies quantize,' i.e. that given a contact isotopy there is a unique family of sheaf integral kernels 
such that the corresponding autoequivalences of the sheaf category act on microsupports by the specified isotopy.

In \cite{Gui, JinTreumann} (or \cite{NZ,N}) it is shown that `Lagrangians quantize', i.e. that given an eventually conical exact Lagrangian $L \subset T^*M$ there is a fully faithful functor from locally constant (Maslov-twisted) sheaves on $L$ to sheaves on $M$ whose microsupport at infinity is $\partial L$.
The Lagrangians we study here will be simple enough that we can construct the quantization functor by hand, independently of any general results. 

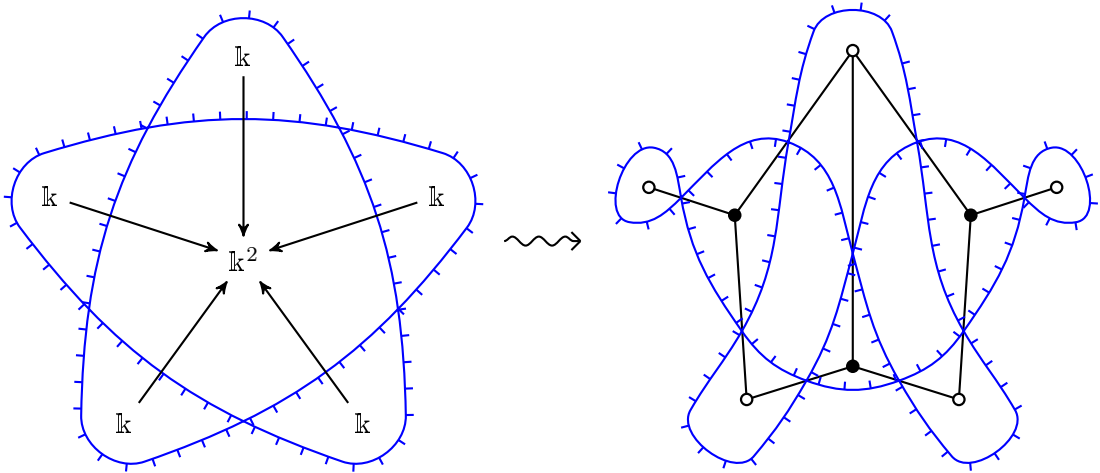
\begin{figure}
\centering
\begin{tikzpicture}
\newcommand*{\boff}{10}; \newcommand*{\aoff}{35}; \newcommand*{\rad}{3};
\node (r) [matrix] at (2.7*\rad,0) {
\coordinate (b0) at (90-\boff:\rad); \coordinate (b1) at (90+\boff:\rad);
\coordinate (b2) at (90+72-\boff:\rad); \coordinate (b3) at (90+72+\boff:\rad);
\coordinate (b4) at (90+2*72-\boff:\rad); \coordinate (b5) at (90+2*72+\boff:\rad);
\coordinate (b6) at (90+3*72-\boff:\rad); \coordinate (b7) at (90+3*72+\boff:\rad);
\coordinate (b8) at (90+4*72-\boff:\rad); \coordinate (b9) at (90+4*72+\boff:\rad);
\coordinate (c0) at (60:.57*\rad); \coordinate (c1) at (120:.57*\rad);
\coordinate (c2) at (18:.8*\rad); \coordinate (c3) at (180-18:.8*\rad);
\coordinate (c4) at (-35:.6*\rad); \coordinate (c5) at (215:.6*\rad);
\coordinate (c6) at (-70:.6*\rad); \coordinate (c7) at (180+70:.6*\rad);
\coordinate (c8) at (90:0*\rad); \coordinate (w0) at (90:.9*\rad);
\coordinate (w1) at (180-18:.95*\rad); \coordinate (w2) at (18:.95*\rad);
\coordinate (w3) at (-126:.8*\rad); \coordinate (w4) at (180+126:.8*\rad);
\coordinate (l0) at (-90:.5*\rad); \coordinate (l1) at (180-18:.55*\rad);
\coordinate (l2) at (18:.55*\rad);
\draw[graphstyle] (w0) -- (l1) -- (w3) -- (l0) -- (w4) -- (l2) -- (w0);
\draw[graphstyle] (w1) -- (l1); \draw[graphstyle] (w2) -- (l2); \draw[graphstyle] (w0) -- (l0);
\foreach \c in {w0,w1,w2,w3,w4,l0,l1,l2} {\fill (\c) circle (\bvertrad);}
\foreach \c in {w0,w1,w2,w3,w4} {\fill[white] (\c) circle (\wvertrad);}
\draw[asdstyle,righthairs] (b0) to[out=110,in=70] (b1) to[out=180+70,in=80] (c1) to[out=180+80,in=60] (c5) to[out=180+60,in=60] (b4) to[out=180+60,in=230] (b5) to[in=-120,out=180-130] (c7) to[in=-105,out=180-120] (c8) to[in=180+20,out=180-105] (c0) to[in=-225,out=20] (c2) to[in=-170,out=180-225] (b8) to[in=10,out=180-170] (b9) to[in=180-100,out=-170] (c2) to[in=180-125,out=-100] (c4) to[in=180+200,out=-125] (c6) to[in=180-200,out=180+20] (c7) to[in=180+125,out=-200] (c5) to[in=180+100,out=125] (c3)  to[in=180+170,out=100] (b2) to[out=180-10,in=170] (b3) to[out=180+170,in=225] (c3) to[out=180+225,in=180-20] (c1) to[out=180+180-20,in=105] (c8) to[out=180+105,in=120] (c6) to[out=180+120,in=130] (b6) to[in=-60,out=180-230] (b7) to[in=-60,out=180-60] (c4) to[in=-80,out=180-60] (c0) to[in=-70,out=180-80] (b0);\\};
\node (l) [matrix] at  (0,0) {
\coordinate (b0) at (90-\boff:\rad); \coordinate (b1) at (90+\boff:\rad);
\coordinate (b2) at (90+72-\boff:\rad); \coordinate (b3) at (90+72+\boff:\rad);
\coordinate (b4) at (90+2*72-\boff:\rad); \coordinate (b5) at (90+2*72+\boff:\rad);
\coordinate (b6) at (90+3*72-\boff:\rad); \coordinate (b7) at (90+3*72+\boff:\rad);
\coordinate (b8) at (90+4*72-\boff:\rad); \coordinate (b9) at (90+4*72+\boff:\rad);
\draw[asdstyle,righthairs] (b0) to[out=90+\aoff,in=90-\aoff] (b1) to[out=180+90-\aoff,in=180+2*72+90+\aoff] (b4) to[out=2*72+90+\aoff,in=2*72+90-\aoff] (b5) to[out=180+2*72+90-\aoff,in=180-1*72+90+\aoff] (b8) to[out=-1*72+90+\aoff,in=-1*72+90-\aoff] (b9) to[out=180-1*72+90-\aoff,in=180+1*72+90+\aoff] (b2) to[out=1*72+90+\aoff,in=1*72+90-\aoff] (b3) to[out=180+1*72+90-\aoff,in=180-2*72+90+\aoff] (b6) to[out=-2*72+90+\aoff,in=-2*72+90-\aoff] (b7) to[out=180-2*72+90-\aoff,in=180+90+\aoff] (b0);
\node (v0) at (90:.9*\rad) {$\coeffs$}; \node (v1) at (90+72:.9*\rad) {$\coeffs$};
\node (v2) at (90+2*72:.9*\rad) {$\coeffs$}; \node (v3) at (90+3*72:.9*\rad) {$\coeffs$};
\node (v4) at (90+4*72:.9*\rad) {$\coeffs$}; \node (w) at (0,0) {$\coeffs^2$};
\foreach \n in {v0,v1,v2,v3,v4} {\draw[genmapstyle] (\n) to (w);}\\};
\coordinate (rw) at ($(r.west)+(-1mm,0)$); \coordinate (le) at ($(l.east)+(1mm,0)$);
\draw[arrowstyle] (le) -- (rw);
\draw[arrhdstyle] ($(rw)+(-1.2mm,1.2mm)$) -- (rw); \draw[arrhdstyle] ($(rw)+(-1.2mm,-1.2mm)$) -- (rw);
\end{tikzpicture}\caption{ \label{fig:introGr24} On the left, a Legendrian braid closure whose rank-one moduli space is the open positroid stratum in $\Gr(2,5)$. The Legendrian lives in the cocircle bundle of the page;  the hairs drawn along the immersed curve indicate the conormal directions in which it is lifted from its projection. We can isotope it to become alternating as on the right, obtaining an associated bipartite graph.
If we restrict our attention to a disk whose boundary passes through the five white vertices we obtain a reduced plabic graph in the terminology of \cite{Po}.}
\end{figure}

We can now describe our construction.  Spaces of local
systems with invariant flags at punctures can be understood as moduli spaces of sheaves 
microsupported at infinity on Legendrians projecting to concentric circles around the punctures.  
A similar description applies to moduli of `wild' local systems (Section~\ref{sec:wild}) and open positroid varieties
(Theorem~\ref{thm:positroidasmoduli}) by considering Legendrian braid closures.  

Rank one local systems on an
exact Lagrangian filling $L$ of such a Legendrian are parameterized by the complex torus 
$\Hom(\pi_1(L), \C^*) \cong (\C^*)^{b_1(L)}$.  By sheaf quantization, this determines a chart on moduli, specifically of sheaves whose microstalks on the Legendrian boundary have rank one.  
We may move the Legendrian around until we are best able to see such a filling, then carry it back to a filling of the original link. By the quantization of contact isotopies we can capture this process sheaf-theoretically. 

We show that when the Legendrian has been isotoped so that its front projection has crossings of alternating orientations one obtains a natural exact Lagrangian filling (Proposition \ref{prop:conjlagexistence}). 
Such an isotopy transforms the front projection into the alternating
strand diagram of a bipartite graph embedded in the base surface. The filling retracts onto this graph and is an exact Lagrangian embedding of the conjugate surface of \cite{GK}. Holonomies around the faces of the graph form a natural coordinate system on the associated chart.

In the case of positroid varieties, we show that the isotopy from the defining braid of the positroid
to any given alternating representative is unique up to homotopy (Proposition \ref{prop:uniqueisotopies}). By the general principles discussed above, quantizing the associated Lagrangian and isotopy results in a toric chart on the positroid variety. We calculate this chart explicitly and show that it is exactly the boundary measurement map defined in \cite{Po} (Theorem \ref{thm:positroid}).

Following \cite{Po,Thu} a {\em square move} of bipartite graphs induces a Legendrian
isotopy from one alternating Legendrian to another. 
The corresponding Lagrangian fillings differ
by a Lagrangian surgery (Proposition \ref{prop:squaremovesurgery}).  
We calculate that the resulting charts differ by a cluster $\cX$-transformation (Theorem~\ref{prop:localclustertrans}). 
This is a sheaf-theoretic incarnation of the known relation between surgeries and cluster transformations in Floer theory. We emphasize that even before the calculation, it is immediate from the local nature of sheaves and the square move that the transition function will have a universal form independent of the global structure of the graph. 

\begin{figure}
\centering
\begin{tikzpicture}
\newcommand*{\off}{18};\newcommand*{\rad}{2.3};\newcommand*{\vrad}{1.0};
\node (l) [matrix] at (0,0) {
\coordinate (bl) at (180:\rad/2); \coordinate (br) at (0:\rad/2);
\coordinate (wt) at (90:\rad*.8); \coordinate (wb) at (-90:\rad*.8);
\coordinate (wr) at (0:\rad*.9); \coordinate (wl) at (180:\rad*.9);
\coordinate (ur) at ($(-\off:\rad)!\ifac!(-180+\off:\rad)$); \coordinate (lr) at ($(\off-180:\rad)!\ifac!(-\off:\rad)$);
\coordinate (ul) at ($(\off:\rad)!\ifac!(-180-\off:\rad)$); \coordinate (ll) at ($(-\off-180:\rad)!\ifac!(\off:\rad)$);
\coordinate (mr) at ($(ur)!.5!(lr)$); \coordinate (ml) at ($(ul)!.5!(ll)$);
\path[fill=\fcolor, name path=p1] (90+\off:\rad) -- (-90-\off:\rad) arc[start angle=-90-\off,end angle=-90+\off,radius=\rad] (-90+\off:\rad) -- (90-\off:\rad) arc[start angle=90-\off,end angle=90+\off,radius=\rad] (90+\off:\rad);
\path[fill=\fcolor, name path=p2] (180-\off:\rad)  to[out=0,in=-180] (lr) to (ur) [out=0,in=-180] to (\off:\rad) arc[start angle=\off,delta angle=-2*\off,radius=\rad] (-\off:\rad) [out=-180,in=0] to (ul) to (ll) to[out=-180,in=0] (180+\off:\rad) arc[start angle=180+\off,delta angle=-2*\off,radius=\rad] (180-\off:\rad);
\path[fill=white, name intersections={of=p1 and p2, name=i}] (i-1) -- (i-2) -- (i-4) -- (i-3) -- cycle;
\draw[dashed] (0,0) circle (\rad);
\draw[graphstyle] (wb) to (br) to (wt) to (bl) to (wb);
\draw[graphstyle] (br) to (wr) to (0:\rad);
\draw[graphstyle] (bl) -- (wl) -- (180:\rad);
\draw[graphstyle] (wt) to (90:\rad); \draw[graphstyle] (wb) to (-90:\rad);
\foreach \c in {bl,br,wt,wb,wr,wl} {\fill[black] (\c) circle (\bvertrad);}
\foreach \c in {wt,wb,wr,wl} {\fill[white] (\c) circle (\wvertrad);}
\draw[asdstyle,righthairs] (90+\off:\rad) -- (-90-\off:\rad);
\draw[asdstyle,righthairs] (-90+\off:\rad) -- (90-\off:\rad);
\draw[asdstyle,lefthairs] (mr) to (ur) [out=0,in=180] to (\off:\rad);
\draw[asdstyle,righthairs] (mr) to (lr) to[in=0,out=180] (180-\off:\rad);
\draw[asdstyle,righthairs] (ml) to (ul) [out=0,in=180] to (-\off:\rad);
\draw[asdstyle,lefthairs] (ml) to (ll) to[in=0,out=180] (180+\off:\rad);\\};

\node (m) [matrix] at (2.5*\rad,0) {
\path[fill=\fcolor] (0,0) to[out=-90,in=90]  (-90-\off:\rad) arc[start angle=-90-\off,delta angle=2*\off,radius=\rad] (-90+\off:\rad) to[out=90,in=-90] (0,0);
\path[fill=\fcolor] (0,0) to[out=0,in=180]  (-\off:\rad) arc[start angle=-\off,delta angle=2*\off,radius=\rad] (\off:\rad) to[out=180,in=0] (0,0);
\path[fill=\fcolor] (0,0) to[out=90,in=-90]  (90-\off:\rad) arc[start angle=90-\off,delta angle=2*\off,radius=\rad] (90+\off:\rad) to[out=-90,in=90] (0,0);
\path[fill=\fcolor] (0,0) to[out=180,in=0]  (180-\off:\rad) arc[start angle=180-\off,delta angle=2*\off,radius=\rad] (180+\off:\rad) to[out=0,in=180] (0,0);
\draw[dashed] (0,0) circle (\rad);
\draw[asdstyle,lefthairsnogap] (0,0) to[out=90,in=-90] (90+\off:\rad);
\draw[asdstyle,righthairsnogap] (0,0) to[out=-90,in=90]  (-90-\off:\rad);
\draw[asdstyle,righthairsnogap] (0,0) to[out=90,in=-90]  (90-\off:\rad);
\draw[asdstyle,lefthairsnogap] (0,0) to[out=-90,in=90]  (-90+\off:\rad);
\draw[asdstyle,lefthairsnogap] (0,0) to[out=0,in=180] (\off:\rad);
\draw[asdstyle,righthairsnogap] (0,0) to[out=180,in=0] (180-\off:\rad);
\draw[asdstyle,righthairsnogap] (0,0) to[out=0,in=180] (-\off:\rad);
\draw[asdstyle,lefthairsnogap] (0,0) to[out=180,in=0] (180+\off:\rad);\\};

\node (r) [matrix] at (5*\rad,0) {
\coordinate (bl) at (180+90:\rad/2); \coordinate (br) at (0+90:\rad/2);
\coordinate (wt) at (90+90:\rad*.8); \coordinate (wb) at (-90+90:\rad*.8);
\coordinate (wr) at (0+90:\rad*.9); \coordinate (wl) at (180+90:\rad*.9);
\coordinate (ur) at ($(-\off+90:\rad)!\ifac!(-90+\off:\rad)$); \coordinate (lr) at ($(\off-90:\rad)!\ifac!(90-\off:\rad)$);
\coordinate (ul) at ($(\off+90:\rad)!\ifac!(-90-\off:\rad)$); \coordinate (ll) at ($(-\off-90:\rad)!\ifac!(90+\off:\rad)$);
\coordinate (mr) at ($(ur)!.5!(lr)$); \coordinate (ml) at ($(ul)!.5!(ll)$);
\path[fill=\fcolor, name path=p1] (90+90+\off:\rad) -- (-90+90-\off:\rad) arc[start angle=-\off,end angle=\off,radius=\rad] -- (180-\off:\rad) arc[start angle=180-\off,end angle=180+\off,radius=\rad];
\path[fill=\fcolor, name path=p2] (180+90-\off:\rad)  to[out=90,in=-90] (lr) to (ur) [out=90,in=-90] to (\off+90:\rad) arc[start angle=90+\off,delta angle=-2*\off,radius=\rad] [out=-90,in=90] to (ul) to (ll) to[out=-90,in=90] (180+90+\off:\rad) arc[start angle=180+90+\off,delta angle=-2*\off,radius=\rad];
\path[fill=white, name intersections={of=p1 and p2, name=i}] (i-1) -- (i-2) -- (i-4) -- (i-3) -- cycle;
\draw[dashed] (0,0) circle (\rad);
\draw[graphstyle] (wb) to (br) to (wt) to (bl) to (wb);
\draw[graphstyle] (br) to (wr) to (90:\rad);
\draw[graphstyle] (bl) -- (wl) -- (180+90:\rad);
\draw[graphstyle] (wt) to (90+90:\rad); \draw[graphstyle] (wb) to (-90+90:\rad);
\foreach \c in {bl,br,wt,wb,wr,wl} {\fill[black] (\c) circle (\bvertrad);}
\foreach \c in {wt,wb,wr,wl} {\fill[white] (\c) circle (\wvertrad);}
\draw[asdstyle,righthairs] (90+90+\off:\rad) -- (-90+90-\off:\rad);
\draw[asdstyle,righthairs] (-90+90+\off:\rad) -- (90+90-\off:\rad);
\draw[asdstyle,lefthairs] (mr) to (ur) [out=0+90,in=180+90] to (\off+90:\rad);
\draw[asdstyle,righthairs] (mr) to (lr) to[in=0+90,out=180+90] (180+90-\off:\rad);
\draw[asdstyle,righthairs] (ml) to (ul) [out=0+90,in=180+90] to (-\off+90:\rad);
\draw[asdstyle,lefthairs] (ml) to (ll) to[in=0+90,out=180+90] (180+90+\off:\rad);\\};
\coordinate (rw) at ($(r.west)+(-1mm,0)$); \coordinate (me) at ($(m.east)+(1mm,0)$);
\coordinate (mw) at ($(m.west)+(-1mm,0)$); \coordinate (le) at ($(l.east)+(1mm,0)$);
\foreach \ca/\cb in {le/mw,me/rw} {
\draw[arrowstyle] (\ca) -- (\cb);
\draw[arrhdstyle] ($(\cb)+(-1.2mm,1.2mm)$) -- (\cb); \draw[arrhdstyle] ($(\cb)+(-1.2mm,-1.2mm)$) -- (\cb);}
\end{tikzpicture}
\caption{The left and right frames show the front projections of alternating Legendrians related by a square move of bipartite graphs.  The canonical Legendrian isotopy between them passes through alternating Legendrians except at one moment pictured in the middle.  This Legendrian has a singular exact filling which meets the cotangent space of the origin in the union of the conormal lines to the $x$- and $y$-axes.  Altogether we have a family of exact fillings that undergoes a Lagrangian surgery.}
\label{fig:squaremoveintro}
\end{figure}
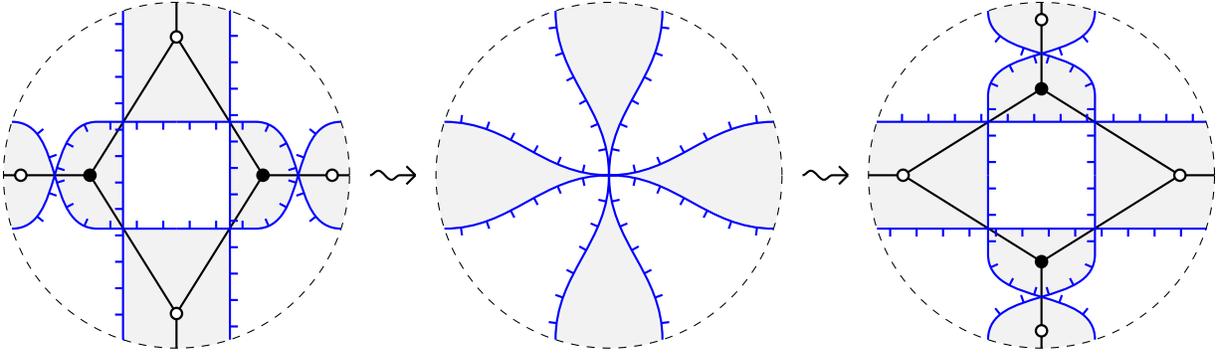

Though we have focused on positroid varieties and (wild) character varieties, the scope of our discussion is really that of moduli spaces of sheaves microsupported on any Legendrian link that admits an alternating representative. We show, for example, that collections of arbitrary Legendrian braids around punctures, of which the Stokes diagrams of wild character varieties form a subset, are of this type (Theorem \ref{thm:wildcharvty}). 

We also emphasize the following novel aspect of our framework.  
Each bicolored graph 
(equivalently, each 
alternating Legendrian) 
determines {\em at once} both the 
algebraic torus of local systems on the filling, and the 
larger moduli space of sheaves with microsupport
in the Legendrian.  This latter moduli space has a direct global definition 
which is a priori invariant under Legendrian isotopy, in particular under the square move. 
Thus we do not need to define it independently of the graph nor show by hand that it admits charts corresponding to arbitrary sequences of square moves. 
Instead, the fact that it possesses an atlas of toric charts related by cluster transformations is deduced 
in a universal fashion from our construction of Lagrangian fillings.

\vspace{2mm}
The fact that the fillings of a knot give rise to a cluster structure can be used to deduce
consequences in symplectic geometry.  We conclude the paper with one example.  

Sheaf quantization (or Floer theory) 
implies that Lagrangian fillings of Legendrians determine objects in a certain
category; one can in principle distinguish them by computing Homs.  In practice this can be somewhat unwieldy.
Here we observe that the cluster structures we have constructed allow a different argument.  Each Lagrangian
filling determines a chart on a moduli space, and we know that the relation of charts is governed by cluster algebra.  
That is, we can use results in cluster algebra to show that certain Lagrangians give different charts.  But, we know
that Hamiltonian isotopic fillings must give the same chart.  

In Section \ref{sec:rainspir} we show how to apply this, for example, to construct and distinguish fillings of arbitrary Legendrian positive braid closures in $\R^3$. In simple examples this recovers and distinguishes known fillings via purely sheaf-theoretic techniques. In \cite{EHK} a
Catalan number $C_n = \frac{1}{n+1}{2n \choose n}$ of fillings was constructed of the $(2,n)$ torus link. 
Our framework reconstructs these fillings and distinguishes them from one another --- their distinctness
corresponds to the combinatorial fact that the $A_{n-1}$ cluster structure has 
$C_n$ clusters \cite{FZ2}. Our approach also highlights the surgery relations among these, for example explicitly producing a surgery relating fillings of the $(2,n)$ link for each edge of the associahedron. Finally, we believe our constructions make clear that the combinatorics of cluster algebras should play a governing role in the theory of exact Lagrangian fillings. For example, the classification of which cluster algebras of Grassmannians are of finite type \cite{Sco} should be closely related to the classification of which Legendrian torus links have finitely many inequivalent fillings (some progress towards this question is made in \cite{STW}). 

\vspace{2mm}\noindent {\bf Acknowledgements.}
We would like to thank Sheng-Fu Chiu, Andy Neitzke, Xin Jin, Emmy Murphy and Lenny Ng for their help and contributions.
This work was supported in part by a SQuaRE program at the American Institute of Mathematics.
V. S. was supported by NSF grant DMS-1406871 and a Sloan fellowship.  D. T. was supported by NSF grant DMS-1510444.   H. W. was supported by an NSF Postdoctoral Research Fellowship DMS-1502845.  E. Z. was supported by NSF grant DMS-1406024.  
\section{Constructible sheaves, microlocalization, and moduli} \label{sec:moduli}

In this section we review the relevant background on constructible sheaves,
including their singular support, microlocalization, moduli spaces, and invariance under contact transformations.  We refer the reader to \cite{KS} for detailed foundations.  We note for the expert that here we adopt certain conventions adapted to the fact that we work throughout with Legendrians whose front projection is an immersion and we can therefore canonically
trivialize all Maslov obstructions. 

Throughout we fix a commutative ring $\coeffs$. For a manifold $X$ write $\Sh(X) := \Sh(X;\coeffs)$ for the differential graded derived category of constructible sheaves on $X$ -- the dg category of constructible sheaves of perfect $\coeffs$-modules on $X$, localized at the acyclic complexes. 
We refer to \cite{Kel,Toe3} for background on dg categories.
We write e.g. isomorphism instead of quasi-isomorphism when no confusion should arise.

\subsection{Singular support}
Given a sheaf $\cF \in \Sh(X)$, the \textbf{singular support} (or \textbf{microsupport}) $SS(\cF)$ is a closed, conic, Lagrangian subset of $T^*X$.  The singular support at infinity of $\cF$ is the Legendrian image of $SS(\cF)$
in the cocircle bundle $T^\infty X := (T^*X \setminus 0_X)/\R_+$, where $0_X \subset T^*X$ denotes the zero section.  These notions are meant to capture the locus in $T^*M$ of obstructions to the propagation of sections of $\cF$ (see \Cref{ex:A2,ex:crossing}). 
For instance, if $f: M \to \R$ is a function such that the graph of $df$ avoids $SS(\cF)$ over the 
locus $f^{-1}((a, b])$, then the restriction of sections is an isomorphism \cite[Prop. 5.2.1]{KS}:
 $$H^*(f^{-1}(-\infty, b], \cF) \xrightarrow{\sim} H^*(f^{-1}(-\infty, a], \cF)$$
The formal definition is a local version of the above criterion: 

\begin{definition} \label{def:microsupport} \cite[Chap. 5]{KS}
A point $p = (x, \xi) \in T^*X$ is in the microsupport of a sheaf $\cF$ if there are 
points $(x', \xi')$ arbitrarily close to $(x, \xi)$ and functions $f: M \to \R$ 
with $f(x') = 0, df(x') = \xi'$, such that the following property holds:
if $c_f: \{x\,|\, f(x) \ge 0\} \to M$ is the inclusion, then $(c_f^! \cF)_{x'} \ne 0$. 
\end{definition}

Shriek pullback to a closed subset gives the local sections  
supported on that subset.  Thus the statement $(c_f^! \cF)_{x'} \ne 0$ is informally
read as: ``there is a section of $\cF$ beginning at $x'$ and propagating in 
the direction along which $f$ increases.''  Note that, taking the zero function, the support
of $\cF$ is contained in its microsupport. 

\begin{definition}
Given a Legendrian $\Lambda \subset T^\infty X$, we write $\Sh_\Lambda(X) := \Sh_\Lambda(X;\coeffs)$ for the full subcategory of $\Sh_\Lambda(X)$ consisting of sheaves whose singular support at infinity is contained in $\Lambda$. Note that every locally constant sheaf belongs to $\Sh_{\Lambda}(X)$.  If $\sigma$ is a set of points in $X$ not meeting the front projection of $\Lambda$, we write $\Sh_\Lambda(X, \sigma)$ for the full subcategory of $\Sh_\Lambda(X)$ consisting of sheaves that vanish on $\sigma$.  
\end{definition}

The subcategories $\Sh_\Lambda(X)$, $\Sh_\Lambda(X, \sigma)$ are triangulated, since given a triangle $A \to B \to C \xrightarrow{[1]}$ we have $SS(C) \subset SS(A) \cup SS(B)$.  Any sheaf in $\Sh_\Lambda(X, \sigma)$ vanishes not only on $\sigma$, but on each component of $\Sigma\smallsetminus\pi(\Lambda)$ containing a point in $\sigma$.  

A key principle of \cite{KS} is that a sheaf $\cF$ localizes not just over $X$, but ``microlocalizes'' over the cotangent bundle $T^*X$ and its own singular support $SS(\cF)$ in particular.  There is a dg category $\mu loc(\Lambda)$, the category of microlocal sheaves on $\Lambda$, and a functor 
$$\mu: \Sh_{\Lambda}(X) \to \mu loc(\Lambda).$$

The category $\mu loc(\Lambda)$ is defined as follows.  Following \cite{KS}, for $\Omega \subset T^*X$ 
one takes the category $\mu sh^{pre}(\Omega) := Sh(X) / Sh_{T^*X \setminus \Omega}(X)$.  This is a presheaf of dg categories
on $T^*X$, we write $\mu sh$ for its sheafification.  One shows that the (micro) support of an object in $\mu sh(\Omega)$
is a well defined conical Lagrangian in $\Omega$, and that this construction is respected by restriction.  It follows that there is a subsheaf 
$\mu sh_\Lambda$ on
full subcategories on objects supported in $\Lambda$; we write $\mu loc (\Lambda) := \mu sh_\Lambda(\Lambda)$. 

The local sections $\mu sh_\Lambda$ can be understood explicitly in certain cases.  When 
$\Lambda \to X$ is finite, and $U \subset X$ is sufficiently small, then $\mu loc_\Lambda(U)$ is the quotient of $Sh_\Lambda(U)$ by
local systems on $U$.  This quotient is in turn equivalent to the subcategory of $Sh_\Lambda(U)$ of objects with stalk zero at 
any specified point (though the equivalence depends on the point).  I.e., in this case the presheaf-of-categories description remains 
true after sheafifying.  A sufficient condition that a chart is small enough is that the stratification is $C^1$-conical.  
One generally computes $\mu loc(\Lambda)$ by finding enough such charts and gluing.  (In case $\Lambda \to X$ is not finite, 
one can use the local invariance of sheaves under contact transformations to make it so.)

In particular, when $\Lambda$ is smooth, $\mu loc(\Lambda)$ is itself locally equivalent to $Loc(\Lambda)$.  
This is seen at a point by applying a contact transformation to make $\Lambda$ locally the conormal to a smooth hypersurface. 
In general, there is only a `Maslov' obstruction to globalizing this 
equivalence.  A trivialization $\mu loc(\Lambda) \cong Loc(\Lambda)$ is determined by a Maslov 
potential on $\Lambda$; see \cite{STZ} for a detailed discussion in the case where $\Lambda$ is one dimensional, or \cite{Gui} for a general account. 

In this paper we generally do not consider front projections with cusps.  As such we can
always take the zero Maslov potential, and identify $\mu loc(\Lambda) \cong Loc(\Lambda)$.  
Given a sheaf $\cF \in \Sh_\Lambda(X)$, we write $\cF|_\Lambda$ 
for its image in $Loc(\Lambda)$.

\begin{definition}
A sheaf $\cF \in \Sh_\Lambda(X)$ has \textbf{microlocal rank $n$} if the microlocalization 
of $\cF$ along $\Lambda$ is a local system of locally free $\coeffs$-modules of rank $n$ supported in degree zero.  We denote by $\cC_n( X, \Lambda)$ the full subcategory of $\Sh_\Lambda(X)$ consisting of microlocal rank-$n$ sheaves, similarly for $\cC_n(\Lambda,\sigma)$. 
When $X$ is fixed or clear from context,  we omit it from the notation.
\end{definition}

If $(x,\xi) \in \Lambda$ is a point, the \textbf{microlocal stalk} $\cF|_{(x,\xi)}$ of $\cF \in \Sh_\Lambda(X)$ is by definition the stalk of $\cF|_\Lambda$
at $(x, \xi)$. 
For a point {\em at which $\Lambda \to X$ is an immersion} it can be computed directly as follows.  Pick a function $f$ defined in a neighborhood of $x$
so that $\xi = df(x)$, as well as a small ball $U$ around $x$ and $\epsilon >0$.  Then $\cF|_{(x,\xi)}$ is the cone over the restriction map from $\Gamma(U \cap \{f<\epsilon\}; \cF)$ to $\Gamma(U \cap \{f<-\epsilon\}; \cF)$.  This does not depend on the precise choice of a sufficiently small $U$ and $\epsilon$.

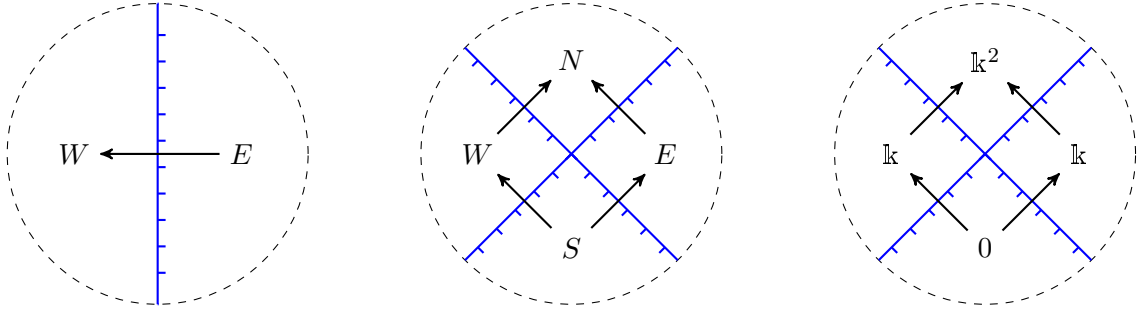
\begin{figure}
  \centering
  
\begin{tikzpicture}
\newcommand*{\rad}{2};
\pgfmathsetmacro{\vrad}{\rad/1.6};
\newcommand*{\acol}{black}; 
\newcommand*{\scol}{blue}; 
\newcommand{\nhrs}{18}; 
\pgfmathsetmacro{\nhrsminusone}{\nhrs-1}; 
\newcommand*{\hhrln}{.13}; 
\newcommand*{\dhrln}{.1}; 
\newcommand*{\mdist}{5.5}; 

\node (a) [matrix] at (-\mdist,0) {
\draw[dashed] (0,0) circle (\rad);
\draw[asdstyle,righthairs] (0,0) to (90:\rad);
\draw[asdstyle,lefthairs] (0,0) to (-90:\rad);
\node (W) at (180:\rad/1.8) {$W$}; \node (E) at (0:\rad/1.8) {$E$};
\draw[genmapstyle] (E) to (W);\\};

\node (l) [matrix] at (0,0) {
\draw[dashed] (0,0) circle (\rad);
\foreach \ang in {-45,45} {\draw[asdstyle,righthairsnogap] (0,0) to (\ang:\rad);}
\foreach \ang in {135,225} {\draw[asdstyle,lefthairsnogap] (0,0) to (\ang:\rad);}
\node (S) at (-90:\vrad) {$S$}; \node (W) at (180:\vrad) {$W$};
\node (E) at (0:\vrad) {$E$}; \node (N) at (90:\vrad) {$N$};
\foreach \c/\d in {S/W,S/E,W/N,E/N} {\draw[genmapstyle] (\c) to (\d);}\\};

\node (r) [matrix] at (\mdist,0) {
\draw[dashed] (0,0) circle (\rad);
\foreach \ang in {-45,45} {\draw[asdstyle,righthairsnogap] (0,0) to (\ang:\rad);}
\foreach \ang in {135,225} {\draw[asdstyle,lefthairsnogap] (0,0) to (\ang:\rad);}
\node (S) at (-90:\vrad) {$0$}; \node (W) at (180:\vrad) {$\coeffs$};
\node (E) at (0:\vrad) {$\coeffs$}; \node (N) at (90:\vrad) {$\coeffs^2$};
\foreach \c/\d in {S/W,S/E,W/N,E/N} {\draw[genmapstyle] (\c) to (\d);}\\};
\end{tikzpicture}
  \caption{The local models of the sheaf categories we consider.  On the left is an open disk where $\Lambda \to D^2$ is a single embedded strand, and $\Sh_\Lambda(D^2) \cong \coeffs A_2\dmod$ as described in \Cref{ex:A2}.  In the middle $\Lambda \to D^2$ is two embedded strands crossing, and $\Sh_\Lambda(D^2) \cong \coeffs A_3\dmod$ as described in \Cref{ex:crossing}.  The rightmost picture illustrates a microlocal rank-one sheaf in this case.}
  \label{fig:crossing}
\end{figure}

Our Legendrians will generally be smooth 1-dimensional submanifolds of the cosphere bundle of a surface, 
which we denote by $\Sigma$ rather than $X$.  Except in Section \ref{sec:rainspir}, 
it will also be true that the projection $\Lambda \to \Sigma$ is an immersion, and moreover generic (i.e. without
triple points).  Thus the
sheaves we need to work with are locally of one of the following forms. 

\begin{example}\label{ex:A2}
Let $D^2$ be the open unit disk in $\R^2$ and $\Lambda = dx|_{\{x=0\}}$ the Legendrian whose front projection is the $y$-axis, cooriented to the right.  Then $\Sh_\Lambda(D^2)$ is equivalent to $\coeffs A_2\dmod$, the (dg derived) category of (perfect) representations of the $A_2$ quiver, as follows.  We write $W$ and $E$ for any stalks in the open left half-disk $\{x< 0\}$ and closed right half-disk $\{x \geq 0\}$, respectively (all stalks in either region are canonically isomorphic up to homotopy).  There is a generization map $E \to W$ given by restricting from a neighborhood of a point on the $y$-axis to a smaller open set lying entirely to the left of the $y$-axis (note the non-isomorphic restriction maps go ``against the grain'' of the covector in general).  The microlocal stalk at a point of $\Lambda$ is the cone over this map.  An example of a sheaf of microlocal rank one is $i_!\coeffs_{\{x<0\}}$, the extension by zero of the constant sheaf on the open left half-disk, which corresponds to $W = \coeffs$, $E = 0$.  
\end{example}  

\begin{example}\label{ex:crossing}
Let $D^2$ be the open unit disk in $\R^2$ and $\Lambda = (dx - dy)|_{\{x=y\}} \cup (-dx-dy)|_{\{x=-y\}}$ the Legendrian whose front projection is the union of the lines  $x=y$ and $x=-y$, co-oriented downwards (see \Cref{fig:crossing}).  Then $\Sh_\Lambda(D^2)$ can be described in terms of the dg category of quadruples $N$, $W$, $E$, $S$ of perfect complexes of $\coeffs$-modules, with a commuting square of maps as pictured.  Such data gives rise to an object of $\Sh_\Lambda(D^2)$ 
under the following \textbf{crossing condition}: the total complex $S\to W\oplus E\to N$ must be acyclic \cite[Theorem 3.12]{STZ}.  

The restrictions $\Sh_\Lambda(D^2) \to \Sh_\Lambda(D^2 \cap \{y>\epsilon\})$, $\Sh_\Lambda(D^2) \to \Sh_\Lambda(D^2 \cap \{y<-\epsilon\})$ to the regions above and below the $x$-axis are equivalences, the codomains of which can be identified with $Rep(\bullet \to \bullet \leftarrow \bullet)$
and $Rep(\bullet \leftarrow \bullet \to \bullet)$ by forgetting $S$, $N$, respectively.  The induced equivalence $\Sh_\Lambda(D^2 \cap \{y>\epsilon\}) \cong \Sh_\Lambda(D^2 \cap \{y<-\epsilon\})$ is a reflection functor.

An example of a sheaf of microlocal rank one is the direct sum $i_!\coeffs_{\{x+y>0\}} \oplus i_!\coeffs_{\{y-x>0\}}$, which has $S = 0$, $W = E = \coeffs$, and $N = \coeffs^2$.  The crossing condition here says that $N$ is the direct sum of the images of $W$ and $E$.  If $\sigma$ is any point in the bottom quadrant, this sheaf is in $\Sh_\Lambda(D^2,\sigma)$.  This is the only  object of $\cC_1(\Lambda,\sigma;\coeffs)$ if $\coeffs$ has no nontrivial invertible modules.
\end{example}

\subsection{Invariance under contact transformations}\label{sec:contactinv}

Invariance of the category $\cC_1(\Lambda,\sigma;\coeffs)$
under Legendrian isotopy follows from the main theorem of \cite{GKS}: 

\begin{theorem}
\cite{GKS} Let $M$ be a manifold and $\phi_t: T^\infty M \to T^\infty M$ be a 
contact isotopy.  Then there is a unique sheaf $\Phi$ on $M \times M \times [0, 1]$ 
which restricts to the constant sheaf on the diagonal at $M \times M \times \{0\}$ 
and whose microsupport at infinity is the graph of the contact isotopy. 
\end{theorem}

We will mostly use the following consequence.  
\begin{proposition} \label{prop:GKS}
A Legendrian isotopy $\Lambda \to \Lambda'$ supported in the complement of the conormal
to $\sigma$ induces an equivalence of categories $Sh_\Lambda(M, \sigma) \cong Sh_{\Lambda'}(M, \sigma)$. 
The equivalence induced by a composition of isotopies is the composition of the equivalences
induced by the isotopies. 
\end{proposition}
\begin{proof}
Recall that given an isotopy of smooth Legendrians, there is an ambient contact
isotopy which induces it.  One can chose such a contact isotopy to be supported in a neighborhood of the Legendrian
isotopy, and with this requirement said contact isotopy is unique up to homotopy.  Let $\Phi_t$ be the sheaf quantization of one such
isotopy.  

The theory of sheaf integral transforms developed in \cite{KS} implies that using
$\Phi_1$ as an integral kernel gives an equivalence of categories 
$Sh_\Lambda(M) \cong Sh_{\phi_1(\Lambda)}(M)$. 
If $U$ is a neighborhood of $\sigma$ the support assumption on the Legendrian isotopy implies that $\Phi$ restricts to the constant sheaf on $U \times U \times [0,1]$, hence in particular the transform given by $\Phi_1$ preserves the stalk at $\sigma$. 

To see that this functor does not depend on the choice of contact isotopy inducing the Legendrian isotopy,
consider two and connect them by an homotopy of contact isotopies.  The sheaf quantization provides
a kernel $\Phi_{t, s}$ on $M \times M \times I \times I$.  This is {\em not} constant in the final $I$ direction; 
however when applied to elements in $Sh_\Lambda(M)$ it produces sheaves 
on $M \times I \times I$ whose microsupport is in the movie of $\phi_t(\Lambda)$ times a trivial factor
in the final direction.  It follows that these sheaves are constant in the final direction, and hence that
the functors induced by $\Phi_{t, 0}$ and $\Phi_{t,1}$ are the same.  

Having shown this, the uniqueness of $\Phi_t$ for a given contact isotopy implies the functoriality of this correspondence
asserted above. 
\end{proof}

\subsubsection*{Reidemeister Moves}

Since the GKS equivalence is compatible with composition of isotopies, 
in order to compute the equivalence $\cC_1(\Lambda,\sigma;\coeffs)\cong \cC_1(\Lambda',\sigma;\coeffs)$
associated to general Legendrian isotopy, it is enough to determine the equivalences associated to Legendrian Reidemeister moves.  The ones relevant to our immediate purposes are pictured in \Cref{fig:rmove,fig:reid}.   Like all isotopy equivalences, these are determined by the kernels constructed in \cite{GKS}. 
However, in these simple cases the equivalences are determined 
by the property that they restrict to the identity on the boundary of the picture, and can be described explicitly in terms of quiver representations. 

\begin{proposition}\label{prop:reid}\cite{STZ}
Let $\Lambda$, $\Lambda'$ be a pair of Legendrians in $T^\infty D^2$ differing by a Legendrian Reidemeister move, as in \Cref{fig:rmove,fig:reid} or \cite[Sec. 4.4]{STZ}.  There is a unique equivalence $\Sh_\Lambda(D^2) \cong \Sh_{\Lambda'}(D^2)$ that restricts to the identity of the boundary of the disk.
\end{proposition}
\begin{proof}
In all cases, the restrictions to the boundary of $\Sh_\Lambda(D^2)$ and $\Sh_{\Lambda'}(D^2)$ are fully faithful with the same essential image.  This follows from the fact that restriction from sheaves on a neighborhood of the crossing pictured in \Cref{fig:crossing} to the top and bottom regions is an equivalence; see \Cref{ex:crossing}.
\end{proof}

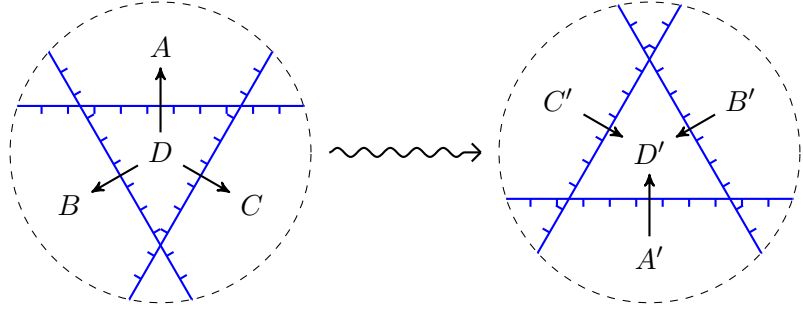
\begin{figure}
  \centering
\begin{tikzpicture}
\newcommand*{\off}{12}; \newcommand*{\rad}{2};
\node (tl) [matrix] at (0,0) {
\draw[dashed] (0,0) circle (\rad);
\coordinate (l1a) at (150+\off:\rad); \coordinate (l1b) at (30-\off:\rad);
\coordinate (l2a) at (150-\off:\rad); \coordinate (l2b) at (-90+\off:\rad);
\coordinate (l3a) at (-90-\off:\rad); \coordinate (l3b) at (30+\off:\rad);
\coordinate (l1m) at ($(l1a)!.5!(l1b)$); \coordinate (l2m) at ($(l2a)!.5!(l2b)$);
\coordinate (l3m) at ($(l3a)!.5!(l3b)$);
\foreach \c/\d in {l1m/l1a,l2m/l2b,l3m/l3b} {\draw[asdstyle,lefthairs] (\c) to (\d);}
\foreach \c/\d in {l1m/l1b,l2m/l2a,l3m/l3a} {\draw[asdstyle,righthairs] (\c) to (\d);}
\foreach \n/\ang in {A/90,B/210,C/330} {\node (\n) at (\ang:.7*\rad) {$\n$};}
\node (D) at (0,0) {$D$};
\foreach \n in {A,B,C} {\draw[genmapstyle] (D) to (\n);}\\};

\node (tr) [matrix] at (6.5,0) {
\draw[dashed] (0,0) circle (\rad);
\coordinate (l1a) at (210-\off:\rad); \coordinate (l1b) at (-30+\off:\rad);
\coordinate (l2a) at (90+\off:\rad); \coordinate (l2b) at (-30-\off:\rad);
\coordinate (l3a) at (210+\off:\rad); \coordinate (l3b) at (90-\off:\rad);
\coordinate (l1m) at ($(l1a)!.5!(l1b)$); \coordinate (l2m) at ($(l2a)!.5!(l2b)$);
\coordinate (l3m) at ($(l3a)!.5!(l3b)$);
\foreach \c/\d in {l1m/l1a,l2m/l2b,l3m/l3b} {\draw[asdstyle,lefthairs] (\c) to (\d);}
\foreach \c/\d in {l1m/l1b,l2m/l2a,l3m/l3a} {\draw[asdstyle,righthairs] (\c) to (\d);}
\foreach \n/\ang in {A/-90,B/30,C/150} {\node (\n) at (\ang:.7*\rad) {$\n'$};}
\node (D) at (0,0) {$D'$};
\foreach \n in {A,B,C} {\draw[genmapstyle] (\n) to (D);}\\};
\coordinate (rw) at ($(tr.west)+(-1mm,0)$); \coordinate (le) at ($(tl.east)+(1mm,0)$);
\draw[arrowstyle] (le) -- (rw);
\draw[arrhdstyle] ($(rw)+(-1.2mm,1.2mm)$) -- (rw); \draw[arrhdstyle] ($(rw)+(-1.2mm,-1.2mm)$) -- (rw);
\end{tikzpicture}

\caption{The Legendrian Reidemeister-III considered in \Cref{lem:R3}. Sheaves microsupported on the respective Legendrians are equivalent to representations of oppositely-oriented $D_4$-quivers. Writing $s_A$, etc. for reflection functors, the equivalence between the two sides is given by the composition $s_D s_A s_B s_C s_D$.}\label{fig:rmove}
\end{figure}

We will be interested in having explicit descriptions of the monodromy of the microlocal stalks along our
Legendrians, in order to show that certain Legendrian isotopies give rise to cluster transformations in \Cref{thm:clustertrans}.  
To follow these through a Reidemeister-III move, we note the following: 

\begin{lemma}\label{lem:R3}
In the Legendrian Reidemeister-III, on a given component of the knot, all microstalks on the component
before the move are canonically identified; all microstalks after the move are canonically identified, 
and there is a canonical identification of these canonical identifications. 
%
\end{lemma}
\begin{proof}
Before and after, each component of the microsupport is contractible, hence carries only trivial local systems, hence
all microstalks are identified.  To identify the before and after stalks, consider the quantization of the isotopy.  
It produces a sheaf on a cylinder $D^2 \times I$.  Each component of the microsupport
of this sheaf is contractible.
\end{proof}

\begin{figure}
\begin{tikzpicture}
\newcommand*{\off}{25};\newcommand*{\rad}{2.0}; \newcommand*{\vrad}{1.0};
\newcommand*{\fac}{.5};
\node (r) [matrix] at (6.5,0) {
\coordinate (ur) at ($(-\off:\rad)!\fac!(-180+\off:\rad)$); \coordinate (lr) at ($(\off-180:\rad)!\fac!(-\off:\rad)$);
\coordinate (ul) at ($(\off:\rad)!\fac!(-180-\off:\rad)$); \coordinate (ll) at ($(-\off-180:\rad)!\fac!(\off:\rad)$);
\coordinate (mr) at ($(ur)!.5!(lr)$); \coordinate (ml) at ($(ul)!.5!(ll)$);
\draw[dashed] (0,0) circle (\rad);
\draw[asdstyle,lefthairs] (mr) to (ur) [out=0,in=180] to (\off:\rad);
\draw[asdstyle,righthairs] (mr) to (lr) to[in=0,out=180] (180-\off:\rad);
\draw[asdstyle,righthairs] (ml) to (ul) [out=0,in=180] to (-\off:\rad);
\draw[asdstyle,lefthairs] (ml) to (ll) to[in=0,out=180] (180+\off:\rad);
\node (A) at (90:.75*\rad) {$A$}; \node (B) at (-90:.75*\rad) {$B$};
\node (cm) at (0,0) {$C'$};
\foreach \n in {A,B} {\draw[genmapstyle] (cm) to (\n);}\\};

\node (l) [matrix] at (0,0) {
\draw[dashed] (0,0) circle (\rad);
\coordinate (br) at (-\off:\rad); \coordinate (bl) at (\off-180:\rad);
\coordinate (ur) at (\off:\rad); \coordinate (ul) at (-\off-180:\rad);
\coordinate (bm) at ($(br)!.5!(bl)$); \coordinate (um) at ($(ul)!.5!(ur)$);
\foreach \c/\d in {um/ur,bm/bl} {\draw[asdstyle,lefthairs] (\c) to (\d);}
\foreach \c/\d in {um/ul,bm/br} {\draw[asdstyle,righthairs] (\c) to (\d);}
\node (A) at (90:.75*\rad) {$A$}; \node (B) at (-90:.75*\rad) {$B$};
\node (cm) at (0,0) {$C$};
\foreach \n in {A,B} {\draw[genmapstyle] (\n) to (cm);}\\};
\coordinate (rw) at ($(r.west)+(-1mm,0)$); \coordinate (le) at ($(l.east)+(1mm,0)$);
\draw[arrowstyle] (le) -- (rw);
\draw[arrhdstyle] ($(rw)+(-1.2mm,1.2mm)$) -- (rw); \draw[arrhdstyle] ($(rw)+(-1.2mm,-1.2mm)$) -- (rw);
\end{tikzpicture}
\caption{Sheaves microsupported on the Legendrians on either side of Reidemeister-II are equivalent to representations of oppositely-oriented $A_3$-quivers. The equivalence between the two sides is given by the reflection functor $s_C$; in other words, $C' = Cone(A\oplus B \to C)[-1]$.}
\label{fig:reid}
\end{figure}
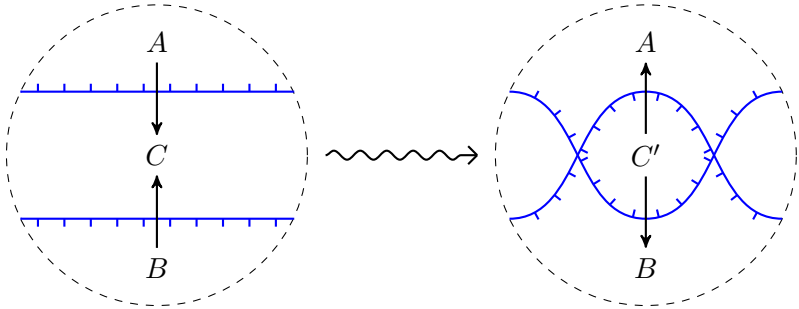

\subsection{Quantization of Lagrangians}

Let $L \subset T^*M$ be an eventually conical embedded exact Lagrangian with vanishing Maslov obstruction and with Legendrian boundary $\partial L \subset T^\infty M$. 
After a Hamiltonian perturbation we can make $L$ lower exact (i.e. have a proper, bounded above primitive) \cite{JinTreumann}, 
lift it to a Legendrian $\widetilde{L} \subset J^1 M$, and then embed $J^1 M \hookrightarrow T^\infty (M \times \R)$. 
Microlocalization gives a map $Sh_{\widetilde{L}}(M \times \R) \to \mu loc(\widetilde{L})$, and by assumption we can trivialize $\mu loc(\widetilde{L}) \cong Loc(L)$.
Restriction to $M = M \times \{r\}$ for $r \ll 0$ gives a map $Sh_{\widetilde{L}}(M \times \R) \to Sh_{\partial L}(M)$. 

Unlike those we usually work with, a general Legendrian in $J^1 M$ need not have locally trivial Maslov obstruction.
However except in this section we won't explicitly discuss this Legendrian lift, so no confusion should arise.  

\begin{theorem}\label{thm:sheafquantization}\cite{Gui, JinTreumann} 
The left functor in $Loc(L) \leftarrow Sh_{\widetilde{L}}(M \times \R, M \times \{\infty\}) \rightarrow Sh_{\partial L}(M)$ is an equivalence and the right is fully faithful. The composite thus gives a fully faithful embedding of the category of local systems on $L$ into the category $Sh_{\partial L}(M)$. \end{theorem}

We say the objects in the image of this functor are obtained from $L$ by sheaf quantization. Sheaf quantization transforms naturally under compactly supported Hamiltonian isotopy of $L$, since such becomes
a Legendrian isotopy of $\widetilde{L}$. In particular we have the following.

\begin{lemma} The image of $Loc(L)$ in $Sh_{\partial L}(M)$ is invariant under compactly supported Hamiltonian isotopy.
\end{lemma}

It is useful to observe the following additional property.

\begin{lemma}
Let $\cF$ be the sheaf on
$M$ corresponding to a rank one local system on $L$.  
If $m \in M$ is a point over which the projection $\pi: L \to M$ is locally an $n$-sheeted cover, then $\cF_m$ can be computed by a complex with 
$n$ generators.
\end{lemma}
\begin{proof}
Let $\widetilde{\cF}$ be the corresponding sheaf on $M \times \R$.  Consider the line $m \times \R$.  This line
meets the front projection of $\widetilde{L}$ transversely.  These intersections are in bijection with the intersection
of the cotangent fiber of $M$.  At an intersection point $m \times l$, one has a triangle 
$$\widetilde{F}_{m \times (l \pm \epsilon)} \to \widetilde{F}_{m \times (l \mp \epsilon)} \to \coeffs[d] \xrightarrow{[1]}$$
for some $d \in \Z$. 
We have $\widetilde{F}_{m \times -\infty} = 0$, so the result follows by induction. 
\end{proof}

In fact for us only the following consequence will directly appear, which doesn't require the induction step.

\begin{corollary}\label{cor:quantranks}
Let $\cF$ be the sheaf on
$M$ corresponding to a rank one local system on $L$.  
Let $\pi: L \to M$ be the projection.  
If $\pi^{-1}(m) = \emptyset$, then $\cF_m = 0$, and if $\pi^{-1}(m) = 1$ then $\cF_m \cong \coeffs[d]$ for some $d \in \Z$. 
\end{corollary}

\subsection{Moduli spaces}\label{sec:derived}
In studying microlocal rank-one objects of
$Sh_\Lambda(\Sigma)$, we necessarily consider objects which are honestly complexes
of sheaves, rather than simply sheaves.  The resulting subcategory is not abelian in general --- in particular, objects
may have negative self-extensions, even for Legendrian knots in the standard contact 
$\R^3$; examples can be found in \cite{STZ}.  The correct
setting for studying moduli spaces of objects in such dg categories 
is 
derived algebraic geometry \cite{TVa}.  For background we refer to the survey \cite{Toe2} and the foundational
works \cite{Lur, Toe, TV1, TV2, TV3}.  

\begin{definition/proposition} \label{def:moduli}
Let $X$ be the interior of a compact manifold with boundary, $\Lambda \subset T^\infty X$ a Legendrian contained in the spherically projectivized conormals of a Whitney stratification that extends to the boundary, and 
$\sigma$ a collection of points in $X$.  We write $\R\cM(X, \Lambda,\sigma)$ for 
the moduli of objects in 
$Sh_\Lambda(X,\sigma)$. It is a locally geometric derived stack. We write $\R \cM_n(\Sigma, \Lambda,\sigma)$ for the substack
parameterizing sheaves of microlocal rank $n$. 
\end{definition/proposition}
\begin{proof}
The existence of these spaces is guaranteed by 
\cite{TVa}, which constructs derived moduli stacks of pseudoperfect modules of finite-type dg categories (that is, of functors from finite-type dg categories to categories of perfect modules).
The finite-type category in question is that of wrapped constructible sheaves on $X$ microsupported on $\Lambda$; this is the full subcategory of compact objects in the cocomplete dg category of all sheaves microsupported on $\Lambda$ (i.e. with not-necessarily perfect stalks) \cite{N6}. Taking $\Hom$ spaces identifies the category $Sh_\Lambda(X,\sigma)$ of sheaves with perfect stalks as pseudoperfect modules of the wrapped category \cite{N6}. 
The assumptions on $X$ and $\Lambda$, together with the results of \cite{N5} adapted to the setting of exact Lagrangians, imply that the wrapped sheaf category in question is a finite colimit of finite-type dg categories (specifically, of categories of perfect representations of acyclic quivers). The claim now follows since a finite colimit of finite-type dg categories is again finite type.
\end{proof}

The higher and derived structures on the spaces $\R \cM_n(X, \Lambda, \sigma)$ are essential
from various points of view: for instance, to get meaningful point counts \cite{NRSS}
and to construct symplectic structures on these spaces \cite{ShT} as was done for
moduli of local systems in \cite{PTVV}. An important point is that the infinitesimal study of derived moduli spaces is generally more accessible than that of ordinary moduli spaces. For example, letting $\R Loc(L)$ denote the derived moduli space of local systems on $L$ we have following consequence of Theorem \ref{thm:sheafquantization}. 

\begin{proposition}\label{prop:open}
Let $L \subset T^* X$ be an embedded eventually conical exact Lagrangian with Legendrian boundary $\partial L = \Lambda$ and whose projection is disjoint from $\sigma \subset X$. 
Given a trivialization $\mu loc(\widetilde{L}) \cong Loc(L)$, sheaf quantization induces an open inclusion $\R Loc_n(L) \to \R \cM_n(X, \Lambda, \sigma)$. 
\end{proposition}
\begin{proof}
This follows formally from the fact that the map on moduli spaces is induced by a
fully faithful inclusion of dg categories.  Indeed, it follows from this 
that the morphism is injective on points, and since 
the tangent complexes to the moduli spaces are given by self-ext algebras \cite[Thm 0.2]{TVa}, it follows that the map is \'etale. 
\end{proof}


When $L$ is a Lagrangian \emph{surface} and $\Lambda$ nonempty, $L$ is homotopy equivalent to a wedge of circles. 
In this case the derived stack $\R Loc_n(L)$ is truncated: it is isomorphic to its truncation $t_0 \R Loc_n(L)$ -- which is simply the classical Artin stack $\Loc_n(L)$ of local systems -- regarded as a derived stack. 
Thus if we are only interested in branes supported on Lagrangians of this kind, and the relations among them, we lose no information by working at the level of Artin stacks in the classical sense. Following \cite[Sec. 3.4]{TVa} we have the following classical moduli spaces.

\begin{definition/proposition}
Let $\Sigma$ be a surface, $\Lambda \subset T^\infty \Sigma$ a nonempty Legendrian, $\sigma$ a collection of points in $\Sigma$. We write $\cM(\Sigma,\Lambda,\sigma)$ for the 1-rigid locus of $t_0 \R Loc(\Sigma, \Lambda,\sigma)$; that is, the locus parametrizing objects without negative self-extensions. It is an Artin stack in the classical sense. We write $\cM_n(\Sigma,\Lambda,\sigma)$ for the substack parametrizing sheaves of microlocal rank $n$.
\end{definition/proposition}

In many cases of interest the objects parameterized
by $\R \cM_n(\Sigma, \Lambda, \sigma)$ are ordinary sheaves (i.e. not complexes). For example, this holds for the $\beta$-filtered local systems studied in \Cref{sec:braids}, see \Cref{prop:abeliancat}.
Hence these objects live in an abelian category, and in particular have no negative self-extensions to begin with --- so $\cM_n(\Sigma, \Lambda, \sigma)$ is equal to $t_0 \R \cM_n(\Sigma, \Lambda, \sigma)$. 
Since the truncation of an \'etale map is \'etale \cite[Sec. 2.2.4]{TV3}, we also have the underived analogue of Proposition \ref{prop:open}.

\begin{proposition}
Let $L \subset T^* \Sigma$ be an exact Lagrangian surface with $\partial L = \Lambda$ and whose projection is disjoint from $\sigma \subset \Sigma$. 
Given a choice of brane structure on $L$, the functor $\cN$ induces an open inclusion $Loc_n(L) \to \cM_n(\Sigma, \Lambda, \sigma)$. 
\end{proposition}

\begin{remark} \label{rem:gohere}
Explicitly, $\cM_n(\Sigma, \Lambda, x)$ represents the functor from commutative rings to 
groupoids taking $\coeffs$ to $\cC_n(\Lambda,\sigma; \coeffs)^{gpd}$, the groupoid whose objects are sheaves in $\cC_n(\Lambda,\sigma;\coeffs)$ without negative self-extensions and whose morphisms are quasi-isomorphisms up to homotopy, with pullback defined by base change. 
\end{remark}

We also need to consider moduli spaces of framed sheaves, constructed as follows.  Let $U$ be an open subset of $S$ and $\cF_\Z$ an object of $\Sh_{\Lambda|_U}(U, \sigma \cap U; \Z)$.  The sheaf $\cF_\Z$ defines a map $\Spec \Z \to \cM_1(U,\Lambda|_U)$.

\begin{definition}\label{def:framed} 
The moduli space $\cM_{\:1}^{fr}(\Sigma,\Lambda,\sigma)$ of framed sheaves of microlocal rank one is the fiber product
\[
\begin{tikzpicture}
\node (tl) at (0,0) {$\cM_{\:1}^{fr}(\Sigma,\Lambda,\sigma)$}; \node (tr) at (4,0) {$\cM_1(\Sigma,\Lambda,\sigma)$};
\node (bl) at (0,-1) {$\Spec \Z$}; \node (br) at (4,-1) {$\cM_1(U,\Lambda|_U$),};
\draw[->] (tl) to (tr); \draw[->] (tr) to (br);
\draw[->] (tl) to (bl); \draw[->] (bl) to (br);
\end{tikzpicture}
\]
where the right-hand map is restriction to $U$ and the bottom is the inclusion of $\cF$.
\end{definition}

The $\coeffs$-points of $\cM_{\:1}^{fr}(\Sigma,\Lambda,\sigma)$ are thus objects of $\cC_1(\Lambda,\sigma;\coeffs)$ together with an isomorphism of their restriction to $U$ with $\cF_\coeffs$, the object of $\Sh_{\Lambda|_U}(U, \sigma \cap U; \coeffs)$ obtained from $\cF_\Z$ by base change.  In practice $\cF$ will always be chosen in some trivial way, hence we omit it from the notation for $\cM_{\:1}^{fr}(\Sigma,\Lambda,\sigma)$. Proposition \ref{prop:GKS} extends in the obvious way to the framed moduli spaces. 
\section{Microlocally abelian moduli problems}\label{sec:braids}

Here we focus our attention on a class of moduli problems $\cM_1(\Sigma, \Lambda, \{\sigma_i\})$
in which the link $\Lambda$ is a disjoint union of positive $n$-strand braids, one placed in
a neighborhood of a co-circle over each $\sigma_i$. 
As we detail in this section, these spaces include various 
ones of current interest, in particular:

\begin{itemize}
\item Positroid strata in the Grassmannian  \cite{Po}
occur when a single braid is placed on a sphere, for particular choices of braid.
See Sec. \ref{sec:positroids4life}. 
\item More generally, placing a single braid on a sphere gives rise to a
certain moduli of flag configurations, whose point-count gives a 
term of the HOMFLY polynomial of 
the braid \cite[Sec. 6]{STZ}. 
\item Moduli spaces of local systems with monodromy-invariant filtrations in the case when all the braids
are trivial (see Ex. \ref{ex:filtered}), and more generally any moduli space of monodromy and Stokes data, i.e., 
any wild character variety (see Sec. \ref{sec:wild}).
\end{itemize} 

Moreover, if we are ultimately interested in cluster structures related to the moduli space of 
rank $n$ local systems on a punctured surface, we are forced to consider spaces of the above kind. 
Most naturally, rank-$n$ local systems correspond to sheaves with microlocal
rank $n$ along a collection of circles around the punctures, and vanishing at the punctures. 
However, our expected
sources of cluster charts are {\em abelian} Lagrangian branes, which determine sheaves of microlocal rank one along their Legendrian boundary.  Thus
as a preliminary to abelianization of the rank $n$ local systems, we must perform a microlocal
abelianization of the boundary condition --- i.e., replacing the circle labelled by $n$ with 
an $n$-strand braid.  


\subsection{Microlocal abelianization}\label{sec:microab}

Consider a surface $\Sigma$ with a set $\sigma = \{\sigma_i\}$ of marked points.  Let $\Lambda_i$
be a small circle around $\sigma_i$, co-oriented inward, and $D_i$ the disk around
$\sigma_i$ whose boundary is the front projection of $\Lambda_i$.  Consider the inclusions
$$ \Sigma \smallsetminus \sigma \overset{r} \hookleftarrow \Sigma \smallsetminus \bigcup \overline{D}_i \overset{j}\hookrightarrow \Sigma$$ 
These induce an equivalence
$$j_! r^*: Loc(\Sigma \smallsetminus \sigma;\coeffs) \to Sh_{\bigcup \Lambda_i}(\Sigma,\sigma;\coeffs)$$
between the categories of local systems on the punctured curve and of sheaves on the complete curve with
microsupport in the circles and vanishing stalks at the points.  The equivalence carries
the rank of the local system to the rank of the microstalk on any $\Lambda_i$. 

It is the condition that the sheaves should have {\em rank-one microstalks} 
that gives rise to cluster structures.  The moduli space corresponding to 
local systems of rank $n$ does not have this property.  We get ones which do by 
replacing the circle labelled by $n$ with a suitable $n$-strand satellite.  

By definition, the satellite construction takes as input data a triple
$(V, \Lambda, \beta)$ where $V$ is a contact manifold,  $\Lambda$ is a 
Legendrian  and $\beta$ is a Legendrian in the 1-jet bundle
$J^1(\Lambda)$. The output is a new Legendrian $\beta \looparrowright \Lambda$
in the same contact manifold $V$, formed by 
replacing a standard neighborhood of $\Lambda$ by the $J^1(\Lambda)$ containing
 $\beta$.   The Legendrian $\beta$ is the {\em pattern} of the satellite construction. 
For some discussion and examples, see \cite{Ng-sat}. 

\begin{lemma}
Let $M$ be a manifold, $\Lambda \subset T^\infty M$ a Legendrian,
and $\beta \subset J^1(\Lambda)$ a Legendrian. 
Assume that $\beta \to \Lambda$ is a covering map.  
Then there is a natural morphism
$$\pi: Sh_{\beta \looparrowright \Lambda}(M) \to Sh_{\Lambda}(M)$$ 
such that $$rank_{\Lambda}(\pi(\cF)) = \deg(\beta \to \Lambda) \cdot 
rank_{\beta \looparrowright \Lambda}(\cF)$$
\end{lemma}
We omit the (easy) proof, 
as we only use this proposition in the case when $M$ is a surface 
and the Legendrian $\Lambda$ is a union of circles around the punctures. In this case the 
result is evident; we just include the above formulation for the sake of clarity.  

We can associate a Legendrian in $J^1(\Lambda)$ to any positive (annular) braid. 
Thus a choice of a positive $n$-strand braid $\beta_i$ at the $i$'th puncture determines
a morphism
\begin{equation}\label{eq:forget} \cM_1(\cup_i (\beta_i \looparrowright \Lambda_i), \sigma) \to \cM_n(\cup_i \Lambda_i, \sigma) \cong Loc_n(\Sigma \smallsetminus \sigma).
\end{equation}
That is, we draw $n$-stranded braids around the points and consider sheaves microsupported
along these braids. 

\begin{definition}\label{def:filtloc}
Let $\Sigma$ be a surface and $\sigma = \{ \sigma_i \}$ a collection of points.  Let 
$\sigma_i \mapsto \beta_i \in Br^+_n$ be an assignment of a positive braid to each
point of $\sigma$; by abuse of notation we also write
$\beta_i$ for $\beta_i \looparrowright \Lambda_i$ where $\Lambda_i$ is an inward-co-oriented
circle around $\sigma_i$.   Writing $\beta = \coprod \beta_i$, we refer to the points of $\cM_1(\Sigma,\beta,\sigma)$ as \textbf{$\beta$-filtered local systems}. 
\end{definition}

For the trivial braid, this recovers exactly the notion of filtered local system:  

\begin{example} \label{ex:filtered}
Let $D^2$ be a disk, and $\bigcirc^n \subset T^\infty (D^2 \smallsetminus 0)$ be the link whose front projection is $n$
concentric circles. Then $Sh_{\bigcirc^n}(D^2, 0)$ is the category of pairs
$(0 = K_0 \to K_1 \to \cdots \to K_n = K;\, m: K \to K)$ where $K$ is a filtered complex
and $\phi$ is an endomorphism preserving the filtration.  The correspondence is that
$K_i$ is the stalk between the $i$th and $i+1$st strands away from $0$, and 
$m$ is the monodromy. 

Fixing the microlocal rank to equal one forces $K$ to be (quasi-isomorphic to) a locally
free $\coeffs$-module of rank $n$, and the filtration to be the same as a full flag. 
Thus, $\cM_1(\bigcirc^n, 0)$ is
the total space of the Grothendieck-Springer resolution:  $\cM_1(\bigcirc^n, 0) \cong \widetilde{GL}_n/GL_n$.  The resolution morphism $\widetilde{GL}_n/GL_n \to GL_n/GL_n$ itself is the map $\cM_1(\bigcirc^n, 0) \to Loc_n(D^2 \smallsetminus 0)$ of \Cref{eq:forget}.
\end{example}

\Cref{ex:filtered} illustrates a general feature of $\beta$-filtered local systems: up to quasi-isomorphism they are sheaves in the usual sense rather than merely complexes of sheaves.  

\begin{proposition}\label{prop:abeliancat}
Let $\Sigma$, $\beta$, and $\sigma$ be as in \Cref{def:filtloc}. Every microlocal rank-one object of $\Sh_\beta(\Sigma,\sigma;\coeffs)$ is isomorphic to an object supported in cohomological degree zero. In particular, no such objects have negative self-extensions, so $\cM_n(\Sigma,\beta,\sigma)$ is exactly the truncation $t_0 \R \cM_n(\Sigma,\beta,\sigma)$.
\end{proposition} 

We omit the proof, which is a straightforward generalization of \cite[Prop. 5.19]{STZ}.

\subsection{Positroid strata and the Grassmannian}\label{sec:positroids4life}
The positroid stratification of the Grassmannian is the common refinement of the Schubert stratification and its cyclic shifts, and arises naturally from the perspective of total positivity  \cite{Po}.  
The positroid strata of $\Gr(k,n)$ are indexed by a number of equivalent combinatorial objects, 
the most relevant of which for us will be cyclic rank matrices \cite[Cor. 3.12]{KLS}: in this section we use these to give microlocal descriptions of positroid strata.

\begin{definition}
A \textbf{cyclic rank matrix} of type $(k,n)$ is a $\mathbb{Z} \times \mathbb{Z}$ integer matrix $r$ such that
\begin{description}
\item[C1] $r_{ij} = 0$ for $j <i$
\item[C2] $r_{ij} = k$ for $j \geq i+n-1$
\item[C3] $r_{ij}-r_{(i+1)j} \in \{0,1\}$ and $r_{ij}-r_{i(j-1)} \in \{0,1\}$ for all $i$, $j$
\item[C4] if $r_{(i+1)(j-1)} = r_{(i+1)j} = r_{i(j-1)}$ then $r_{ij} = r_{(i+1)(j-1)}$
\item[C5] $r_{(i+n)(j+n)} = r_{ij}$
`\end{description}
\end{definition}

For each $V \in \Gr(k,n)$ there is an associated cyclic rank matrix $r(V)$ of type $(k,n)$, and the positroid strata will be the level sets of this assignment.  Let $c_1,\dotsc,c_n \in \C^k$ be the columns of any matrix representative of $V$, and for arbitrary $i \in \Z$ define $c_i$ so that $c_i = c_{i+n}$ for all $i$.  Then we set $r(V)_{ij}$ to be the dimension of the span of the columns $c_i \cdots c_j$.  Note that for $j < i$, we have the empty
collection of columns, hence $r(V)_{ij} = 0$, and that for $j > i + n - 1$ we have all the columns,
hence $r(V)_{ij} = k$.  
The conditions \textbf{C1}-\textbf{C5} exactly characterize the matrices that arise from $\Gr(k,n)$ in this fashion \cite{KLS}.  

\begin{definition}
Given a cyclic rank matrix $r$ of type $(k,n)$, the associated \textbf{positroid stratum} is
\[
\Pi_r = \{V \in \Gr(k,n)|r(V) = r\}.
\] 
\end{definition}

In our context the most natural cyclic rank matrices are those such that $r_{ii} \neq 0$ for all $i$, and we assume this from now on.  That is, we assume the columns of any matrix representative of $V \in \Pi_r$ are all nonzero.  No generality is lost in the sense that any positroid stratum $\Pi_r$ not satisfying this condition be embedded into a smaller Grassmannian as a positroid stratum that does.

We record loci where the entries of $r$ jump as a Legendrian $\Lambda_r$ in $T^\infty D^2$ as follows.  The basic idea is to regard $r$ as an actual geometric object in $\R^2$, and build $\La_r$ in such a way that the faces of its front projection correspond to patches of $r$ where its entries are constant. 
\begin{itemize}
\item We first define a Legendrian $\Lambda_r'$ in $T^\infty\R^2$ lying over a neighborhood of the square grid $\R \times \Z \cup \Z \times \R$.  Consider the union of the segments $\{i\} \times (j, j+1)$ with $r_{i(j-1)} < r_{ij}$ and the segments $(i,i+1) \times \{j\}$ with $r_{ij} > r_{(i+1)j}$.  We co-orient the former leftward and the latter downward.  Its closure is a collection of pairwise-transverse immersed co-oriented curves with corners at the points $(i,j) \in \Z^2$ such that
\[
r_{ij} = r_{(i+1)j} = r_{i(j-1)} = r_{(i+1)(j-1)} + 1.
\]
We smooth all such corners and let $\Lambda_r'$ be the Legendrian lift of the resulting collection of smooth immersed curves.
\item Consider the restriction of $\Lambda_r'$ to the infinite strip 
\[
\cS=\{(x,y)| \frac12 < y+x < k + \frac12\}.
\]
By \textbf{C5}, the restriction is invariant under the translation $T_n: (x,y) \mapsto (x+n,y-n)$ of $\cS$, hence gives rise to a Legendrian $\Lambda_r$ in $T^\infty\cA$, where $\cA$ is the annulus $\cS/\langle T_n\rangle$.  Since $\Lambda_r$ does not meet the boundary component whose preimage is the line $y+x = k + \frac12$, we can embed $\cA$ in a disk to regard $\Lambda_r$ as a Legendrian in $T^\infty D^2$. 
\end{itemize}

\begin{example}\label{ex:introGr24}
Let us compute the moduli space of microlocal rank-one sheaves associated to the Legendrian 
whose front projection is pictured on the left in Figure \ref{fig:introGr24}.  
This Legendrian is the satellite 
formed by taking the braid on two strands, twisted five times, and inserting it in
a neighborhood of the cocircle over the North pole of $S^2$; the front
projection of this lands in the complement of the North pole, which we identify with the page.

 A sheaf with microsupport in this Legendrian restricts to a local 
system on each component of the complement of the front projection.  
We are considering sheaves vanishing
at the north pole, so these local systems are only nonzero on the six components which are
bounded in the picture.  Since these are contractible, the local systems are just the data of six vector spaces (for now we work over a ground field $\coeffs$). 
The sheaf is then determined by the data of these vector spaces together with linear maps associated to paths going against the hairs.  If the sheaf has microlocal rank-one, the dimensions increase by one as we move against the direction of the hairs, hence are as indicated in the picture.  Choosing bases for the vector spaces involved, we can can encode the linear maps as the columns of a matrix:

\begin{equation}
\label{eq:enws-mat}
\left[
\begin{array}{rrrrr}
a_1 & b_1 & c_1 & d_1 & e_1 \\
a_2 & b_2 & c_2 & d_2 & e_2
\end{array}
\right]
\end{equation}
The condition that the pictured sheaf has singular support on $\Lambda$ --- as opposed to the 
union of $\Lambda$ with the cocircle fibers over the crossings in the front projection --- translates to the condition that 
any two cyclically adjacent columns of \eqref{eq:enws-mat} are linearly independent (see \cite[Sec. 5]{STZ}). 

Two such matrices correspond to isomorphic sheaves if they are related by a combination of left multiplication by $\GL_2$ and right multiplication by the diagonal subgroup of $\GL_5$.  Thus the moduli space $\cM_1(\Lambda)$ is the quotient of the space of $2 \times 5$ matrices satisfying the crossing conditions by these symmetries.  In other words, $\cM_1(\Lambda)$ is the configuration space of $5$ cyclically ordered points in $\bP^1$, with the condition that cyclic neighbors are distinct.

We can also frame this story by considering sheaves with fixed trivializations of their one-dimensional stalks.  These are still represented by matrices as in \eqref{eq:enws-mat}, but now two matrices are equivalent if and only if they are related by the left $\GL_2$ action.  The moduli space of so-framed microlocal rank-one sheaves is an open subset of the Grassmannian of two-planes in five-space.  The crossing conditions above  define the \emph{big positroid stratum} of $\Gr(2,5)$ \cite{Po}. 
\end{example}

\begin{example}\label{ex:Gr35}
Let $V$ be the point of $\Gr(3,5)$ represented by the matrix
\[
\begin{pmatrix} 0 & 1 & 1 & 1 & 1 \\ 0 & 0 & 1 & 1 & 1 \\ 1 & 0 & 0 & 0 & 0\end{pmatrix}.
\] 
The cyclic rank matrix $r(V)$ is shown below left, with the front projection of $\Lambda'_{r(V)}$ shown in red (without smoothed corners). The identification of top and bottom sides by the action of $T_5$ is indicated by the dashed green line.
At right is the front of the associated (smooth) Legendrian knot on the cylinder, the horizontal dashed line indicating where $\partial D^2$ cuts $\Lambda'_{r(V)}$.
\end{example}

\begin{center}
\includegraphics[scale=0.4]{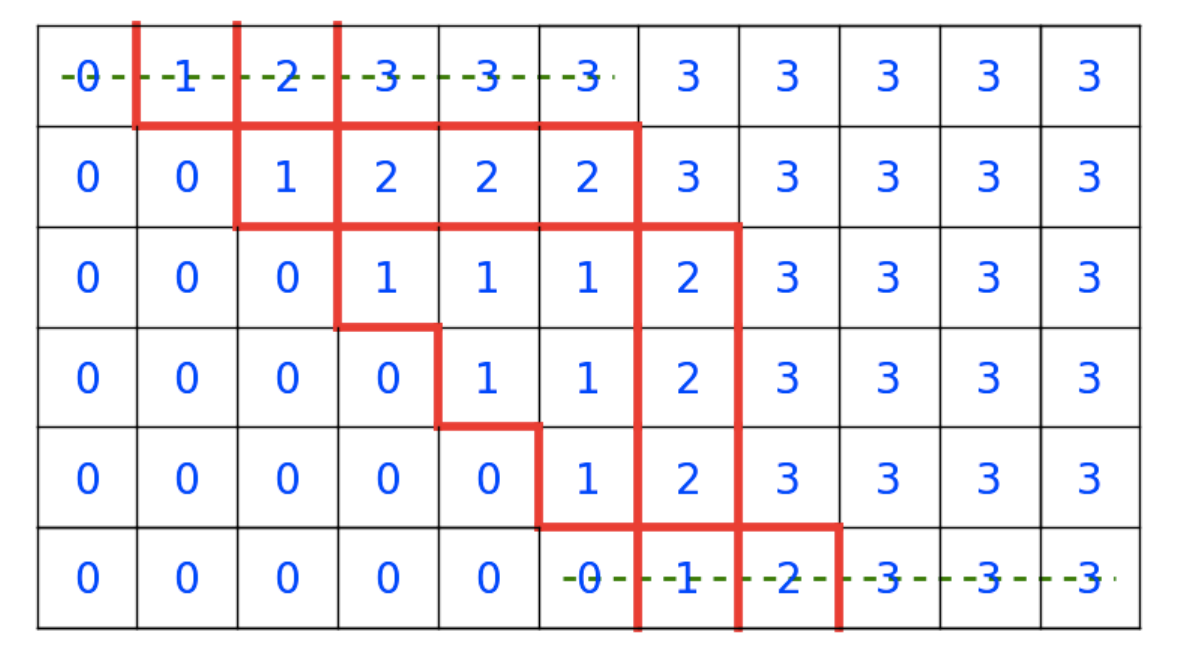}
\qquad 
\raisebox{0.01\height}{\includegraphics[height=1.7in]{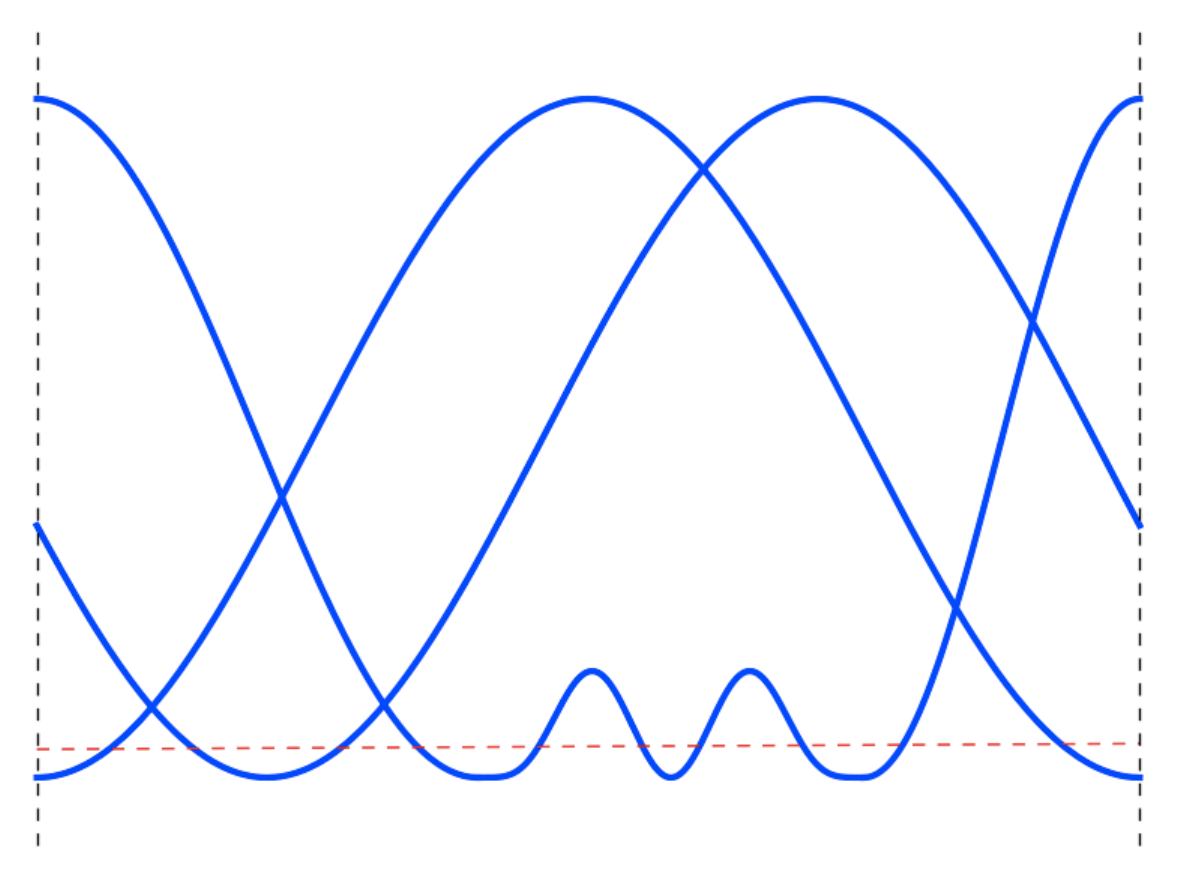}}
\end{center}

Returning to the general case, we fix a set of points $\sigma \subset D^2$, one in each of the $n$ boundary regions of $D^2 \smallsetminus \pi(\Lambda_r)$ for which $\Lambda_r$ is co-oriented into the given region.  In the construction of $\Lambda_r$, these regions come from the ``corners'' of the subdiagonal entries of $r$, which are equal to zero.  The pair $\Lambda_r$, $\sigma$ effectively satisfy the hypotheses of \Cref{prop:abeliancat}, so the objects of $\cC_1(\Lambda,\sigma)$ are ordinary sheaves rather than complexes.  

Let $U$ be an open collar of $\partial D^2$ containing no crossings of $\Lambda_r$, and $V$ the union of the components of $D^2 \smallsetminus \pi(\Lambda_r)$ that do not contain points in $\sigma$.  We let $\cF_\coeffs = i_*\coeffs_{U \cap V} \in \Sh_{\Lambda_r|_U}(U)$; this is a sheaf of microlocal rank one whose stalks are alternately $0$ and $\coeffs$ around the boundary of $D^2$.  We let $\cM_{\:1}^{fr}(\Lambda_r)$ denote the space of sheaves framed by $\cF_\Z$, as in \Cref{def:framed} (since $\sigma$ is fixed throughout the section, we omit it from the notation).  The $\coeffs$-points of $\cM_{\:1}^{fr}(\Lambda_r)$ are objects in $\cC_1(\Lambda,\sigma;\coeffs)$ equipped with an isomorphism between their restriction to $U$ and $\cF_\coeffs$. 

\begin{theorem}\label{thm:positroidasmoduli}
For any cyclic rank matrix $r$ of type $(k,n)$, there is a canonical isomorphism between the positroid stratum $\Pi_r$ and the framed moduli space $\cM^{fr}_{\:1}(\Lambda_r)$.
\end{theorem}
\begin{proof}
We first describe an embedding of $\cM^{fr}_{\:1}(\Lambda_r) \hookrightarrow \Gr(k,n)$, then show its image is exactly $\Pi_r$.  

The definition of $\Lambda_r$ (and choice of $\sigma$) is set up to ensure that the connected components of $D^2 \smallsetminus \pi(\Lambda_r)$ correspond to entries of $r$, in such a way that the rank of the stalk of any object of $\cC_1(\Lambda_r,\sigma;\coeffs)$ in a given 
connected component is the corresponding entry of $r$. 

Each of the $n$ rows of $r$ determines a boundary component of  $D^2 \smallsetminus \pi(\Lambda_r)$ which does not meet $\sigma$;
fix points $p_1, \ldots, p_n$ in these components.  Also fix a point $x$ in the middle region where stalks have rank $k$.    

A $\coeffs$-point of $\cM^{fr}_{\:1}(\Lambda_r)$ is a framed sheaf $F$, which has stalks $F_1,\dotsc, F_n$ at $p_1,\dotsc,p_n$, each equipped with a trivialization $F_i \cong \coeffs$, and a stalk $F_x$ at $x$. 
Choose characteristic paths from each $p_i$ to $x$ (that is, paths that only cross strands of $\pi(\Lambda_r)$ going against their co-orientations; the construction will be independent of the choice).  Each defines an inclusion $F_i \into F_x$, the composition of the generization maps along the path (see Figure \ref{fig:Gr35sheaf}).  The crossing conditions (see \Cref{ex:crossing}) guarantee that the images of the $F_i$ together generate $F_x$.  Thus from $F$ we obtain a locally free $\coeffs$-module $F_x$ of rank $k$ with a quotient map from $\oplus F_i \cong \coeffs^n$, which is the data of a $\coeffs$-point of $\Gr(k,n)$.  This is clearly compatible with base change, so we obtain a map $\cM^{fr}_{\:1}(\Lambda_r) \to \Gr(k,n)$.  The claim that this is faithful is equivalent to the claim that restriction from $\cM^{fr}_{\:1}(\Lambda_r)$ to the framed moduli space of a neighborhood of the union of the paths from the $p_i$ to $x$ is faithful.  This follows by inductively applying the fact that in a neighborhood of a crossing, restriction to the upper or lower regions as pictured in \Cref{fig:crossing} is an equivalence.  That is, we can expand the neighborhood of the paths to include a new crossing and meet a new region of $D^2 \smallsetminus \pi(\Lambda_r)$ one at a time, each time yielding an equivalence of sheaf categories, until the whole disk is covered.

\begin{figure}\label{fig:Gr35sheaf}
\centering
\begin{tikzpicture}
\newcommand*{\off}{10}; \newcommand*{\rad}{2.7}; \newcommand*{\ang}{70}; \newcommand*{\fac}{.3};
\node (l) [matrix] at (0,0) {
\draw[dashed] (0,0) circle (\rad);
\coordinate (a) at (90:\rad*.5); \coordinate (b) at (-140:\rad*.53);
\coordinate (c) at (-20:\rad*.35); \coordinate (d) at (80:\rad*.7);
\coordinate (e) at (-40:\rad*.53); \coordinate (f) at (-160:\rad*.35);
\coordinate (g) at (100:\rad*.7);
\draw[asdstyle,righthairs] (-90+\off:\rad) to[out=180-90+\off,in=180-18-\off] (-18-\off:\rad);
\draw[asdstyle,lefthairs] (-90-\off:\rad) to[in=180-162+\off,out=180-90-\off] (-162+\off:\rad);
\draw[asdstyle,righthairs] (a) to[out=180,in=180-162-\off] (-162-\off:\rad);
\draw[asdstyle,lefthairs] (a) to[in=180-18+\off,out=0] (-18+\off:\rad);
\draw[asdstyle,righthairs] (126+\off:\rad) to[out=180+126+\off,in=145] (b) to[out=180+145,in=180+50] (c) to[out=50,in=-20] (d) to[out=180-20,in=180+126-\off] (126-\off:\rad);
\draw[asdstyle,lefthairs] (54-\off:\rad) to[out=180+54-\off,in=35] (e) to[out=180+35,in=180+130] (f) to[out=130,in=-160] (g) to[out=180-160,in=180+54+\off] (54+\off:\rad);
\node (m) at (90:.1*\rad) {$F_x$};
\foreach \n in {1,2,3,4,5} {\node (n\n) at (54+72*\n:\rad*.85) {$F_\n$}; \draw[genmapstyle] (n\n) to (m);}\\};
\end{tikzpicture}
\caption{A depiction of a microlocal rank-one sheaf $F \in \cC(\Lambda_r,\sigma;\coeffs)$, where $\Lambda_r$ is as in \Cref{ex:Gr35}.  As in the proof of \Cref{thm:positroidasmoduli}, $F$ has five rank one stalks at the boundary of the picture, each of which includes into the rank 3 stalk $F_x$ in the middle.  If $F$ is framed, so we have isomorphisms $F_i \cong \coeffs$, then the quotient map $\coeffs^5 \onto F_x$ defines a $\coeffs$-point of the positroid stratum $\Pi_r$.}
\end{figure}

On the other hand, we have defined $\Lambda_r$ exactly so that its crossing conditions imply that the span in $F_x$ of any cyclically adjacent subset of $F_1,\dotsc,F_n$ has the rank specified by the corresponding entry of $r$.  Thus the image of $\cM^{fr}_{\:1}(\Lambda_r)$ is contained in $\Pi_r$. Conversely, given a $\coeffs$-point of $\Pi_r$, we can directly define a sheaf whose sections over a given region are just the column span associated to the relevant entry or entries of $r$ (note that while a $\coeffs$-point of $\Gr(k,n)$ is a quotient map $\coeffs^n \onto E$ onto a locally free $\coeffs$-module, $\coeffs$-points of $\Pi_r$ are described by quotients where $E$ is in fact free).  Finally, it follows from the fact that the images of the $F_i$ generate $F_x$ that all points of $\cM^{fr}_{\:1}(\Lambda_r)$ have trivial stabilizers, hence $\cM^{fr}_{\:1}(\Lambda_r) \cong \Pi_r$. 
\end{proof}

\begin{remark}
Many other objects in algebraic geometry and representation theory can be identified with positroid strata or their closures, see for example the discussion in \cite[Section 6]{KLS}.  One class of examples are the double Bruhat cells $GL_n^{u,v}$ of $GL_n$.  Here $u$ and $v$ are elements of the Weyl group of $GL_n$ and $GL_n^{u,v}$ is the intersection of the double cosets $B_- u B_-$, $B_+ v B_+$ of $u$, $v$ with respect to opposite Borel subgroups.  These can be embedded as positroid strata in $\Gr(n,2n)$; on the level of matrix representatives this is just concatenation with an identity matrix.  Another class of examples come from subvarieties of the full or partial flag varieties of $GL_n$ that map isomorphically onto their images in $\Gr(k,n)$ under the natural projection.  These include the moduli spaces of triples of flags in generic position considered in \cite{FG1}, which form building blocks in associating cluster charts on moduli spaces of local systems to triangulations of surfaces (see \Cref{con:triangle}).
\end{remark}

\begin{remark}\label{rmk:GM}
The construction of $\Lambda_r \subset T^\infty D^2$ naturally produces a Legendrian $\beta_r \subset T^\infty \R^2$ whose intersection with $T^\infty D^2$ is $\Lambda_r$. The front projection of $\beta_r$ is obtained from that of $\Lambda_r$ by adding caps around the outside of $D^2$ as pictured in \Cref{fig:introGr24} or \Cref{ex:Gr35}. Moreover, $\beta_r$ is a Legendrian braid satellite of a circle around $\infty$, as considered in \Cref{sec:microab}. The unframed moduli spaces $\cM_1(\Lambda_r)$ and $\cM_1(\beta_r)$ are isomorphic, and this space of $\beta_r$-filtered local systems is a configuration space of points in $\P^k$ satisfying open conditions.  The projection $\cM_{\:1}^{fr}(\Lambda_r) \to \cM_1(\Lambda_r)$ is a torus quotient, and this relationship is a version of Gelfand-Macpherson duality subject to the conditions imposed by $r$ \cite{GM}.
\end{remark}

\subsection{Wild character varieties}  \label{sec:wild}
In its simplest form, the Riemann-Hilbert correspondence asserts an equivalence of categories
between integrable meromorphic connections on a complex analytic space, with regular singularities
along a normal crossings boundary divisor, and locally constant sheaves in the complement of the divisor \cite{Del-RH}.  In particular, the parameter space of such regular connections, considered up to gauge equivalence, can be identified with the  space parameterizing representations of the fundamental group of the complement, up to isomorphism.  

The moduli space of connections is called the {\em de Rham} moduli space, and the moduli
space of locally constant sheaves is called the {\em Betti} moduli space.  The Riemann-Hilbert
correspondence asserts that these have the same points; in fact they are complex-analytically 
isomorphic, but have naturally different algebraic structures -- passing from connections to 
their monodromy involves an exponential.  We restrict attention to the case where the space on which we study connections is a smooth Riemann surface.  

The notion of regular singularities is essential in the above equivalence.  One formulation
is that the connection matrix can, analytically locally, be expressed with poles of order 
at most one.  Equivalently, the local solutions exhibit polynomial growth as one approaches
the singular point.  A consequence of this is that the classification of such 
connections up to analytic local gauge equivalence is the same as the classification up to 
formal local gauge equivalence; the local form of the Riemann-Hilbert correspondence is then
just the statement that both of these are characterized by the conjugacy class of the 
exponential of the singular term of the connection.  

To classify connections with possibly irregular singularities, one records {\em Stokes data}, 
i.e., information about the growth rates of solutions \cite{Mal, DMR, 
BV-formal, BV2, BV3}.  This is often formulated in the following
way: given a meromorphic connection on a disk $D^2$, analytic away from zero, the
space of solutions forms a locally constant sheaf $Sol$ on $D^2 \smallsetminus 0$, hence equivalently on
the real oriented blowup $\pi: \widetilde{D^2} \to D^2$ at $0$.  
Let $\cI$ be the totally ordered set of all possible growth
rates of the absolute value of the solution to a linear meromorphic ODE, modulo polynomial
growth rates -- we discuss what $\cI$ is more explicitly later.   Then the sheaf $Sol|_{\pi^{-1}(0)}$ 
carries a stalkwise filtration by $\cI$, varying continuously in an appropriate sense.  This
filtration is termed the Stokes filtration. 

The Riemann-Hilbert theorem in this possibly irregular case 
implies in particular the following three assertions: first, that connections
up to analytic local gauge equivalence are classified by their solution sheaves equipped
with Stokes filtrations; second, that connections up to formal local gauge equivalence 
are classified by their solution sheaves plus the associated graded of the Stokes filtration;
and finally, that if a given associated graded arises from some connection, then there exist
connections giving rise to any filtration with this associated graded
\cite{Mal}.

The relation to our setting is obtained by projecting the $\cI$ filtration to an $\R$ filtration, 
and then ``turning it sideways'' via the observation than a sheaf of $\R$-filtered objects on 
$X$ is the same as a sheaf on $X \times \R$ with microsupport confined to covectors
negative in the $\R$ direction.  Recording the jumping locus of the filtration by 
passing to the associated graded is just the same as recording the microsupport of 
the sheaf.  

Let us be more precise about how to produce the sheaf on $S^1 \times \R$.  
To this end we recall the 
formal classification of singularities of meromorphic ODE: any vector bundle 
on a disk equipped with a meromorphic connection $\nabla$, analytic away from zero, 
is, over the universal cover of the disk, {\em formally} gauge 
equivalent to some $\bigoplus (\alpha \otimes \nabla_\alpha)$ 
where each $\alpha$ is an irregular connection {\em of rank one} and $\nabla_\alpha$ is a
regular connection.  Note that the asymptotics of the holomorphic local solutions
are controlled by the asymptotics of the formal local solutions.  (The ``main asymptotic
existence theorem'' asserts a converse; that one can lift formal solutions to holomorphic
solutions with similar asymptotics.  It is a key step in the proof of the Riemann-Hilbert
correspondence, but in the black-box presentation we are giving it can be viewed as 
a consequence.)

Consider the 
rank one equation 
$$\frac{df}{d z} = \alpha \cdot f(z)\qquad \qquad \qquad \alpha \in \C((z^{1/\infty}))$$ 
Evidently the solution is $f = e^{\int \alpha dz}$. 
The regular part of the connection does not affect the growth rate of solutions
{\em modulo polynomial
growth rates}, i.e., the growth rate is determined by the class of $\alpha$ in 
$\C((z^{1/\infty})) / z^{-1} \C[[z^{1/\infty}]]$.  We will thus take $\alpha$ to have
no terms of degree greater than $-1$. 

We return to our description of the sideways Stokes sheaf.  Fix again some connection $\nabla$,
and after some gauge transformation defined over $ \C((z^{1/\infty}))$, expand it 
as $\nabla = \bigoplus (\alpha \otimes \nabla_\alpha)$.  Fix some $\epsilon \ll 1$, 
and plot, as a function of $\theta$, the (multivalued-)functions 
$n_{\alpha, r}(\theta) := \log |f(\epsilon e^{i \theta})| = Re ( (\int \alpha dz)_{z = r e^{i \theta}})$, for every $\alpha$ which appears 
in the above decomposition. 

Consider the sheaf $Sol|_{\pi^{-1}(0)}$, and pull 
it back to $\R \times S^1 $ under the projection of the $\R$ factor.  Fix now some 
$\epsilon \ll 1$.  
Note that a stalk of this sheaf is in fact a function on the circle; form the 
subsheaf $\mathbb{S}^{\epsilon}$ whose stalk at $(N, \theta)$ consists of solutions which grow
at most polynomially faster than any formal solution, the logarithm of whose
evaluation at $\epsilon e^{i\theta}$ is at most $N$.  That is: 
$$\mathbb{S}^{\epsilon}_{N, \theta} := \{f \in Sol_{N, \theta} \, | \, 
N \le n_{\alpha, \epsilon}(\theta) \implies 
 \log |f(re^{i \theta})| \le   n_{\alpha, r}(\theta) + O_{r \to 0}(1) \}$$

By construction, the sheaf $\mathbb{S}^\epsilon$ has microsupport at infinity 
equal to the Legendrian link whose front projection is the union of the graphs
of the $n_{\alpha, \epsilon}$, co-oriented towards $-\infty$ in $\R \times S^1$. 
We call this link the {\em Stokes Legendrian} of the connection, and term its
front projection the {\em Stokes diagram}.  Note that the Stokes diagram and
Stokes Legendrian depend only on the formal type.  

\begin{remark}
The fact that this filtration should be viewed as describing a Legendrian 
is mentioned in \cite{KKP}, its front projection having been 
drawn by Stokes himself \cite{Sto} (we thank Philip Boalch for bringing this last reference to our attention). 
\end{remark}

\vspace{2mm}
We can now state more precisely the irregular Riemann-Hilbert
correspondence.  Let $\Sigma$ be a surface, and $p_1, \ldots, p_k$ be points on $\Sigma$.  
Fix a formal type of irregular singularity $\tau_i$ at each $p_i$, i.e., choose some connection on 
a disk near each $\tau_i$, meromorphic on the disk and holomorphic away
from $p_i$, defined up to formal gauge equivalence, and up to changing the regular part of 
the connection.  That is, for the moment we take
our notion of formal type to mean that only the $\alpha$ are specified, and the $\nabla_\alpha$
are left to vary.  
Let $C_{dR}(\Sigma, \{p_i\}, \{\tau_i\})$
be the category of connections with the prescribed formal types.  

Let $\Lambda_i$ be the Stokes Legendrian of the singularity $\tau_i$.  
Draw the knot $\Lambda_i$ on $\Sigma$ by first passing to the real blowup 
$Bl_{p_i} \Sigma$, and then gluing the $\R \times S_1$ above to the inside of the 
boundary circle, with $\infty$ in the $\R$ factor facing `into' the surface.  One now has
a punctured surface; the puncture can be filled in and re-labelled $p_i$.  

The procedure we described locally before can now be performed globally over the surface. 
That is, if we write $C_B(\Sigma, \bigcup \Lambda_i, \bigcup p_i)$ for the subcategory of
$Sh_{\bigcup \Lambda_i}(\Sigma)$ in which the stalk of the sheaf vanishes at all $p_i$, 
then forming the global sideways Stokes sheaf of solutions defines a map 
$$C_{dR}(\Sigma, \{p_i\}, \{\tau_i\}) \to C_B(\Sigma, \bigcup \Lambda_i, \bigcup p_i)$$ 
The irregular Riemann-Hilbert correspondence implies this map is an equivalence. 

The $\nabla_\alpha$ on the de Rham side integrate to the 
microlocal monodromies on the Betti side.  In particular, when all the $\nabla_\alpha$ have
dimension one, the moduli space of objects in the above category is exactly of the sort we have been considering in this section. 

\begin{example}
The ODEs $f'' = z^n f$ generalize the Airy equation ($n=1$), and correspond to the $\cD$-module
$d - A,$ with $A = \;${\tiny $\begin{pmatrix}0&1\\z^n & 0\end{pmatrix}$}.  This gives an $SL_2$-flat connection,
and we will determine the formal type of the singularity by investigating the solutions ${f\choose f'}$ at $z=\infty$.
First put $x=z^{-1}$ to move the irregular singularity to the origin,
then define the differential operator $\Theta = z\frac{d}{dz} = -x\frac{d}{dx},$ after which the
equation becomes $Lf = 0,$ with $L = \Theta^2 + \Theta - x^{-(n+2)}.$  Newton's method instructs us how to
look at the most singular terms and use gauge transformations to reduce the order of the singularity so as
to develop the power-series parts of asymptotic solutions.  In this case, we find
$f_\pm = \exp(\mp\frac{2}{n+2}x^{-\frac{n+2}{2}})x^{\frac{n}{4}}\sum_m a_m x^{\frac{m}{2}}$.
Wasow's ``Main Asymptotic Existence Theorem'' \cite[Section 14]{W}
states that these represent the singularity types of actual solutions.  Which of the two $f_\pm(r\e^{i\theta})$ is most singular
as $r\to 0$ changes at $n+2$ values of $\theta,$ so we can read off the Stokes
data as the $(2,n+2)$ braid.  Compare with  Figure \ref{fig:introGr24}.
\end{example}

\section{Alternating Legendrians} 
\label{sec:conjlag}

In this section we construct exact Lagrangian fillings of Legendrians whose front projections have crossings of alternating orientations. The data of such a Legendrian can be encoded by a bicolored graph on the surface; 
in the terminology of \cite{Po}, the front projection is an {\em alternating strand diagram}. The smooth surface which underlies the exact Lagrangian filling has the same homotopy type as the
graph, and can be topologically identified with what has elsewhere been called
the {\em conjugate surface}  \cite{GK}. 
After defining its Lagrangian embedding we consider Lagrangian branes supported on the filling and the sheaf quantizations thereof, which we refer to as alternating sheaves.

\subsection{Alternating colorings}\label{sec:altcolor}

\begin{definition}
\label{def:ASD}
Let $\Sigma$ be a surface, and let $\Lambda \subset T^\infty \Sigma$ be a Legendrian in its cocircle
bundle
whose front projection $\pi(\Lambda)$ 
has only transverse intersections as singularities. 
An \textbf{alternating coloring} for $\Lambda$ is the data of, for each region in
the complement of the front projection, a label {\em black, white, or null},
subject to the following conditions.
\begin{itemize}
\item The boundary of a black region is co-oriented inward.
\item The boundary of a white region is co-oriented outward.
\item The boundary of a null region has co-orientations that alternate
between inward and outward at each crossing.
\item No black region shares a 1-d border with a white region, and
no null region shares a 1-d border with another null region. 
\end{itemize}
An \textbf{alternating Legendrian} $\Lambda$ is a Legendrian equipped with an alternating coloring.
\end{definition}

We term such colorings alternating because of the following characterization.

\begin{proposition}
A link with an alternating coloring has the property that, following along
any strand, 
successive crossing strands in the front projection have alternating co-orientations. 
For a one-component link, 
this is sufficient to guarantee the existence of an alternating coloring. 
\end{proposition}

We warn the reader that the condition on alternating co-orientations is not naively the same as the condition on over- and undercrossings which defines the notion of an alternating (topological) link in $\R^3$. Indeed, as our links live in a nontrivial circle bundle their crossings lack a canonical notion of over- and undercrossing. Moreover, an alternating coloring is in principle extra structure rather than merely a property. 

\begin{example}
Consider the link composed of two concentric circles in the plane, with the inner one
co-oriented outward and the outer one co-oriented inward.  There are three connected
components of the complement of the front projection.  This admits two 
alternating colorings: proceeding from inside to outside, the three components can
be labelled (white, null, white) or (null, black, null). 
\end{example}

However, this non-uniqueness can be excluded by requiring sufficient crossings.

\begin{proposition}
If every region of the complement of the front projection
abuts a crossing, there is at most one alternating coloring, and the fourth
condition above follows from the first three.  
\end{proposition}
\begin{proof} In the neighborhood of any crossing, there is at most one alternating
coloring, which moreover verifies the fourth condition. 
\end{proof}

\begin{remark}
Even if one is only ultimately interested in diagrams satisfying the condition of the proposition, 
the extra flexibility in our definition of alternating coloring is still needed for it be a local notion on $\Sigma$. 
\end{remark}

We assume from now on that $\Sigma$ is orientable; as such we can identify co-orientations with 
orientations, and do so by orienting the Legendrian
such that the ``hairs'' indicating its co-orientation always point {\bf to the left} when traveling in
the direction of the orientation. 
In terms of orientations rather than co-orientations, 
the Legendrian travels counterclockwise around a black region, and clockwise around a white one. 

\begin{figure}
\begin{tikzpicture}
\newcommand*{\boff}{10}; \newcommand*{\aoff}{35}; \newcommand*{\rad}{3};
\node (r) [matrix] at (2.7*\rad,0) {
\coordinate (b0) at (90-\boff:\rad); \coordinate (b1) at (90+\boff:\rad);
\coordinate (b2) at (90+72-\boff:\rad); \coordinate (b3) at (90+72+\boff:\rad);
\coordinate (b4) at (90+2*72-\boff:\rad); \coordinate (b5) at (90+2*72+\boff:\rad);
\coordinate (b6) at (90+3*72-\boff:\rad); \coordinate (b7) at (90+3*72+\boff:\rad);
\coordinate (b8) at (90+4*72-\boff:\rad); \coordinate (b9) at (90+4*72+\boff:\rad);
\coordinate (c0) at (60:.57*\rad); \coordinate (c1) at (120:.57*\rad);
\coordinate (c2) at (18:.8*\rad); \coordinate (c3) at (180-18:.8*\rad);
\coordinate (c4) at (-35:.6*\rad); \coordinate (c5) at (215:.6*\rad);
\coordinate (c6) at (-70:.6*\rad); \coordinate (c7) at (180+70:.6*\rad);
\coordinate (c8) at (90:0*\rad); \coordinate (w0) at (90:.9*\rad);
\coordinate (w1) at (180-18:.95*\rad); \coordinate (w2) at (18:.95*\rad);
\coordinate (w3) at (-126:.8*\rad); \coordinate (w4) at (180+126:.8*\rad);
\coordinate (l0) at (-90:.5*\rad); \coordinate (l1) at (180-18:.55*\rad);
\coordinate (l2) at (18:.55*\rad);
\draw[graphstyle] (w0) -- (l1) -- (w3) -- (l0) -- (w4) -- (l2) -- (w0);
\draw[graphstyle] (w1) -- (l1); \draw[graphstyle] (w2) -- (l2); \draw[graphstyle] (w0) -- (l0);
\foreach \c in {w0,w1,w2,w3,w4,l0,l1,l2} {\fill (\c) circle (\bvertrad);}
\foreach \c in {w0,w1,w2,w3,w4} {\fill[white] (\c) circle (\wvertrad);}
\draw[asdstyle,righthairs] (b0) to[out=110,in=70] (b1) to[out=180+70,in=80] (c1) to[out=180+80,in=60] (c5) to[out=180+60,in=60] (b4) to[out=180+60,in=230] (b5) to[in=-120,out=180-130] (c7) to[in=-105,out=180-120] (c8) to[in=180+20,out=180-105] (c0) to[in=-225,out=20] (c2) to[in=-170,out=180-225] (b8) to[in=10,out=180-170] (b9) to[in=180-100,out=-170] (c2) to[in=180-125,out=-100] (c4) to[in=180+200,out=-125] (c6) to[in=180-200,out=180+20] (c7) to[in=180+125,out=-200] (c5) to[in=180+100,out=125] (c3)  to[in=180+170,out=100] (b2) to[out=180-10,in=170] (b3) to[out=180+170,in=225] (c3) to[out=180+225,in=180-20] (c1) to[out=180+180-20,in=105] (c8) to[out=180+105,in=120] (c6) to[out=180+120,in=130] (b6) to[in=-60,out=180-230] (b7) to[in=-60,out=180-60] (c4) to[in=-80,out=180-60] (c0) to[in=-70,out=180-80] (b0);\\};
\end{tikzpicture}
\caption{The front projection $\pi(\Lambda)$ of an alternating Legendrian and the associated bipartite graph $\Gamma$.  Given $\pi(\Lambda)$, we recover $\Gamma$ by placing a black/white vertex in each region whose boundary is co-oriented inward/outward, then connecting these by edges passing through crossings.  Given $\Gamma$, we recover $\pi(\Lambda)$ by drawing paths going between midpoints of edges of $\Gamma$, co-orienting them away from white vertices and towards black vertices.}
\label{fig:altexample}
\end{figure}
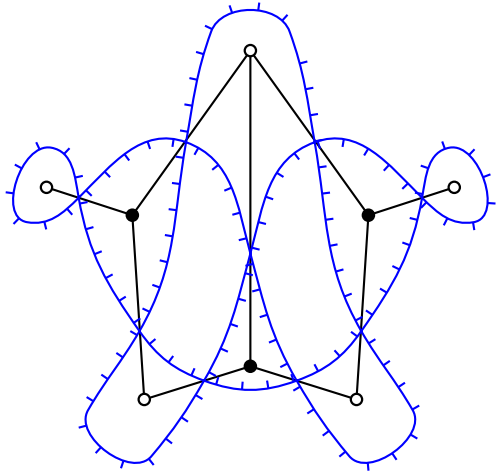

The front projections of alternating Legendrians have been considered elsewhere in the context of bipartite graphs, where they are referred to as {\em alternating strand diagrams} and their components as {\em zig-zag paths} \cite{Po,GK}.  A bicolored graph simply means one whose vertices are labeled white and black; if edges only connect vertices of distinct colors it is bipartite.   

\begin{proposition}
Let $\Gamma$ be a bicolored graph.  Then there is a unique Legendrian lift $\Lambda$ 
of the alternating strand diagram of $\Gamma$ such that the vertex coloring of $\Gamma$ 
gives the labels for an alternating coloring of $\Lambda$.  Every alternating Legendrian admits such a presentation. 
\end{proposition}
\begin{proof}
The alternating strand diagram of an embedded bicolored graph $\Gamma \subset \Sigma$ is determined up to planar isotopy by the following conditions: it lies in an open set that retracts onto $\Gamma$, its crossings are in bijection with edges of the graph meeting a vertex of each 
color, with one crossing lying on each edge, and these crossings are the only points where it
meets $\Gamma$.

Conversely, from an alternating Legendrian we can produce a bicolored graph that gives rise to it in the above fashion. This is simplest when each white and black region
is contractible; in this case we simply place appropriately colored vertices in the white and black regions and connect them with edges across crossings.  More generally, we further attach to each vertex
a configuration of embedded self-loops onto which its black/white region 
retracts. 
\end{proof}


\subsection{The conjugate Lagrangian} 
\label{subsec:lfasd}
From an alternating coloring of $\Lambda \subset T^\infty \Sigma$, we now construct 
an exact Lagrangian filling. 
We begin by describing the filling-to-be as an abstract topological surface, absent the embedding
into the cotangent bundle.  The desired surface coincides with that associated to a 
bipartite graph in \cite{FHKV,GK}, where in the latter it is called the {\em conjugate surface}. 

Let $\widehat{\Sigma}$ denote the real blow up of $\Sigma$ at the finite set of crossings of 
the front projection of $\Lambda$.  
The blow-down map $\widehat{\Sigma} \to \Sigma$ 
is a diffeomorphism away from the crossings, and the fiber above a crossing is the $\R\mathrm{P}^1$ of lines tangent to the the crossing.  
We denote by $W \subset \Sigma$ (resp. $B \subset \Sigma$) the union of the interiors of the white (resp. black) regions of the complement of the front projection. 

\begin{definition}
Let $\overline{L}$ denote the closure of 
the preimage of $W \cup B$ in $\widehat{\Sigma}$. It is a smooth surface
with boundary, and we refer to its interior $L$ as the \textbf{conjugate surface} of $\Lambda$. 
\end{definition}

The boundary of $\overline{L}$ is canonically homeomorphic to $\Lambda$.  The blowdown map identifies the white and black regions of $\Sigma$ with open subsets of $L$, which we also refer to as white and black regions.  Each exceptional $\R \mathrm{P}^1$-curve on $\widehat{\Sigma}$ meets $\overline{L}$ in an line segment that separates a white region from a black region.  We term such line segments ``exceptional arcs'', and sometimes indicate them in red as in 
Figure \ref{fig:redexample}. 

Let $\overline{T^*\Sigma}=T^*\Sigma\cup T^\infty \Sigma$ be the fiberwise compactification of the cotangent bundle by its real-oriented projectivization at infinity.

\begin{definition}
\label{def:conjlag}
A \emph{conjugate Lagrangian} is the image of an exact Lagrangian embedding $L \to T^*\Sigma$ such that
\begin{enumerate}
\item the composition of $L \to T^*\Sigma$ with the projection $T^*\Sigma \to \Sigma$ coincides with the blowdown map,
\item the intersection of the closure of $L$ in $\overline{T^*\Sigma}$ with the boundary $T^\infty \Sigma$ coincides with $\Lambda$ (hence $L$ extends to an embedding $\overline{L} \to \overline{T^*\Sigma}$),
\item for any neighborhood $U$ of $\Lambda$ in $\overline{T^*\Sigma}$ there is a Hamiltonian isotopy $\{\varphi_t\}_{t \in [0,1]}$ of $T^*\Sigma$, stationary outside $U \cap T^*\Sigma$ and with $\varphi_0$ the identity, such that $\varphi_t(L)$ satisfies (2) for all $t$ and such that $\varphi_1(L)$ is eventually conical.
\end{enumerate}
\end{definition}

We note that for some purposes it would be more convenient to simply consider an eventually conical Lagrangian of which $L$ is a perturbation, but we have found it more natural overall to arrange for the projection to be one-to-one away from the crossings. 
  


\begin{proposition}\label{prop:conjlagexistence}
For any alternating Legendrian $\Lambda$ the conjugate surface $L$ can be embedded into $T^* \Sigma$ as a conjugate Lagrangian.
\end{proposition}
\begin{proof}
It suffices to produce a functon $f:L\to \R$ with the following properties:  (1) $f$ is positive on $L\cap \pi^{-1}(B)$, negative on $\pi^{-1}(W)$, and zero over the crossings, and (2)
 $f$ is locally equal to $\pm \sqrt{n}$ in some local normal coordinate $n$ near a non-crossing boundary of a colored region.  Such a function is easy to arrange away from the crossings.  Outside of the crossings,
$L$ embeds in $T^*\Sigma$ as the graph of $df$.  Near a crossing, the following local
model in Example \ref{ex:localcrossingmodel} completes the proof of existence. Condition (3) is clear in the given local models, but more generally  it suffices to show that the
tangent spaces become $C^1$ close to spaces invariant under Liouville flow: then L
is a graph in a Weinstein neighborhood $T^*(\Lambda \times (R, \infty))$ of a collar neighborhood of $\Lambda$ near infinity, hence can be Hamiltonian isotoped to the zero section $\Lambda \times (R, \infty)$.
\end{proof}

\begin{example}
\label{ex:localcrossingmodel}
We coordinatize a neighborhood of a crossing on $\Sigma$ by $(x,y) \in \R^2,$ with $B$ the first quadrant, $W$ the third quadrant, and $\Lambda$ having front projection the union of the coordinate axes.
Coordinatize $L$ locally by $(s, t) \in \R \times (0,1)$ and define a map ${L}\to {T^*\Sigma}$ by
$$(\, x = s(1-t),\, y = st \; ; \; \xi = \sqrt{\frac{t}{1-t}},\, \eta = \sqrt{\frac{t}{1-t}}\, ).$$
This has closure $\overline{L}\subset \overline{T^*\Sigma}$ wtih boundary at infinity equal to $\Lambda$.
One easily checks that this is a conjugate Lagrangian embedding with primitive $$f = 2s\sqrt{t(1-t)}.$$
This is pulled back from the function $f = 2\,{\rm sgn}(x)\sqrt{xy}$ on $B \cup W$, which satisfies the conditions described in the proof of Proposition \ref{prop:conjlagexistence}.  Also note $\xi = \sqrt{y/x}, \eta = \sqrt{x/y}.$
\end{example}

\begin{remark}
We could have chosen other local models.  This one is designed to be readily compatible with our description of the square move --- see
Proposition \ref{prop:squaremovesurgery}.
\end{remark}

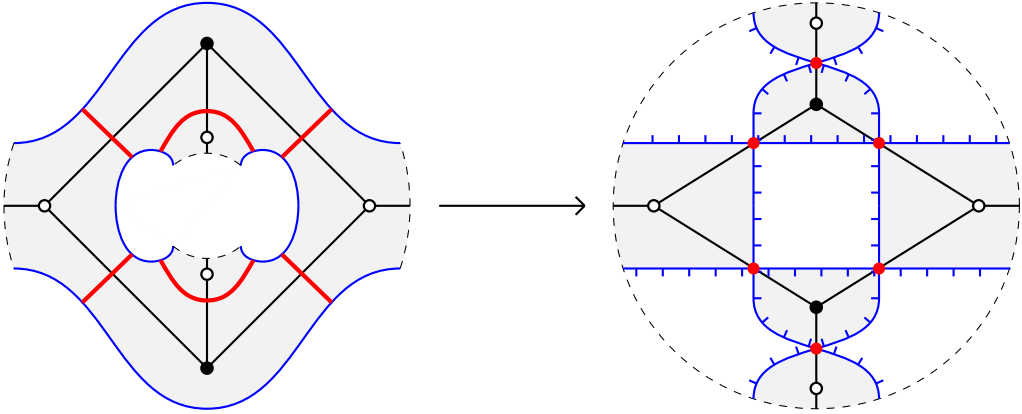
\begin{figure}
\centering
\begin{tikzpicture}
\newcommand*{\off}{18}; \newcommand*{\rad}{2.7}; \newcommand*{\vrad}{1.0};
\newcommand*{\offf}{40}; \newcommand*{\ang}{70}; \newcommand*{\opos}{.68};
\newcommand*{\smallrad}{.7}; \newcommand*{\xrad}{.8}; \newcommand*{\yrad}{.8};
\newcommand*{\ihpos}{.7}; \newcommand*{\ivpos}{.3}; 
\node (l) [matrix] at (0,0) {
\coordinate (b0) at (0,\yrad*\rad); \coordinate (b2) at (0,-\yrad*\rad);
\coordinate (w1) at (-\xrad*\rad,0); \coordinate (w3) at (\xrad*\rad,0);
\coordinate (w4) at (0,\smallrad*1.3); \coordinate (w5) at (0,-\smallrad*1.3);
\path[fill=\fcolor] (180-\off:\rad) to[out=0,in=180] (0,1*\rad) to[out=0,in=180] (\off:\rad) arc[start angle=\off,delta angle=-2*\off,radius=\rad] (-\off:\rad) to[out=180,in=0] (0,-1*\rad) to[out=180,in=0] (180+\off:\rad) arc[start angle=180+\off, delta angle=-2*\off,radius=\rad] (180-\off:\rad);
\path[fill=white] (-.45*\rad,0) to[out=90,in=180] (90+45:1.5*\smallrad) to[out=0,in=90] (90+\offf:\smallrad) arc[start angle=90+\offf,delta angle=-2*\offf,radius=\smallrad] (90-\offf:\smallrad) to[out=90,in=180] (90-45:1.5*\smallrad) to[out=0,in=90] (.45*\rad,0) to[out=-90,in=0] (-90+45:1.5*\smallrad) to[out=180,in=-90] (-90+\offf:\smallrad) arc[start angle=-90+\offf,delta angle=-2*\offf,radius=\smallrad] (-90-\offf:\smallrad) to[in=0,out=-90] (-90-45:1.5*\smallrad) to[in=-90,out=180] (-.45*\rad,0);
\draw[color=white,fill=white] (-.45*\rad,0) to (90-\offf:\smallrad) to (-90-\offf:\smallrad) to (-.45*\rad,0);
\draw[asdstyle] (0,1*\rad) to[in=180,out=0] coordinate[pos=\opos] (otr) (\off:\rad);
\draw[asdstyle] (0,1*\rad) to[out=180,in=0] coordinate[pos=\opos] (otl) (180-\off:\rad);
\draw[asdstyle] (0,-1*\rad) to[in=180,out=0] coordinate[pos=\opos] (obr) (-\off:\rad);
\draw[asdstyle] (0,-1*\rad) to[out=180,in=0] coordinate[pos=\opos] (obl) (180+\off:\rad);
\draw[asdstyle] (-.45*\rad,0) to[out=90,in=180] coordinate[pos=\ihpos] (itl1) (90+45:1.5*\smallrad) to[out=0,in=90] coordinate[pos=\ivpos] (itl2) (90+\offf:\smallrad);
\draw[asdstyle] (-.45*\rad,0) to[out=-90,in=180] coordinate[pos=\ihpos] (ibl1) (-90-45:1.5*\smallrad) to[out=0,in=-90] coordinate[pos=\ivpos] (ibl2) (-90-\offf:\smallrad);
\draw[asdstyle] (.45*\rad,0) to[out=90,in=0] coordinate[pos=\ihpos] (itr1) (90-45:1.5*\smallrad) to[out=180,in=90] coordinate[pos=\ivpos] (itr2) (90-\offf:\smallrad);
\draw[asdstyle] (.45*\rad,0) to[out=-90,in=0] coordinate[pos=\ihpos] (ibr1) (-90+45:1.5*\smallrad) to[out=180,in=-90] coordinate[pos=\ivpos] (ibr2) (-90+\offf:\smallrad);
\draw[dashed] (180-\off:\rad) arc[start angle=180-\off,end angle=180+\off,radius=\rad] (180+\off:\rad);
\draw[dashed] (-\off:\rad) arc[start angle=-\off,end angle=+\off,radius=\rad] (\off:\rad);
\draw[dashed] (90-\offf:\smallrad) arc[start angle=90-\offf,end angle=90+\offf,radius=\smallrad] (90+\offf:\smallrad);
\draw[dashed] (-90-\offf:\smallrad) arc[start angle=-90-\offf,end angle=-90+\offf,radius=\smallrad] (-90+\offf:\smallrad);
\draw[graphstyle] (b0) to (w1) to (b2) to (w3) to (b0);
\draw[graphstyle] (b0) to (90:\smallrad); \draw[graphstyle] (b2) to (-90:\smallrad); \draw[graphstyle] (w1) to (180:\rad); \draw[graphstyle] (w3) to (0:\rad);
\foreach \c in {b0,w1,b2,w3,w4,w5} {\fill (\c) circle (\bvertrad);}
\foreach \c in {w1,w3,w4,w5} {\fill[white] (\c) circle (\wvertrad);}
\draw[exceptarcstyle] (otl) to (itl1); \draw[exceptarcstyle] (otr) to (itr1);
\draw[exceptarcstyle] (obl) to (ibl1); \draw[exceptarcstyle] (obr) to (ibr1);
\draw[exceptarcstyle] (itl2) to[out=60,in=180] (0,\smallrad*1.8) to[out=0,in=180-60] (itr2);
\draw[exceptarcstyle] (ibl2) to[out=-60,in=180] (0,-\smallrad*1.8) to[out=0,in=180+60]  (ibr2);\\};

\node (r) [matrix] at (3*\rad,0) {
\coordinate (bl) at (180+90:\rad/2); \coordinate (br) at (0+90:\rad/2);
\coordinate (wt) at (90+90:\rad*.8); \coordinate (wb) at (-90+90:\rad*.8);
\coordinate (wr) at (0+90:\rad*.9); \coordinate (wl) at (180+90:\rad*.9);
\coordinate (ur) at ($(-\off+90:\rad)!\ifac!(-90+\off:\rad)$); \coordinate (lr) at ($(\off-90:\rad)!\ifac!(90-\off:\rad)$);
\coordinate (ul) at ($(\off+90:\rad)!\ifac!(-90-\off:\rad)$); \coordinate (ll) at ($(-\off-90:\rad)!\ifac!(90+\off:\rad)$);
\coordinate (mr) at ($(ur)!.5!(lr)$); \coordinate (ml) at ($(ul)!.5!(ll)$);
\path[fill=\fcolor, name path=p1] (90+90+\off:\rad) -- (-90+90-\off:\rad) arc[start angle=-\off,end angle=\off,radius=\rad] (\off:\rad) -- (180-\off:\rad) arc[start angle=180-\off,end angle=180+\off,radius=\rad] (180+\off:\rad);
\path[fill=\fcolor, name path=p2]  (180+90-\off:\rad)  to[out=90,in=-90] (lr) to (ur) [out=90,in=-90] to (\off+90:\rad) arc[start angle=90+\off,delta angle=-2*\off,radius=\rad] (90-\off:\rad) [out=-90,in=90] to (ul) to (ll) to[out=-90,in=90] (180+90+\off:\rad) arc[start angle=180+90+\off,delta angle=-2*\off,radius=\rad] (180+90-\off:\rad);
\path[fill=white, name intersections={of=p1 and p2, name=i}] (i-1) -- (i-2) -- (i-4) -- (i-3) -- cycle;
\draw[dashed] (0,0) circle (\rad);
\draw[graphstyle] (wb) to (br) to (wt) to (bl) to (wb);
\draw[graphstyle] (br) to (wr) to (90:\rad);
\draw[graphstyle] (bl) -- (wl) -- (180+90:\rad);
\draw[graphstyle] (wt) to (90+90:\rad); \draw[graphstyle] (wb) to (-90+90:\rad);
\foreach \c in {bl,br,wt,wb,wr,wl} {\fill[black] (\c) circle (\bvertrad);}
\foreach \c in {wt,wb,wr,wl} {\fill[white] (\c) circle (\wvertrad);}
\draw[asdstyle,righthairs,name path=h1] (90+90+\off:\rad) -- (-90+90-\off:\rad);
\draw[asdstyle,righthairs,name path=h2] (-90+90+\off:\rad) -- (90+90-\off:\rad);
\draw[asdstyle,lefthairs,name path=v1a] (mr) to (ur) [out=0+90,in=180+90] to (\off+90:\rad);
\draw[asdstyle,righthairs,name path=v1b] (mr) to (lr) to[in=0+90,out=180+90] (180+90-\off:\rad);
\draw[asdstyle,righthairs,name path=v2a] (ml) to (ul) [out=0+90,in=180+90] to (-\off+90:\rad);
\draw[asdstyle,lefthairs,name path=v2b] (ml) to (ll) to[in=0+90,out=180+90] (180+90+\off:\rad);
\foreach \p/\q in {h1/v1b,h1/v2b,h2/v1a,h2/v2a,v1a/v2a,v1b/v2b} {\fill[name intersections={of={\p} and {\q}, name=i},red] (i-1) circle (.08);}\\};
\coordinate (rw) at ($(r.west)+(-2mm,0)$); \coordinate (le) at ($(l.east)+(2mm,0)$);
\draw[thick] (le) -- ($(rw)+(-2mm,0)$) -- (rw);
\draw[arrhdstyle] ($(rw)+(-1.2mm,1.2mm)$) -- (rw); \draw[arrhdstyle] ($(rw)+(-1.2mm,-1.2mm)$) -- (rw);
\end{tikzpicture}
\caption{\label{fig:redexample} 
The right picture shows an alternating Legendrian in $T^\infty D^2$ and the associated bicolored graph.  The left shows its conjugate Lagrangian $L$ together with the strict transform of the bicolored graph under $L \to D^2$. The shaded regions on the right indicate the image of $L \to D^2$. The exceptional arcs and their images are indicated in red.}
\end{figure}

\begin{remark}
\label{rmk:zerosection}
If we isotope $\Lambda$ so that at each crossing both strands are tangent to the graph $\Gamma$ to all orders, hence nontransverse to each other, we can arrange a conjugate Lagrangian embedding of $L$ such that its intersection with the zero section is exactly $\Gamma$. To do this we choose $f \in C^\infty(B \cup W)$ so that on the white regions it is equal to $0$ along $\Gamma$ and increases monotonically to $\infty$ towards the boundary of $L$ (so the preimage of $(R,\infty)$ retracts onto $\Gamma$ for all $R$).  In particular the critical locus of $f$ on the white regions is exactly $\Gamma$.  Likewise, we ask that on the black regions $f$ is zero along $\Gamma$ and decreases monotonically to $-\infty$ towards the boundary. Then the tangency condition at the crossing implies that the graph of $df$ above $B \cup W$ has as its closure a conjugate Lagrangian filling of $\Lambda$ -- the exceptional arcs are embedded as the conormal lines to $\Gamma$ at the crossings.  It is clear that $L$ is exact since it retracts onto its intersection with the zero section, where the Liouville form is zero.
\end{remark}

\subsection{Alternating sheaves}\label{sec:altsh}

Given an alternating Legendrian $\Lambda \subset T^\infty \Sigma$ and a
conjugate Lagrangian $L$, we are interested in Lagrangian branes supported on $L$ and their sheaf quantizations. 


Note first that since
each cotangent fiber intersects $L$ in a unique point above the white and black regions, and in no points
above the null regions, any such quantization will be supported on the closure of the union of the black and white regions.
Moreover, the stalks of the sheaf on these regions will be locally constant and isomorphic to the stalk of the local system on $L$, with a degree shift between white and black regions. 
Since $\partial L = \Lambda$, the resulting sheaf
will have microsupport at infinity contained in $\Lambda$.
Thus we are led to consider sheaves of the following form:

\begin{definition}
Let $\Lambda \subset T^\infty \Sigma$ be equipped with an alternating coloring. 
An \textbf{alternating sheaf} is an object of $\Sh_{\Lambda}(\Sigma;\coeffs)$ 
whose support is contained in the closure of the union of the white and black regions. 
\end{definition}

By a locally costandard sheaf on $\Sigma$ we mean a sheaf of the form $j_!\cL$ for a locally constant sheaf $\cL$ of invertible $\coeffs$-modules on an open subset $U$ with inclusion $j: U \into \Sigma$.  Likewise, by a locally standard sheaf we mean any sheaf of the form $j_*\cL$ for such an $\cL$.  

\begin{proposition}\label{prop:altsheaftriangle}
Let $\cF$ be an alternating sheaf whose microstalks have cohomology vanishing outside of degree zero.
Then $\cF$ fits into an exact triangle
$$\cF_W[1] \to \cF \to \cF_B \xrightarrow{[1]}$$
where $\cF_W$ is a locally costandard sheaf supported on $\overline{W}$ and $\cF_B$ a locally standard sheaf supported on $\overline{B}$. 
\end{proposition}
\begin{proof}
Let $w: W \to \Sigma$ and $b: B \to \Sigma$ be the inclusions of 
the interior of the (open) white and black regions, respectively. 
Then it suffices to show $\cH^0(\cF) \cong w_! w^! \cF$ is a locally costandard sheaf
on the union of the white regions, $\cH^1(\cF) \cong b_* b^* \cF[1]$ is a locally standard sheaf
on the union of the black regions, and all other cohomology sheaves vanish. 
Note that the closure of a given white region is disjoint from all other white regions; similarly
the closure of a given black region is disjoint from all other black regions. 

In a neighborhood $U$ of a smooth point $p$ of the front projection, one has either a null
region separated from a white region, or a black region separated from a null region.  In the first case consider the morphism $\cF_n \to \cF_w$ of stalks at either end of a characteristic path between points $n$ and $w$ in the null and white regions of $U$; in the second consider similarly the morphism $\cF_b \to \cF_n$ between points $b$ and $n$ in the black and null regions.
In both cases, the cone has, by assumption,
cohomology only in degree zero, and also by assumption $\cF_n = 0$.  It follows
that $\cF_b$ has cohomology only in degree 1, and $\cF_w$ has cohomology only in degree 0. 

The microsupport prescribes
that the generization maps from stalks at the boundary of the black regions into the black
regions give isomorphisms, and that the generization maps from the stalks on the 
smooth boundaries of the white regions to the nearby null regions are isomorphisms
(to zero).  In particular, the stalks of $\cH^0(\cF)$ vanish outside the open white regions, 
and $\cH^1(\cF)$ is a locally constant sheaf on the closure of the black regions. 
The result follows. 
\end{proof}

The triangles appearing in Proposition \ref{prop:altsheaftriangle} are classified by  $\sheafhom(\cF_B[-1],\cF_W[1])$; we now recall from \cite{STZ}
which such extensions give objects in $Sh_\Lambda(\Sigma)$. 

\begin{proposition}\label{prop:altsasexts}
Let $\cL_W$ and $\cL_B$ be local systems on the interiors of the white
and black regions of an alternating coloring.  Then 
$\sheafhom(b_* \cL_B [-1], w_! \cL_W[1])$ is a direct sum of skyscraper sheaves in degree zero at the crossings of $\pi(\Lambda)$. The stalk at a crossing is isomorphic to $\Hom(\ell_W, \ell_B)$, where $\ell_W$, $\ell_B$ are generic stalks in the white and black regions of a neighborhood of the crossing.

The extension determined by a given class in $\Ext^1(b_* \cL_B, w_! \cL_W[1])$ 
has microsupport in $\Lambda$ iff the corresponding local elements of $\Hom(\ell_W, \ell_B)$
are all isomorphisms. 
\end{proposition}

\begin{proof}
Recall that $\sheafhom(X, Y) = \D( \D Y \otimes X)$, where $\D$ denotes Verdier duality.  
So 
\begin{align*}\sheafhom(b_* \cL_B[-1],w_! \cL_W[1]) &= \D ( \D (w_! \cL_W [1]) \otimes b_* \cL_B[-1]) \\
                                                    &= 
  \D(w_* \cL_W^\vee[1] \otimes b_* \cL_B[-1]) = \D(w_* \cL_W^\vee \otimes b_* \cL_B).
\end{align*}  
If $p$ is any point, then 
$(w_* \cL_W^\vee \otimes b_* \cL_B)_p = (w_* \cL_W^\vee)_p \otimes (b_* \cL_B)_p$,
which can only be nonzero for $p$ in the intersection of the closures of the white
and black regions, i.e., at a crossing.  The above formula
shows that the stalk here is evidently the hom space
between nearby stalks of the local systems in the white and black regions.  This proves the 
first statement. 

The second statement follows from a direct computation as in \cite[Theorem 3.12]{STZ}.  This calculation can be packaged as the statement that 
the above hom space is also the stalk of the Kashiwara-Schapira $\mu hom$ sheaf 
(see \cite[Sec. 6]{KS}) along the interior of $ss(w_! \cL_W[1]) \cap ss(b_* \cL_B[-1])$
--- the covectors pointing into the black region ---
and the cone over the stalk of the $\mu hom$ becomes the microstalk of the cone.  

Another way to see this is to perform a contact transformation, moving 
the Legendrian graph $ss(w_! \cL_W[1]) \cap ss(b_* \cL_B[-1])$ to one whose front
projection is locally an embedding near the desired microstalk, whereupon the 
desired $\mu hom$ calculation reduces to the above $\sheafhom$ calculation. 
\end{proof}

\begin{theorem} \label{thm:charts}
The full subcategory of $\Sh_\Lambda(\Sigma;\coeffs)$ consisting of alternating sheaves is equivalent to the category of locally constant sheaves on $L$.
\end{theorem}
\begin{proof}
The preceding propositions provide a complete description of the full subcategory of alternating sheaves as glued out of locally constant sheaves on white and black regions. To see the claim, we observe that an identical description applies to locally constant sheaves on $L$.

Indeed, it is clear that a local system $\cL$ on $L$ fits into a triangle
$$w_!\cL_W \to \cL \to b_*\cL_B \xrightarrow{[1]},$$
where $\cL_W$, $\cL_B$ are local systems on the white and black regions and $w$, $b$ now denote the inclusion of these regions into $L$ (rather than $\Sigma$). We now have that $\sheafhom(b_* \cL_B, w_! \cL_W[1])$ is a direct sum of constant sheaves supported on the exceptional arcs. Their stalks are locally isomorphic to $\Hom(\ell_W, \ell_B)$ and the extensions which are locally constant are classified by sections nonvanishing on all arcs.

From this description it is clear not only that we have a correspondence at the level of objects, but that morphisms in the indicated categories are the same: they can be reduced to identical calculations involving the outer terms in the triangle above and that of Proposition \ref{prop:altsheaftriangle}.
\end{proof}





Recall from the proof of Proposition \ref{prop:open} that full faithfulness implies we have an open inclusion of moduli spaces.

\begin{corollary}
The locus of $\cM_1(\Lambda)$ parametrizing alternating sheaves of microlocal rank one is open and isomorphic to the algebraic torus $\Loc_1(L)$. 
\end{corollary}

As indicated earlier, it is straightforward to see that alternating sheaves are exactly the objects obtained from sheaf quantization of the conjugate Lagrangian. Note that to discuss its quantization we perturb $L$ to be eventually conical as guaranteed by condition (3) in Definition \ref{def:conjlag}. 


\begin{proposition}
An object of $Sh_{\Lambda}(\Sigma)$ is obtained by sheaf quantization of a rank one local system on $L$ if and only if it is an alternating sheaf of microlocal rank one.
\end{proposition}
\begin{proof}
By Corollary \ref{cor:quantranks} the sheaf quantization of a rank one local system on $L$ is supported on the union of the closures of the white and black regions and has rank one stalks (in some degree) on their interiors. It follows that it has microlocal rank one. Since it is microsupported at infinity along $\Lambda$, it follows from Proposition \ref{prop:altsasexts} that it is alternating. The only if part of the statement follows since a fully faithful functor from $\Loc_1(L)$ to itself which is defined over $\Z$ must be an equivalence. 
\end{proof}

Since $\Loc_1(L)$ is also isomorphic to the space of local systems on $\Gamma$, it has natural coordinates described by holonomies around the faces of $\Gamma$ (that is, around the contractible regions of $\Sigma \smallsetminus \Gamma$). However, we will see later that it is also natural twist our identification between alternating sheaves and local systems on $L$ by signs.


Following \cite[Prop. 5.12]{STZ} such sign choices may be organized as follows. Suppose $\cF$ is an alternating sheaf and $\ell_W$, $\ell_B$ stalks of $\cF_W$, $\cF_B$ in the neighborhood of a fixed crossing of $\pi(\Lambda)$. Picking one of the two components of $\Lambda$ above the crossing picks out an isomorphism between $\ell_W$ and $\ell_B[1]$: each is isomorphic to the microstalk of $\cF$ at a point of that component on either side of the crossing, and parallel transport in the microlocalization $\cF_\Lambda$ defines an isomorphism between these microstalks. Choosing the other component changes the isomorphism $\ell_W \cong \ell_B[1]$ by a sign. 

In particular, if we choose a component of $\Lambda$ above every crossing these isomorphisms between stalks of $\cF_W$ and $\cF_B[1]$ assemble into a local system on $L$: the sheaves $\cF_W$ and $\cF_B[1]$ define a canonical local system on the complement of the exceptional arcs, and the construction above defines a parallel transport across the exceptional arcs. On the other hand, since $\Sigma$ is oriented there is a consistent notion of which component of $\Lambda$ is clockwise from the white/black regions and which is counterclockwise.

The construction described in Theorem \ref{thm:charts} corresponds to making the same choice at all crossings. We refer to the resulting isomorphism with $\Loc_1(L)$ as the \textbf{standard trivialization} of the space of alternating sheaves.

\begin{definition}\label{def:coords}
The \textbf{standard face coordinates} on the space of alternating sheaves are the counterclockwise holonomies around the faces of $\Gamma$ under the standard trivialization. The \textbf{positive face coordinates} are the negatives of the standard face coordinates.
\end{definition}

The positive coordinates are so-called because, as we will see in Section \ref{sec:squaremove}, their transformations are described by subtraction-free expressions. We use the term coordinate somewhat loosely: depending on the number of contractible regions of $\Sigma \smallsetminus \Gamma$ their boundaries may not form a basis of $H_1(\Gamma;\Z)$.

\section{Cluster combinatorics from Legendrian isotopy}
\label{sec:cluster}

Thus far, we have considered structures which arise from the geometry of
a Legendrian link in a fixed position.  We turn now to 
comparisons between these structures arising from Legendrian isotopies. 

At the level of categories or of moduli spaces, isotopies give rise to equivalences: 
given an isotopy $\Lambda \to \Lambda'$, one gets by \cite{GKS} an equivalence 
$Sh_\Lambda(\Sigma) \to Sh_{\Lambda'}(\Sigma)$, and a corresponding isomorphism
of the moduli spaces.  
However, different isotopy representatives of $\Lambda$ present different structures 
on the moduli space.  In particular, we saw in Section \ref{sec:braids} that an isotopy representative
in which $\Lambda$ is presented as a union of positive braids has, in some cases,
a canonical identification with a positroid stratum or a wild character variety.  On the other hand, 
we saw in Section \ref{sec:conjlag} that an isotopy representative which is alternating comes with
a natural filling $L$, hence its moduli space has an abelian chart $\Loc_1(L) \into 
\cM_1(\Sigma, \Lambda)$.

This raises a series of questions: 
\begin{enumerate}
\item Which Legendrians have alternating representatives; how many are there and what 
are the isotopies between them?
\item Given an isotopy between alternating Legendrians, what is the change of coordinates
between the corresponding abelian charts?
\item Given an isotopy from an alternating Legendrian to a localized positive braid, can the 
coordinates of the abelian chart from the filling be written in terms of some coordinates natural from
the nonabelian point of view of the positive braid?
\end{enumerate}

The first question is one of topological combinatorics, and the foundational results in this
direction are due to D. Thurston \cite{Thu}.  We survey and extend his results in Section \ref{sec:Dylan}, showing in particular that the alternating Legendrian of a reduced plabic graph admits a homotopically
unique isotopy to a Legendrian of positroid type.  Following ideas of \cite{GK} we show that this leads to alternating representatives of the Legendrian braid satellites of \Cref{sec:braids}.

Though the second and third questions are implicitly Floer theoretic in nature, they can be reduced to combinatorics given the results collected so far. In the previous section, 
we established a sheaf-theoretic description of the abelian charts, and 
the 
constructible sheaves under consideration can be described locally in terms of quiver representation theory. 
Any isotopy can be factored into a sequence of Reidemeister moves, and the isomorphism
induced by \cite{GKS} factors accordingly; each term in this factorization 
can be described explicitly as reviewed in \Cref{sec:contactinv}.  

Towards the second question, 
we show that the {\em square move} on bicolored graphs is interpolated by a 
Legendrian isotopy of their corresponding Legendrians. This move is fundamental, for example any isotopy between reduced alternating Legendrians in $T^\infty D^2$ can be factored into a sequence of square moves. We show that the abelian charts on either side of a square move are related by the cluster $\cX$-transformation classically associated to the square move.
The conjugate Lagrangians themselves are related by Lagrangian surgery, a perspective 
which we develop more systematically in \cite{STW}. 

In the direction of the third question, we consider the unique isotopy from the alternating
Legendrian associated to a reduced plabic graph to the corresponding positroid braid.  Since
this isotopy gives abelian coordinates on the Grassmannian, the natural question is how to express
these in terms of Pl\"ucker coordinates.  We identify the resulting expression with the 
boundary measurement map of Postnikov \cite{Po}, which describes the answer in terms of the combinatorics of flows or perfect matchings on the graph.

\subsection{Alternating Legendrians from braids} \label{sec:Dylan}
We consider here the existence of alternating isotopy representatives of the Legendrians studied in \Cref{sec:braids}. These were braid satellites of cocircle fibers of $T^\infty \Sigma$, and their rank-one moduli spaces were spaces of filtered local systems on $\Sigma$.  The main result is that essentially all such Legendrians have alternating representatives. From our point of view this accounts for the appearance of bicolored graphs in the study of such spaces. After proving the general statement we explain how in various examples alternating representatives can be constructed explicitly, bringing us into contact with the combinatorics of triangulations and double wiring diagrams familiar in cluster theory.

We begin our discussion with a class of particularly simple Legendrians:

\begin{definition}\label{def:reduced} A Legendrian $\Lambda \subset T^\infty D^2$
is \textbf{reduced} if it satisfies the following conditions:
\begin{enumerate}
\item Along $\partial D^2$ the strands of $\pi(\Lambda)$ have alternating orientations.
\item The front projection $\Lambda \to D^2$ is an immersion (that is, $\pi(\Lambda)$ has no cusps).
\item There are no parallel crossings: if $p_1$, $p_2$ are intersection points of two strands, one is oriented from $p_1$ to $p_2$ and the other from $p_2$ to $p_1$.
\item No strands have self-intersections.
\item All strands meet the boundary of the disk.
\end{enumerate}
\end{definition}

Our terminology is modeled that of \cite{Po} for bicolored graphs: an embedded bicolored graph $\Gamma \subset D^2$ is a \textbf{reduced plabic graph} if its associated Legendrian is reduced. Note that our conventions implicitly allow us to assume that a reduced plabic graph has white vertices where it meets the boundary of $D^2$.

If $\Lambda \subset T^\infty D^2$ is reduced the set $\pi(\Lambda) \cap \partial D^2$ of intersections between its front projection and the boundary of the disk are divided into sets of incoming and outgoing points (we freely pass between orientations and co-orientations following \Cref{sec:altcolor}).  Each strand of $\pi(\Lambda)$ has one incoming endpoint and one outgoing endpoint, hence $\Lambda$ defines a matching between these two sets.  Conversely, we can fix a set of points on $\partial D^2$, label them alternatively incoming and outgoing, choose a matching between those of opposite labels, and ask for reduced Legendrians realizing this matching.  In this direction we have the following reformulation of a fundamental result of D. Thurston:

\begin{proposition}\label{prop:Thurston}
\cite{Thu} Fix a set of points on $\partial D^2$ alternatively labeled as incoming and outgoing.  Every matching between incoming and outgoing points is realized by a reduced alternating Legendrian in $T^\infty D^2$.  Moreover, any two reduced alternating Legendrians with the same matching are Legendrian isotopic through a series of square moves.
\end{proposition}

We note in passing that while in applications this fact is often used as a purely combinatorial statement (e.g. \cite{Po}), its relevance to Legendrian knot theory was specifically anticipated in \cite{Thu}.  We can complement the part of \Cref{prop:Thurston} dealing with isotopies as follows:

\begin{proposition}\label{prop:uniqueisotopies}
Suppose $\Lambda, \Lambda' \subset T^\infty D^2$ are reduced Legendrians such that $\pi(\Lambda) \cap \partial D^2 = \pi(\Lambda') \cap \partial D^2$ compatibly with incoming/outgoing labels.  If $\Lambda$ and $\Lambda'$ define the same matching of boundary points, they are Legendrian isotopic. This isotopy can be chosen so that it is stationary above $\partial D^2$, and only passes through Legendrians whose front projections are immersions.  Moreover, the space of such isotopies is contractible.
\end{proposition}
\begin{proof}
We construct an isotopy $\Lambda \to \Lambda'$ of the stated kind as follows.  We notate it as a family of Legendrian embeddings $f_t:\Lambda \to T^\infty D^2$ depending smoothly on $t \in [0,1]$ such that $f_0$ is the identity map on $\Lambda$ and $f_1$ is a diffeomorphism from $\Lambda$ to $\Lambda'$.  Number the components of $\Lambda'$ (hence also $\Lambda$) $1$ through $m$, denoting the $k$th components by $\Lambda'_k$, $\Lambda_k$.  As $t$ varies from $(\ell-1)/m$ to $\ell/m$, we take $f_t$ to be independent of $t$ except along $\Lambda_\ell$. 

For $(\ell-1)/m \leq t \leq \epsilon + (\ell-1)/m =: t_0$ we let $f_t$ be a small perturbation such that the part of the front projection of $\Lambda'_\ell$ that does not meet the front projection of $f_{t_0}(\Lambda_\ell)$ (that is, $\pi(\Lambda'_\ell) \smallsetminus \pi(f_{t_0}(\Lambda_\ell)) \cap \pi(\Lambda'_\ell)$) has finitely many components. We define $f_t$ for $t_0 \leq t\leq \ell/m$ inductively as follows.  Suppose $t_i$ is such that $\pi(\Lambda'_\ell) \smallsetminus \pi(f_{t_i}(\Lambda_\ell)) \cap \pi(\Lambda'_\ell)$ has finitely many components. If there is only one such component, let $t_{i+1} = \ell/m$, otherwise let $t_{i+1}$ be between $t_i$ and $\ell/m$.  Let $C'$ be the component of $\pi(\Lambda'_\ell) \smallsetminus \pi(f_{t_i}(\Lambda_\ell)) \cap \pi(\Lambda'_\ell)$ closest to one end of $\pi(\Lambda'_\ell)$.  Let $C$ be the segment of $\pi(f_{t_i}(\Lambda_\ell))$ which has the same endpoints as the closure of $C'$.  Together $C'$ and $C$ form the  boundary of an embedded disk, since by assumption there are no self-loops in either.  For $t_i \leq t \leq t_{i+1}$ we let $f_t$ act on the front projections by retracting this disk onto the part of its boundary lying along $C'$.  That this can be done so that it lifts to a Legendrian isotopy follows from the assumption that there are no parallel crossings or cusps (the part of $\pi(\Lambda_\ell)$ just past the end of $C$ should also be perturbed in order to not create a corner in the front projection).  

To show the space of such isotopies is contractible it suffices to show contractiblity of the group $\Aut(\Lambda)$ of Legendrian isotopies from $\Lambda$ to itself that are stationary at the boundary and pass through Legendrians whose front projections are immersions.  To do this it suffices to describe, for any element $g_s$ of $\Aut(\Lambda)$ and any $s \in [0,1]$, an isotopy $f_t$ from $g_s(\Lambda)$ to $\Lambda$ which is smooth in $s$, is the stationary isotopy at $s = 0$, $1$, and which itself only passes through Legendrians with immersed front projections.  But this can be done using the same prescription we used to construct an isotopy from $\Lambda$ to $\Lambda'$.  It is clearly continuous in $s$, and $f_t$ limits to the stationary isotopy at $s=0$, $1$ since the process of retracting the embedded disks does not increase their size. 
\end{proof}

This guarantees the existence of alternating isotopy representatives of reduced Legendrians in $T^\infty D^2$. We have the following more general existence theorem, which is proved by reduction to \Cref{prop:Thurston} following a strategy employed in \cite{GK} for a different class of alternating strand diagrams on $T^2$. We follow the notation of \Cref{sec:microab}, including the use of $\beta_i$ for both a choice of abstract braid at $\sigma_i$ and for the associated Legendrian satellite in $T^\infty(\Sigma \smallsetminus \sigma)$.

\begin{theorem}
\label{thm:wildcharvty}
Let $\Sigma$ be a closed surface, $\sigma = \{\sigma_1,\dotsc,\sigma_k\}$ a nonempty collection of $k$ marked points, and $\sigma_i \mapsto \beta_i \in Br_n^+$ a choice of positive braid
at each marked point.  If $\Sigma$ has genus zero and $k = 1$, assume $\beta_1$ can be written as $\beta' \Delta^2$ where $\Delta$ denotes a half-twist.  Then in $T^\infty (\Sigma \smallsetminus \sigma)$ the associated Legendrian $\beta = \coprod \beta_i$ 
is Legendrian isotopic to an alternating Legendrian. 
\end{theorem}

\begin{proof}
If $\Sigma$ has genus zero and $k < 3$, this follows from \Cref{con:one,con:two} (described after the proof), so from now on we assume $k \geq 3$ in the genus zero case.  The key point in general is to cut $\Sigma$ apart into a polygon in such a way that \Cref{prop:Thurston} may be applied. While spelling this out in detail is regrettably tedious, it is ultimately an elementary construction. If $g$ is the genus of $\Sigma$, fix a $(2g+2)$-gon $P$ with a gluing map $p: P \onto \Sigma$; that is, $\Sigma$ is obtained from $P$ by gluing pairs of edges together.  The image in $\Sigma$ of the boundary $\partial P$ is an embedded graph $C$, either an interval if $g=0$ or a bouquet of circles.  We choose $C$ so that:
\begin{enumerate}
\item $\sigma \subset C$ and the $\sigma_i$ lie in order along on a single component of the smooth locus of $C$.  
\item The front projection of each $\beta_i$ intersects $C$ in $2n$ points, so that $C$ separates it into a pair of $n$-strand braids.  
\item On one side of $C$ all such braids are trivial.
\end{enumerate}

The preimage $p^{-1}(\pi(\beta))$ of the front projection of $\beta$ then consists of $k$ disjoint braids attached to $\partial P$ along each of two edges.  We call these edges $A$ and $A'$, letting $A'$ denote the edge where all the braids are trivial by item (3).

We now subdivide $P$ further into a union $P = P' \cup B_i \cup \cdots \cup B_k$ of smaller polygons with pairwise disjoint interiors.  The role of each $B_i$ will be to isolate the nontrivial part of each braid $\beta_i$.  That is, we choose them to satisfy the following:

\begin{enumerate}
\item Each $B_i$ is a quadrilateral such that $B_i$ meets $\partial P$ only along $A$, and $B_i \cap \partial P$ is an edge of $B_i$.
\item The image in $\Sigma$ of $B_i \cap \partial P$ does not meet $\pi(\beta)$, and contains $\sigma_i$ but no other points of $\sigma$.
\item The interior of $B_i$ contains all crossings in $p^{-1}(\pi(\beta_i))$, and $B_i \cap p^{-1}(\pi(\beta))$ is an $n$-strand braid diagram with all strands going from one edge of $B_i$ to its opposite (hence each of these edges shares an endpoint with $B_i \cap \partial P$).  
\end{enumerate}

\begin{figure}
\includegraphics{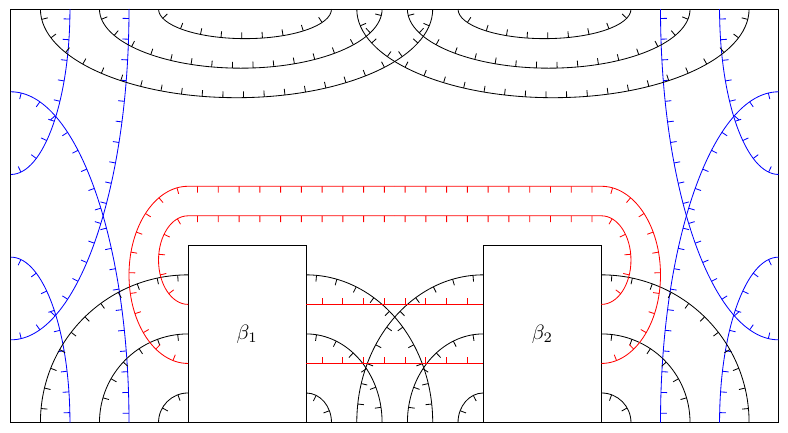}
\caption{\label{fig:existence} The proof of \Cref{thm:wildcharvty} when $\Sigma = T^2$ and $\sigma$ consists of two points. The polygon $P$ is pictured with two squares $B_1$ and $B_2$ cut out to yield the crenellated polygon $P'$. The front projection of $\beta = \beta_1 \sqcup \beta_2$ is in black, the nontrivial part of each braid lying in $B_1$, $B_2$, respectively. The front projections of the auxiliary Legendrians $\Lambda_B$ and $\Lambda_C$ are in red and blue, respectively. Taking their connect sums with $\beta$, we obtain a new Legendrian which is isotopic to $\beta$ and whose co-orientations alternate along the boundary of $P'$.}
\end{figure}

We next isotope $\beta$ so that neighboring front projections $\pi(\beta_i)$ and $\pi(\beta_{i+1})$ overlap in an alternating fashion, stretching each one out along $C$. 
That is, we require the following:
\begin{enumerate}
\item The isotopy is only nontrivial on the complement of $p(\sqcup B_i)$.
\item On the interval of $C$ lying between $p(B_i \cap \partial P)$ and $p(B_{i+1} \cap \partial P)$, the intersections of $\pi(\beta_i)$ and $\pi(\beta_{i+1})$ with $C$ alternate co-orientations ater the isotopy.
\item The number of intersections of $\pi(\beta)$ with $C$ remains constant through the isotopy.
\item After the isotopy there are $n(n-1)(k-1)$ new crossings of $\pi(\beta)$ (the minimal possible number created in order to satisfy (1)--(3)).
\end{enumerate}

In particular, $\pi(\beta_i)$ is still disjoint from $\pi(\beta_j)$ unless $|i-j| \leq 1$.  From now on $\beta$ refers to the result of this isotopy.

We now define some auxiliary Legendrians $\Lambda_B, \Lambda_C \subset T^\infty \Sigma$.  Their role will be to make sure the co-orientations of $p^{-1}(\pi(\beta \cup \Lambda_B \cup \Lambda_C))$ at its intersections with $\partial P'$ are alternating, setting up the application of \Cref{prop:Thurston}.  The front projection of each component of $\Lambda_B$, $\Lambda_C$ will be an embedded loop, those of $\Lambda_C$ encircling the vertex or vertices of $C$ and those of $\Lambda_B$ passing through each of the $B_i$ but not meeting $p(\partial P)$. 

We choose $\Lambda_B$ so that the following properties hold:

\begin{enumerate}
\item $\pi(\Lambda_B)$ consists of $n-1$ pairwise disjoint embedded loops, and is disjoint from $C$.
\item  The intersection of each component of $p^{-1}(\pi(\Lambda_B))$ with $\partial P'$ consists of $2k$ points, one on each edge of the $B_i$ that also intersects $p^{-1}(\pi(\beta))$.
\item The co-orientations of $p^{-1}(\pi(\Lambda \cup \Lambda_B))$ at its intersections with $\partial B_i$ alternate along $\partial B_i$.
\item Inside a given $B_i$ the number of intersections of $p^{-1}(\pi(\Lambda_B))$ with $p^{-1}(\pi(\beta))$ is twice the number of crossings of $p^{-1}(\pi(\beta))$ in $B_i$ (the minimum possible number), and the intersection of  $p^{-1}(\pi(\Lambda \cup \Lambda_B))$ with $B_i$ is an alternating strand diagram.
\end{enumerate}

We choose $\Lambda_C$ so that the following properties hold:

\begin{enumerate}
\item If $g >0$, $\pi(\Lambda_C)$ consists of $n-1$ pairwise disjoint embedded loops lying in a contractible set containing the vertex of $C$, each isotopic to a small loop around the vertex of $C$ and intersecting each component of the smooth part of $C$ exactly twice.  If $g=0$ and $C$ is an embedded interval, $\pi(\Lambda_C)$ consists two sets of $n-1$ pairwise disjoint embedded loops, each set surrounding either end of the interval, and each loop intersecting the interval once.
\item $\pi(\Lambda_C)$ is disjoint from $p(\sqcup B_i)$
\item The co-orientations of $p^{-1}(\pi(\Lambda \cup \Lambda_B \cup \Lambda_C))$ at its intersections with $\partial P'$ alternate along $\partial P'$.
\item $\pi(\Lambda_C)$ is disjoint from $\pi(\Lambda_B)$, and the strands of $\pi(\Lambda_C)$ intersect each other and the strands of $\pi(\beta)$ the minimal number of times such that the above properties hold.
\end{enumerate}

In particular, though $\pi(\Lambda_B)$ and $\pi(\Lambda_C)$ intersect $\pi(\beta)$, $\Lambda_B$ and $\Lambda_C$ can be Legendrian isotoped through $T^\infty(\Sigma \smallsetminus \sigma) \smallsetminus \beta$ so that their front projections are disjoint from $\pi(\beta)$. 

The setup so far is illustrated in \Cref{fig:existence}.
We have arranged so that $p^{-1}(\pi(\beta \cup \Lambda_B \cup \Lambda_C)) \cap P'$ has alternating co-orientations along $\partial P'$ and no self-loops or parallel bigons, hence we may apply \Cref{prop:Thurston} to obtain an alternating Legendrian in $T^\infty P'$.  What we really want, however, is an alternating representative of $p^{-1}(\pi(\beta))$.  Thus the final step is to take connected sums between certain components of $\beta$ and $\Lambda_B$, $\Lambda_C$.  This will yield a new Legendrian $\beta' \subset T^\infty\Sigma$ which is Legendrian isotopic to $\beta$ but retains the desirable combinatorial properties of $\beta \cup \Lambda_B \cup \Lambda_C$.

First we consider the case where $\Sigma$ has positive genus.  The preimage $p^{-1}(\pi(\Lambda_C)) \cap P'$ is a union of embedded intervals, $n-1$ surrounding each corner of $P$. Each $p^{-1}(\pi(\beta_i)) \cap P'$ consists of three sets of $n$ parallel embedded intervals, of which one set has both endpoints on $A'$.  We form connected sums between the outer $n-1$ strands of $p^{-1}(\pi(\beta_1)) \cap P'$ having both endpoints on $A'$ and the strands of $p^{-1}(\pi(\Lambda_C)) \cap P'$ surrounding the corner of $A$ neareast to $\sigma_k$.  That is, we choose a path $\gamma$ in $P'$ between the outermost strands of each set, which only meets $p^{-1}(\pi(\beta \cup \Lambda_B \cup \Lambda_C))$ at the endpoints of $\gamma$.  We now cut both outermost strands at the endpoints of $\gamma$, reattaching them to each other by following $\gamma$ across $P'$.  We repeat this for the remaining $n-2$ strands, never increasing the total number of crossings.  The result is to replace $p^{-1}(\pi(\beta \cup \Lambda_B \cup \Lambda_C)) \cap P'$ with a new collection of immersed co-oriented curves with the same set of crossings but $n-1$ fewer smooth embedded components.

Next we perform a similar procedure with $\Lambda_B$.  The components of $p^{-1}(\pi(\Lambda_B)) \cap P'$ we use in the connected sum are those closest to the middle of $P$; that is, those connecting the edge of $B_1$ closest to one end of $A$ to the edge of $B_k$ closest to the other end of $A$.  We connect these to the strands of $p^{-1}(\pi(\Lambda_C)) \cap P'$ surrounding the corner of $A$ nearest to $\sigma_1$.  

Call $\beta' \subset T^\infty \Sigma$ the Legendrian lift of image of the resulting surgered front projection.  It follows from the construction that $\beta'$ is Legendrian isotopic to $\beta$: beforehand we could isotope each component of $\Lambda_B$, $\Lambda_C$ so that its front projection is an embedded loop disjoint from the front projection of $\beta$.  We could equivalently describe $\beta'$ by doing this isotopy, then taking a connected sum, then isotoping back, and a connected sum of a Legendrian with one whose front projection is a circle does not change its Legendrian isotopy class.  

On the other hand, the connected sum we performed in $P'$ did not create self-loops or parallel bigons (the step involving $\Lambda_B$ harmlessly creates $n-1$ antiparallel bigons, as does the step involving $\Lambda_C$ if $k=1$).  By construction the strands of $p^{-1}(\pi(\beta')) \cap P'$ have alternating co-orientations along the boundary of $P'$, hence we can apply \Cref{prop:Thurston} to find an alternating Legendrian in $T^\infty P'$ which is isotopic to the lift of $p^{-1}(\pi(\beta')) \cap P'$.  But $\beta'$ was already alternating above the image of each $B_i$, so we obtain an alternating Legendrian isotopy representative of $\beta'$, hence $\beta$, in $T^\infty \Sigma$.  Note that since we apply \Cref{prop:Thurston} in $P'$, and $\sigma$ is disjoint from the image of the interior of $P'$, the resulting isotopy from $\beta$ to an alternating representative takes place in $T^\infty( \Sigma \smallsetminus \sigma)$.  

In the genus zero case, the strategy is the same though we have to specify where to take connected sums differently.  We have assumed $k > 2$, so for any choice of $1 < i < k$ the $n$ components of $p^{-1}\pi(\beta_i) \cap P'$ whose endpoints lie on $A'$ do not intersect $p^{-1}(\pi(\Lambda_C))$. As before, we take a connected sum with the $n-1$ outermost of these components with those of $p^{-1}\pi(\Lambda_C)$ surrounding one corner of $P'$.  We have to now separately take a connected sum of the same components of $p^{-1}(\pi(\beta_i)) \cap P'$ with those of $p^{-1}(\pi(\Lambda_C))$ surrounding the other corner of $P'$ (since in the genus zero case $\Lambda_C$ has $2(n-1)$ components).  Finally, we take a connected sum of the same outermost $n-1$ components of $p^{-1}\pi(\beta_i) \cap P'$ with the components of $p^{-1}(\pi(\Lambda_B)) \cap P'$ closest to the middle of $P'$.  Again the resulting collection of immersed co-oriented curves has no self-loops or parallel bigons, so we may apply \Cref{prop:Thurston} as above (if we took $i=1$ or $i=k$ the above prescription would result in self-loops, so we have indeed used the assumption that $k > 2$). 
\end{proof}

\begin{figure}
\centering
\begin{tikzpicture}
\newcommand*{\ydist}{.9}; \newcommand*{\xdist}{.75};
\newcommand*{\sdist}{.5};\newcommand*{\hdist}{.4};
\node (l) [matrix] at (0,0) {
\foreach \x in {0,...,8} \foreach \y in {0,...,4} {\coordinate (\x\y) at (\x*\xdist,\y*\ydist);}
\foreach \c/\d in {00/80,02/82,04/84,22/24,40/42,62/64} {\draw[graphstyle] (\c) to (\d);}
\foreach \c in {00,02,04,80,82,84,22,40,62} {\fill (\c) circle (\bvertrad); \fill[white] (\c) circle (\wvertrad);}
\foreach \c in {24,42,64} {\fill (\c) circle (\bvertrad);}
\draw[asdstyle,lefthairs] ($(44)+(0,\sdist)$) to ($(24)+(0,\sdist)$) to[out=180,in=0] ($(04)-(0,\sdist)$) arc[start angle=-90,delta angle=-180,radius=\sdist] to[out=0,in=135] (14) to (41) to[out=-45,in=180] ($(50)+(0,\sdist)$) to ($(80)+(0,\sdist)$) arc[start angle=90, delta angle=-180,radius=\sdist] to ($(40)-(0,\sdist)$);
\draw[asdstyle,righthairs] ($(44)+(0,\sdist)$) to ($(64)+(0,\sdist)$) to[out=0,in=180-0] ($(84)-(0,\sdist)$) arc[start angle=-90,delta angle=180,radius=\sdist] to[out=180-0,in=180-135] (74) to (41) to[out=180+45,in=180-180] ($(30)+(0,\sdist)$) to ($(00)+(0,\sdist)$) arc[start angle=90, delta angle=180,radius=\sdist] to ($(40)-(0,\sdist)$);
\draw[asdstyle,righthairsnogap] ($(44)-(0,\sdist)$) to ($(34)-(0,\sdist)$) to[out=180,in=45] (23) to[out=225,in=0] ($(12)+(0,\sdist)$) to ($(02)+(0,\sdist)$) arc[start angle=90, delta angle=180,radius=\sdist] to ($(22)-(0,\sdist)$) to[out=0,in=180] ($(42)+(0,\sdist)$);
\draw[asdstyle,lefthairsnogap] ($(44)-(0,\sdist)$) to ($(54)-(0,\sdist)$) to[out=180-180,in=180-45] (63) to[out=180-225,in=180-0] ($(72)+(0,\sdist)$) to ($(82)+(0,\sdist)$) arc[start angle=90, delta angle=-180,radius=\sdist] to ($(62)-(0,\sdist)$) to[out=180-0,in=180-180] ($(42)+(0,\sdist)$);\\};

\node (r) [matrix] at (9,0) {
\newcommand*{\irad}{1.2}; \newcommand*{\mrad}{2.2}; \newcommand*{\morad}{2.5}; \newcommand*{\orad}{3.2};
\draw[asdstyle,righthairsnogap] (90:\morad) to[out=180,in=150-90] (150:\mrad) to[out=240,in=120] (210:\morad) to[out=300,in=180] (-90:\mrad);
\draw[asdstyle,lefthairsnogap] (90:\morad) to[out=180-180,in=180-150+90] (180-150:\mrad) to[out=180-240,in=180-120] (180-210:\morad) to[out=180-300,in=180-180] (180+90:\mrad);
\draw[asdstyle,lefthairsnogap] (-90:\irad) to[out=180,in=150+90] (150:\orad) to[out=150-90,in=120] (30:\irad) to[out=-60,in=0] (-90:\orad);
\draw[asdstyle,righthairsnogap] (-90:\irad) to[out=180-180,in=180-150-90] (180-150:\orad) to[out=180-150+90,in=180-120] (180-30:\irad) to[out=180+60,in=180-0] (-90:\orad);\\};
\coordinate (rw) at ($(r.west)+(5mm,0)$); \coordinate (le) at ($(l.east)+(1mm,0)$);
\draw[arrowstyle] (le) -- (rw);
\draw[arrhdstyle] ($(rw)+(-1.2mm,1.2mm)$) -- (rw); \draw[arrhdstyle] ($(rw)+(-1.2mm,-1.2mm)$) -- (rw);
\end{tikzpicture}
\caption{\Cref{con:one} associates the bicolored graph on the left to the word $s_2 s_1 s_2$ for $\beta' = \Delta$. This produces an alternating representative of $\beta = \Delta^3$ on the right.}
\label{fig:conone}
\end{figure}
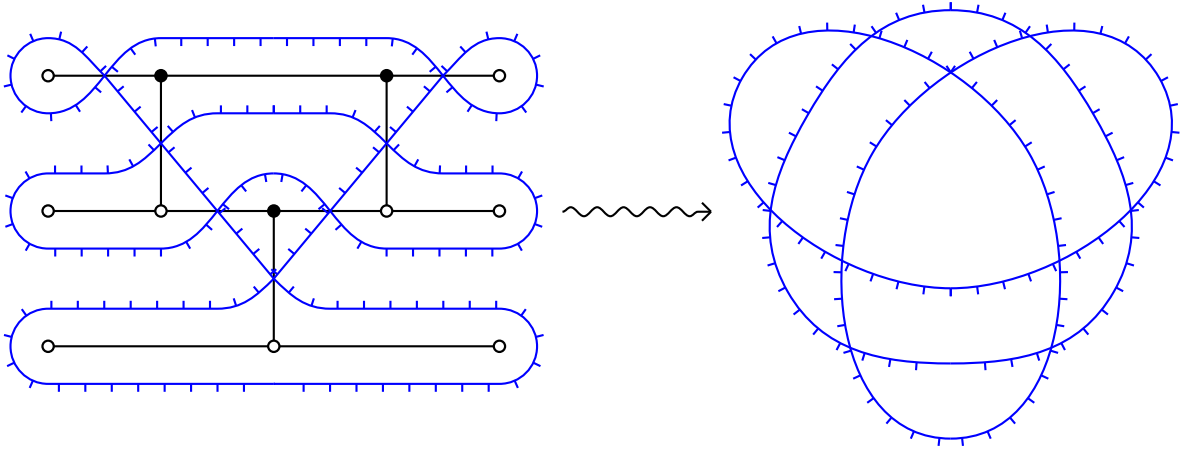

We now describe several constructions of alternating representatives for special cases of the Legendrian satellites appearing in \Cref{thm:wildcharvty}. The discussion largely amounts to reinterpreting well-known constructions in combinatorics into the language of Legendrian knot theory. We write words for the annular braids $\beta_i$ in letters $s_1,\dotsc, s_{n-1}$, with the convention that $s_1$ is a crossing of the strands furthest from $\sigma_i$.  We use $\Delta$ to denote the positive half-twist.  The first two constructions are Legendrian reinterpretations of the notion of double wiring diagram introduced in \cite{FZ1} and considered for general braids in \cite{FG4}:

\begin{construction}\label{con:one}
Let $\Sigma = S^2$, $\sigma = \{\infty\}$, and $\beta$ a positive annular braid on $n$ strands of the form $\beta' \Delta^2$.  A word $\beta' = s_{i_1}\cdots s_{i_k}$ determines an alternating isotopy representative of $\beta$ as follows.  Begin with a bicolored graph in the plane consisting of $n$ horizontal line segments running from $(0,i)$ to $(k+1,i)$ for $1 \leq i \leq n$, with white vertices at both ends. For $1 \leq j \leq k$ adjoin a vertical segment along the line $y=j$ connecting the line $x = i_j$ to the line $x = 1+i_j$, with a black vertex at its top and a white vertex at its bottom.  From the resulting alternating strand diagram, one obtains the front projection of $\beta$ by sliding all upward co-oriented strands to the top of the picture and all downward co-oriented strands to the bottom; see \Cref{fig:conone}.
\end{construction}

The above assumption that $\beta$ be of the form $\beta' \Delta^2$ is present for good reason.  For example, if $\beta$ has no crossings at all then $\cM_1(S^2,\beta,\infty)$ is a single point whose stabilizer is a Borel subgroup of $GL_n$.  But if $\beta$ had any smooth exact fillings at all, let alone one arising from an alternating representative, the moduli space would necessarily have a point with an abelian stabilizer.  We do not know whether the requirement $\beta = \beta' \Delta^2$ is a necessary condition for $\beta$ to have alternating representatives. 

We note that when $s_{i_1}\cdots s_{i_k}$ is a reduced word for an element $w$ of the symmetric group, a suitably-framed moduli space $\cM^{fr}_{\:1}(S^2,\beta,\infty)$ recovers the double Bruhat cell $B_+ \cap B_- w B_-$ \cite{FZ1} as well as a certain positroid stratum of $\Gr(n,2n)$.

\begin{construction}\label{con:two}
Let $\Sigma = S^2$, $\sigma = \{0, \infty\}$, and $\beta_0$, $\beta_\infty$ any positive annular braids at $0$ and $\infty$.  A double word for $(\beta_0,\beta_\infty)$ is a shuffle of words for $\beta_0$ and $\beta_\infty$ \cite{FZ}.  We encode a double word as a sequence $(s_{i_1},\dotsc,s_{i_k})$ and a function $\tau: \{1,\dotsc,k\} \to \{0,\infty\}$ such that the ordered product $\prod_{\tau(j) = \ell} s_j$ is a word for $\beta_\ell$.  A double word determines an alternating isotopy representative as follows.  Begin with a bicolored graph in the punctured plane consisting of $n$ concentric circles centered at the origin, with white vertices where each intersects the positive $x$-axis.  We adjoin a radial line segment of phase $j2\pi i/(k+1)$ for each $1 \leq j \leq k$. 
This segment connects the $i_j$th and $(i_j+1)$th circles closest to $\tau(j)$, and has white/black vertices at its farthest/closest endpoint to $\tau(j)$, respectively.  From the resulting alternating strand diagram, one obtains the front projection of $\beta$ by sliding the strands co-oriented towards $0$, $\infty$ past each other towards $0$, $\infty$, respectively.
\end{construction}

These two constructions cover the remaining cases of Theorem \ref{thm:wildcharvty}.

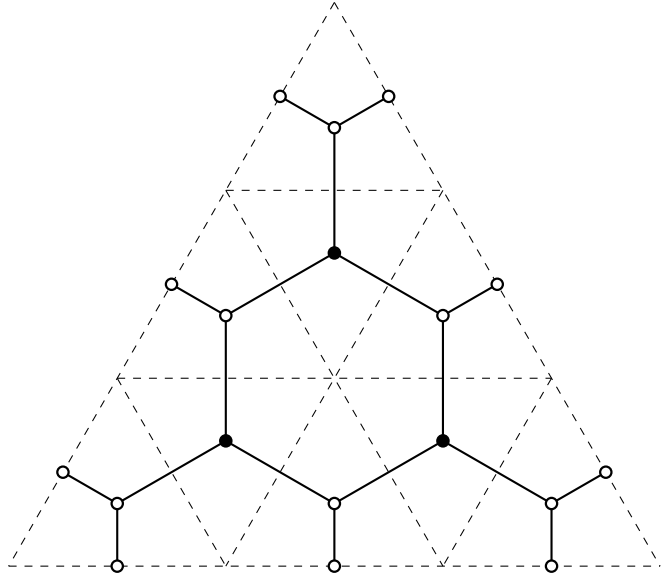
\begin{figure}
\centering
\begin{tikzpicture}
\newcommand*{\rad}{5};
\foreach \c/\angle in {300/90,030/210,003/-30} {\coordinate (\c) at (\angle:\rad);}
\foreach \c/\d/\e in {210/300/030,120/030/300,201/300/003,102/003/300,021/030/003,012/003/030} {\coordinate (\c) at ($(\d)!1/3!(\e)$);}
\coordinate (111) at (0,0);
\foreach \n/\c/\d/\e in {200/300/210/201,020/030/120/021,002/003/102/012,110/210/120/111,101/201/102/111,011/012/021/111} {\coordinate (w\n) at ($1/3*(\c)+1/3*(\d)+1/3*(\e)$);}
\foreach \n/\c/\d/\e in {100/210/201/111,010/120/021/111,001/102/012/111} {\coordinate (b\n) at ($1/3*(\c)+1/3*(\d)+1/3*(\e)$);}
\foreach \n/\c/\d in {310/300/210,220/210/120,130/120/030,301/300/201,202/201/102,103/102/003,031/030/021,022/021/012,013/012/003} {\coordinate (w\n) at ($(\c)!1/2!(\d)$);}
\foreach \c/\d in {300/030,201/021,102/012,300/003,210/012,120/021,030/003,120/102,210/201} {\draw[dashed] (\c) to (\d);}
\foreach \c/\d in {b100/w200,b100/w101,b100/w110,b010/w020,b010/w110,b010/w011,b001/w002,b001/w101,b001/w011,w200/w310,w200/w301,w020/w130,w020/w031,w002/w103,w002/w013,w110/w220,w101/w202,w011/w022} {\draw[graphstyle] (\c) to (\d);}
\foreach \c in {w310,w220,w130,w301,w202,w103,w031,w022,w013,w200,w020,w002,w110,w101,w011} {\fill (\c) circle (\bvertrad); \fill[white] (\c) circle (\wvertrad);}
\foreach \c in {b100,b010,b001} {\fill (\c) circle (\bvertrad);}
\end{tikzpicture}
\caption{The graph $\Gamma_3$ of \Cref{con:triangle}.  If we regard it as a freestanding graph in $\R^2$, rather than attaching it to other copies of itself following some ideal triangulation, it defines an alternating Legendrian isotopic to the Legendrian of \Cref{fig:conone}.}
\label{fig:triangle}
\end{figure}

The combinatorics essential to the next construction is due to \cite{FG1}. The cluster algebras associated to these examples when $n=2$ are often just called cluster algebras from surfaces \cite{FST}. In the literature one often starts with the surface $\Sigma'$ with marked boundary appearing in the construction, but from our point of view this is simply a convenient way of encoding the number of half-twists.

\begin{construction}\label{con:triangle}
Let $\Sigma$, $\sigma$ be arbitrary and each $\beta_i$ of the form $\Delta^{k_i}$ for some $k_i \in \N$.  Alternating representatives of $\beta \subset T^\infty \Sigma$ can be constructed using triangulations.  First cut out a disk $D_i$ around each $\sigma_i$ with $k_i > 0$ and let $\Sigma' \subset \Sigma$ denote the resulting surface with boundary. For each such $\sigma_i$ we also mark $k_i$ points points on the component of $\partial \Sigma'$ that surrounds it.  We say an ideal triangulation of $\Sigma'$ is a triangulation such that all triangles have vertices either on marked points of boundary components or on points $\sigma_i$ for which $k_i = 0$.

There is a standard bicolored graph $\Gamma_n$ we can embed in any triangle \cite{FG1}. This is dual to a triangulation of the given triangle into $n^2$ smaller triangles as in \Cref{fig:triangle}. If we are given an equilateral triangle in $\R^2$, we cut it by $n-1$ equally-spaced lines parallel to each of its three sides.  We label the triangles of the resulting triangulation as white or black so that every triangle on the boundary of the original one is white, and no triangles of the same color share an edge. The graph $\Gamma_n$ has a black/white vertex in the center of each black/white triangle.  It also has $n$ white vertices along each edge of the original triangle, one in the middle of each outward-facing edge of a white triangle. There is an edge between any black vertex and each of its three white neighbors, as well as between each white vertex on the boundary and the white vertex in the center of the white triangle whose boundary it lies on.

We associate a bicolored graph $\Gamma \subset \Sigma$ to an ideal triangulation of $\Sigma'$ by embedding $\Gamma_n$ into each triangle.  This is done so that the white vertices on the boundaries of adjacent triangles coincide.  We can choose the alterating Legendrian of $\Gamma$ so that the $\sigma_i$ lie in null regions of its front projections. After isotoping this Legendrian to its standard (i.e. positroid) form in each triangle separately, one easily sees it is Legendrian isotopic to $\beta$ inside $T^\infty(\Sigma \smallsetminus \sigma)$.
\end{construction}
\subsection{The square move}
\label{sec:squaremove}

There is a local operation on quadrilateral faces of bicolored graphs --- the so-called
{\em square move} --- that induces isotopies between alternating 
Legendrians.  The corresponding conjugate Lagrangians each determine an abelian chart on the moduli space; we will show here
that these charts are related by a cluster $\cX$-transformation.  

Geometrically, the associated conjugate Lagrangians differ by a certain Lagrangian
surgery (see  \Cref{fig:surgery}); 
one could imagine using this fact directly to compare,  Floer theoretically, the
spaces of local systems supported on each.  Instead, we use the results of \Cref{sec:altsh}, which capture all the relevant Floer theoretic data in the categories of 
alternating sheaves associated to the two alternating Legendrians.  The comparison between
these categories is computed using local calculations of the \cite{GKS} equivalence.  

The local model for the square move is the Legendrian isotopy $\Lambda \to \Lambda'$ pictured in \Cref{fig:squaremove}. 
Let $\cF \in \Sh_\Lambda(D^2)$ be an alternating sheaf and $N$, $W$, $S$, $E$ its nonzero stalks near the boundary of the picture.
We will compute the image of $\cF$ under $\Sh_{\Lambda} \congto \Sh_{\Lambda'}$ in terms of the positive face coordinates of \Cref{def:coords}. 
Here we must consider these not just for closed faces of the graph, but also for the four regions on the boundary. The associated coordinates are more properly isomorphisms
\[
X_{NE}: N \congto E, \quad X_{ES}: E \congto S, \quad X_{SW}: S \congto W, \quad X_{WN}: W \congto N,
\]
which together with the positive coordinate of the middle region satisfy
\[ X_M = -(X_{NE}X_{ES}X_{SW} X_{WN})^{-1}. \]
For example, $X_{NE}$ is explicitly the composition of
\begin{enumerate}
\item the isomorphism of $N$ with a microstalk of $\cF$ on the component of $\Lambda$ passing immediately below the Northern region
\item parallel transport in $\cF_\Lambda$ to a microstalk of $\cF$ at the far right of the picture
\item the isomorphism between this microstalk and $E$.
\end{enumerate}
The remaining isomorphisms $X_{ES}$, $X_{SW}$, $X_{WN}$ can be described symmetrically. Note that the existence of the isomorphisms in (1) and (3) depends on the vanishing of $\cF$ in the middle and Northeast null regions. Given an alternating sheaf microsupported on $\Lambda'$, we similarly denote its positive coordinates by $Y_{NE}$, $Y_{ES}$, $Y_{SW}$, $Y_{WN}$, and $Y_M$.

\begin{proposition}\label{prop:localclustertrans}
Let $\Lambda, \Lambda' \subset T^\infty D^2$ be the alternating Legendrians related by the square move of \Cref{fig:squaremove}.  Let $\cF \in \Sh_\Lambda(D^2)$ be an alternating sheaf and $\cF' \in \Sh_{\Lambda'}(D^2)$ its image under the isotopy equivalence $\Sh_{\Lambda} \congto \Sh_{\Lambda'}$.  Then $\cF$ is an alternating sheaf if and only if  $X_M \neq -1$, and its positive face coordinates are related to those of $\cF$ by
\begin{gather*}
X_{NE} = Y_{NE}(1+Y_M), \quad X_{ES} = Y_{ES}(1+Y_M^{-1})^{-1}, \quad X_{SW} = Y_{SW}(1+Y_M),\\ X_{WN} = Y_{WN}(1+Y_M^{-1})^{-1}, \quad X_M = Y_M^{-1}.
\end{gather*}
\end{proposition}
\begin{proof}
Denote by $\cF_\#$ the image of $\cF$ under $\Sh_{\Lambda} \congto \Sh_{\Lambda_\#}$, where $\Lambda_\#$ is as in the middle of \Cref{fig:squaremove}.  $\cF_\#$ is determined by the data of a generic stalk $V$ in the middle region, along with the four generization maps from $S$, $W$, $N$, and $E$.

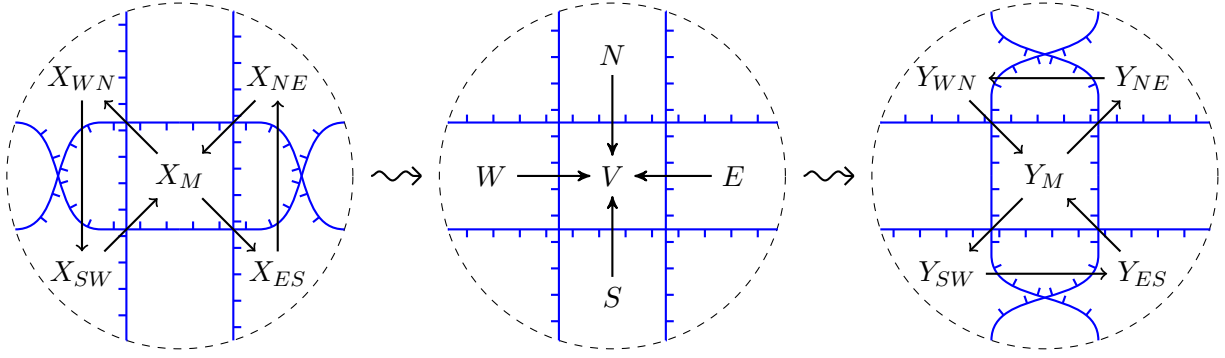
\begin{figure}
\begin{tikzpicture}
\newcommand*{\off}{18}; \newcommand*{\rad}{2.3}; \newcommand*{\vrad}{1.0};
\newcommand*{\stalkdist}{.7}; \newcommand*{\qvertdist}{.8}; \newcommand{\qoff}{45};
\node (l) [matrix] at (0,0) {
\coordinate (bl) at (180:\rad/2); \coordinate (br) at (0:\rad/2);
\coordinate (wt) at (90:\rad*.8); \coordinate (wb) at (-90:\rad*.8);
\coordinate (wr) at (0:\rad*.9); \coordinate (wl) at (180:\rad*.9);
\coordinate (ur) at ($(-\off:\rad)!\ifac!(-180+\off:\rad)$); \coordinate (lr) at ($(\off-180:\rad)!\ifac!(-\off:\rad)$);
\coordinate (ul) at ($(\off:\rad)!\ifac!(-180-\off:\rad)$); \coordinate (ll) at ($(-\off-180:\rad)!\ifac!(\off:\rad)$);
\coordinate (mr) at ($(ur)!.5!(lr)$); \coordinate (ml) at ($(ul)!.5!(ll)$);
\draw[dashed] (0,0) circle (\rad);
\draw[asdstyle,righthairs] (90+\off:\rad) -- (-90-\off:\rad);
\draw[asdstyle,righthairs] (-90+\off:\rad) -- (90-\off:\rad);
\draw[asdstyle,lefthairs] (mr) to (ur) [out=0,in=180] to (\off:\rad);
\draw[asdstyle,righthairs] (mr) to (lr) to[in=0,out=180] (180-\off:\rad);
\draw[asdstyle,righthairs] (ml) to (ul) [out=0,in=180] to (-\off:\rad);
\draw[asdstyle,lefthairs] (ml) to (ll) to[in=0,out=180] (180+\off:\rad);
\node (ne) at (90-\qoff:\qvertdist*\rad) {$X_{NE}$}; \node (wn) at (90+\qoff:\qvertdist*\rad) {$X_{WN}$};
\node (es) at (-90+\qoff:\qvertdist*\rad) {$X_{ES}$}; \node (sw) at (-90-\qoff:\qvertdist*\rad) {$X_{SW}$};
\node (m) at (0,0) {$X_{M}$};
\foreach \c/\d in {ne/m,sw/m,m/wn,m/es,wn/sw,es/ne} {\draw[dualquiverstyle] (\c) to (\d);}\\};

\node (mid) [matrix] at (2.5*\rad,0) {
\coordinate (ln) at ($(\off:\rad)!.5!(180-\off:\rad)$);
\coordinate (ls) at ($(-\off:\rad)!.5!(180+\off:\rad)$);
\coordinate (le) at ($(-90+\off:\rad)!.5!(90-\off:\rad)$);
\coordinate (lw) at ($(-90-\off:\rad)!.5!(90+\off:\rad)$);
\draw[dashed] (0,0) circle (\rad);
\draw[asdstyle,lefthairs] (lw) to[out=90,in=-90] (90+\off:\rad);
\draw[asdstyle,righthairs] (lw) to[out=-90,in=90]  (-90-\off:\rad);
\draw[asdstyle,righthairs] (le) to[out=90,in=-90]  (90-\off:\rad);
\draw[asdstyle,lefthairs] (le) to[out=-90,in=90]  (-90+\off:\rad);
\draw[asdstyle,lefthairs] (ln) to[out=0,in=180] (\off:\rad);
\draw[asdstyle,righthairs] (ln) to[out=180,in=0] (180-\off:\rad);
\draw[asdstyle,righthairs] (ls) to[out=0,in=180] (-\off:\rad);
\draw[asdstyle,lefthairs] (ls) to[out=180,in=0] (180+\off:\rad);
\node (n) at (90:\stalkdist*\rad) {$N$}; \node (s) at (-90:\stalkdist*\rad) {$S$};
\node (e) at (0:\stalkdist*\rad) {$E$}; \node (w) at (180:\stalkdist*\rad) {$W$};
\node (v) at (0,0) {$V$};
\foreach \c in {n,s,e,w} {\draw[genmapstyle] (\c) to (v);}\\};

\node (r) [matrix] at (5*\rad,0) {
\coordinate (bl) at (180+90:\rad/2); \coordinate (br) at (0+90:\rad/2);
\coordinate (wt) at (90+90:\rad*.8); \coordinate (wb) at (-90+90:\rad*.8);
\coordinate (wr) at (0+90:\rad*.9); \coordinate (wl) at (180+90:\rad*.9);
\coordinate (ur) at ($(-\off+90:\rad)!\ifac!(-90+\off:\rad)$); \coordinate (lr) at ($(\off-90:\rad)!\ifac!(90-\off:\rad)$);
\coordinate (ul) at ($(\off+90:\rad)!\ifac!(-90-\off:\rad)$); \coordinate (ll) at ($(-\off-90:\rad)!\ifac!(90+\off:\rad)$);
\coordinate (mr) at ($(ur)!.5!(lr)$); \coordinate (ml) at ($(ul)!.5!(ll)$);
\draw[dashed] (0,0) circle (\rad);
\draw[asdstyle,righthairs] (90+90+\off:\rad) -- (-90+90-\off:\rad);
\draw[asdstyle,righthairs] (-90+90+\off:\rad) -- (90+90-\off:\rad);
\draw[asdstyle,lefthairs] (mr) to (ur) [out=0+90,in=180+90] to (\off+90:\rad);
\draw[asdstyle,righthairs] (mr) to (lr) to[in=0+90,out=180+90] (180+90-\off:\rad);
\draw[asdstyle,righthairs] (ml) to (ul) [out=0+90,in=180+90] to (-\off+90:\rad);
\draw[asdstyle,lefthairs] (ml) to (ll) to[in=0+90,out=180+90] (180+90+\off:\rad);
\node (ne) at (90+90-\qoff:\qvertdist*\rad) {$Y_{WN}$}; \node (wn) at (90+90+\qoff:\qvertdist*\rad) {$Y_{SW}$};
\node (es) at (-90+90+\qoff:\qvertdist*\rad) {$Y_{NE}$}; \node (sw) at (-90+90-\qoff:\qvertdist*\rad) {$Y_{ES}$};
\node (m) at (0,0) {$Y_{M}$};
\foreach \c/\d in {ne/m,sw/m,m/wn,m/es,wn/sw,es/ne} {\draw[dualquiverstyle] (\c) to (\d);}\\};

\coordinate (rw) at ($(r.west)+(-1mm,0)$); \coordinate (me) at ($(mid.east)+(1mm,0)$);
\coordinate (mw) at ($(mid.west)+(-1mm,0)$); \coordinate (le) at ($(l.east)+(1mm,0)$);
\foreach \ca/\cb in {le/mw,me/rw} {
\draw[arrowstyle] (\ca) -- (\cb);
\draw[arrhdstyle] ($(\cb)+(-1.2mm,1.2mm)$) -- (\cb); \draw[arrhdstyle] ($(\cb)+(-1.2mm,-1.2mm)$) -- (\cb);}
\end{tikzpicture}
\caption{The alternating Legendrian $\Lambda'$ on the right is obtained from $\Lambda$ on the left by a square move.  The isotopy between them can be chosen to pass through $\Lambda_\#$ pictured in the middle.  If $\cF \in \Sh_\Lambda(D^2)$ is alternating, its image $\cF_\#$ in $\Sh_{\Lambda_\#}(D^2)$ is described by the stalks and generization maps pictured.  The dual quivers of the bicolored graphs of $\Lambda$, $\Lambda'$ have vertices labeling the positive face coordinates on their spaces of alternating sheaves.}
\label{fig:squaremove}
\end{figure}

We claim that in terms of $\cF_\#$, $X_{NE}$ is the composition of
\begin{enumerate}
\item the generization map from $N$ to $V$
\item the map from $V$ to the microstalk of $\cF_\#$ on the component of $\Lambda_\#$ passing immediately below the middle region
\item parallel transport in $(\cF_\#)_{\Lambda_\#}$ to a microstalk of $\cF_\#$ at the far right of the picture
\item the isomorphism between this microstalk and $E$.
\end{enumerate}
Note that under the isotopy $\Lambda \to \Lambda_\#$ the component of $\Lambda_\#$ passing below the middle region corresponds to the component of $\Lambda$ passing below the Northern region.  We obtain $\Lambda_\#$ from $\Lambda$ by performing Reidemesiter III moves at the black vertices of its bipartite graph followed by a Reidemeister II. The above description of $X_{NE}$ follows from \Cref{lem:R3}, which asserts the invariance of microlocal parallel transport under Reidemeister III, and the following observation about Reidemeister II.  In both sides of \Cref{fig:reid} there is a map from $A$ to $Cone(B \to C)$: on the left the composition $A \to C \to Cone(B \to C)$ and on the right $A \to Cone(C' \to A) \to Cone(B \to C)$.  Here $Cone(C' \to A) \to Cone(B \to C)$ is the canonical isomorphism coming from $C' \cong Cone(A \oplus B \to C)[-1]$.  These two maps agree by elementary homological algebra, and are exactly the maps being compared at the beginning of the two descriptions of $X_{NE}$.

More explicitly, the above description identifies $X_{NE}$ with the natural composition $N \congto V/S \congto E$.  Note that $N \to V/S$ being invertible is exactly the condition that $\cF_\#$ arose from an alternating sheaf on $\Lambda$.  It also follows that $N$, $E$, $S$, $W$ are all isomorphic to a single invertible $\coeffs$-module $I$ (if $\cF$ were not alternating $\cF_\#$ could have $N \cong S \cong I_1$ and $W \cong E \cong I_2$ for nonisomorphic invertible modules $I_1$, $I_2$).  Fixing identifications of each with $I$ and of $V$ with $I^2$, we encode the generization maps as a $2 \times 4$ matrix with columns labeled $S$, $W$, $N$, $E$ and entries in $\coeffs = \End_\coeffs I$.  We write $\Delta_{NE}$, etc., for the minors of this matrix, keeping track of orders of indices so, for example, $\Delta_{NE} = - \Delta_{EN}$.   The crossing conditions imply that minors of cyclically consecutive columns are invertible, and in terms of minors we can rewrite the above calculation as $X_{NE} = \Delta_{NS} / \Delta_{ES}$.  By symmetry, we also have: 
\begin{gather*}
X_{NE} = \frac{\Delta_{SN}}{\Delta_{SE}}, \quad X_{ES} = \frac{\Delta_{NE}}{\Delta_{NS}}, \quad X_{SW}=\frac{\Delta_{NS}}{\Delta_{NW}}, \quad X_{WN} =\frac{\Delta_{SW}}{\Delta_{SN}} \\
Y_{NE} = \frac{\Delta_{WN}}{\Delta_{WE}}, \quad Y_{ES} = \frac{\Delta_{WE}}{\Delta_{WS}}, \quad Y_{SW}=\frac{\Delta_{ES}}{\Delta_{EW}}, \quad Y_{WN} =\frac{\Delta_{EW}}{\Delta_{EN}}
\end{gather*}
The remaining holonomy $X_M$ is  determined by the relation $X_M^{-1} = -X_{WN}X_{NE}X_{ES}X_{SW}$, likewise for $Y_M$.

Recall the 2-term Pl\"{u}cker relation
\[
\Delta_{SN} \Delta_{EW} = \Delta_{SE}\Delta_{NW} + \Delta_{SW}\Delta_{EN}
\]
Dividing by $\Delta_{ES}\Delta_{WE}$ and reordering indices, we obtain the desired relation
\[
X_{NE} = \frac{\Delta_{SN}}{\Delta_{SE}} = \frac{\Delta_{WN}}{\Delta_{WE}}(1-\frac{\Delta_{WS}\Delta_{EN}}{\Delta_{ES}\Delta_{WN}}) = Y_{NE}(1+Y_M).
\]
The remaining relations follow from a symmetric calculation.
\end{proof}

By locality, the preceding result also determines how alternating sheaves microsupported on more complex alternating Legendrians transform under square moves. The more general transformation rules are naturally expressed in the language of cluster algebra \cite{Fom,Lec,Kel2}, which we briefly review following the notation of \cite{FG2,GHK}. Below we write $[a]_+$ for $\mathrm{max}(a,0)$.

\begin{definition}
A \textbf{seed} $s = (N,\{e_i\})$ is the data of a lattice $N$ with skew-symmetric integral form $\{,\}$ and a finite collection $\{e_i\}_{i \in I} \subset N$ of distinct primitive elements indexed by a set $I$. 
The mutation of $s$ at $k \in I$ is the seed $\mu_ks = (N,\{\mu_ke_i\})$, where 
\begin{equation}\label{eq:seedmutation}
\mu_ke_i = \begin{cases} e_i + [\{e_i,e_k\}]_+ e_k & i \neq k \\ -e_k & i = k. \end{cases}
\end{equation}
\end{definition}

To a seed we associate a quiver without oriented 2-cycles and with vertex set $\{v_i\}_{i \in I}$.  The number of arrows from $v_i$ to $v_j$ is $[\{e_i,e_j\}]_+$, and if the $e_i$ are a basis the seed is determined up to isomorphism by the quiver. 
Conversely, given such a quiver $Q$ we have a seed given by $\Z Q_0$ with its natural basis and skew-symmetric form; in the literature one often only considers seeds of this form. One can also consider seeds related to skew-symmetrizable matrices, but these do not arise in our setting. We also suppress a discussion of frozen indices, obviated in our case by allowing the $e_i$ to fail to generate $N$.

Given a seed $s = (N,\{e_i\})$, we write  $M = \Hom(N,\Z)$ and consider the dual algebraic tori
\[
\cX_s = \Spec \Z N, \quad \cA_s = \Spec \Z M,
\]
We let $z^n \in \Z N$ denote the monomial associated to $n \in N$, likewise $z^m \in \Z M$ for $m \in M$. 

\begin{definition}
For $k \in I$, the \textbf{cluster $\cX$- and $\cA$-transformations} $\mu_k: \cX_s \dasharrow \cX_{\mu_k s}$, $\mu_k: \cA_s \dasharrow \cA_{\mu_k s}$ are the rational maps defined by
\begin{equation}\label{eq:clustertrans}
\mu_k^* z^n = z^n(1+z^{e_k})^{\{e_k,n\}}, \quad \mu_k^* z^m = z^m(1+z^{\{e_k,-\}})^{-\langle e_k, m\rangle},
\end{equation}
where $\langle e_k, m \rangle$ denotes the evaluation pairing.
We use the term \textbf{signed cluster transformations} to refer to the counterparts of these maps where the plus signs are replaced by minus signs.
\end{definition}

Let $T$ be an infinite $|I|$-ary tree with edges labeled by $I$ so that the edges incident to a given vertex have distinct labels.  Fix a root $t_0 \in T_0$ and label it by the seed $s$.  Label the remaining $t\in T_0$ by seeds $s_t$ such that if $t$ and $t'$ are connected by an edge labeled $k$, and $t'$ is farther from $t_0$ than $t$, then $s_{t'} = \mu_k s_t$.

\begin{definition}\label{def:clusterstructure}
A \textbf{cluster $\cX$-structure} on $Y$ is a collection $\{\cX_{s_t} \into Y\}_{t \in T_0}$ of open maps such that the images of $\cX_{s_t}$ and $\cX_{\mu_k s_t}$ are related by a cluster $\cX$-transformation for all $t$, $k$. A partial cluster $\cX$-structure is the same but with maps only for a subset of $T_0$, a cluster $\cA$-structure the same but with $\cA$-tori and $\cA$-transformations, and a signed cluster structure the same but with signed cluster transformations.
\end{definition}

When the $e_i$ are linearly independent, the notions of signed and ordinary cluster $\cX$-structure coincide: given a homomorphism $\sigma: N \to \{\pm 1\}$ such that $\sigma(e_i) = -1$ for all $i$, the automorphism $z^n \mapsto \sigma(n)z^n$ intertwines the signed and ordinary cluster transformations.

There is a canonical seed associated to an embedded bicolored graph $\Gamma \subset \Sigma$ and a collection of marked points $\sigma \subset \Sigma$. Let $\{\partial F_i\} \subset H_1(\Gamma;\Z)$ be the set of boundaries of faces $F_i$ not meeting $\sigma$, where by faces we mean the contractible regions of $\Sigma \smallsetminus \Gamma$.  If $L$ is the conjugate Lagrangian of $\Gamma$, we have $H_1(\Gamma;\Z) \cong H_1(L;\Z)$, so the intersection pairing makes $(H_1(L;\Z),\{\partial F_i\})$ a seed.

The quiver of $(H_1(L;\Z),\{\partial F_i\})$ has vertices labeled by $\{\partial F_i\}$ and $\langle e_i, e_j \rangle_+$ arrows from $e_i$ to $e_j$.  
It can be drawn on $\Sigma$ as follows: the vertex labeled by $\partial F_i$ is drawn in $F_i$, and an edge of $\Gamma$ with distinctly colored endpoints and separating two faces is crossed by an arrow with the white endpoint on its right. This is pictured in \Cref{fig:squaremove}.  More precisely, the drawn quiver may have oriented 2-cycles, but removing these one obtains the quiver of $(H_1(L;\Z),\{\partial F_i\})$. Comparing the left picture of Figure \ref{fig:redexample} and its rotation gives the following.

\begin{proposition}\label{prop:GK} 
\cite[Sec. 4.1]{GK} Let $\Gamma \subset \Sigma$ be an embedded bicolored graph and $\Gamma'$ the result of performing a square move at a face $F_k$.  There is a homeomorphism of their conjugate Lagrangians $L$, $L'$ which identifies the seed $(H_1(L';\Z), \{\partial F'_i \})$ with the one obtained from $(H_1(L;\Z), \{\partial F_i \})$ by mutation at $\partial F_k$.
\end{proposition}

With this in hand we can state:

\begin{theorem} \label{thm:clustertrans}
Let $L$ be the conjugate Lagrangian of an alternating Legendrian $\Lambda \subset T^\infty \Sigma$, and $L'$, $\Lambda'$ their counterparts upon performing a square move at a face $F_k$ not meeting $\sigma$. We identify the underlying topological spaces of $L$ and $L'$ as in \Cref{prop:GK}, and identify the spaces of alternating sheaves with $\Loc_1(L)$  as in \Cref{def:coords}. 
Under the isomorphism of moduli spaces induced by the Legendrian isotopy $\Lambda \to \Lambda'$, the inclusions
$$\Loc_1(L) \hookrightarrow \cM_1(\Lambda,\sigma) \simeq \cM_1(\Lambda',\sigma) \hookleftarrow \Loc_1(L)$$
are related by the signed cluster $\cX$-transformation associated to the mutation of $(H_1(L;\Z),
\{\partial F_i \})$ at $\partial F_k$.
\end{theorem}
\begin{proof}
Since the isotopy $\Lambda \to \Lambda'$ is stationary outside a neighborhood of the face $F$, the equivalence $\Sh_\Lambda(\Sigma) \cong \Sh_{\Lambda'}(\Sigma)$  restricts to the identity outside such a neighborhood.  On the other hand, in a neighborhood of $F_k$ it restricts to the equivalence explicitly computed in \Cref{prop:localclustertrans}. Since positive and standard face coordinates differ by a sign, the formulas computed there are exactly those expressing the signed cluster $\cX$-transformation associated to mutation of $(H_1(L;\Z),\{\partial F_i\})$ at $\partial F_k$.
\end{proof}

The result extends to framed moduli spaces in an obvious way. Strictly speaking, in the unframed case the points of  $\Loc_1(L)$  have $\G_m$ stabilizers, but we use the usual cluster terminology regardless.

\begin{corollary}
Suppose $\Lambda \in T^\infty \Sigma$ is an alternating Legendrian. Then $\cM_1(\Lambda,\sigma)$ has a partial, signed cluster $\cX$-structure with charts labeled by alternating Legendrians obtained by some series of square moves from $\Lambda$. If $\sigma$ is nonempty it has an ordinary partial cluster $\cX$-structure.
\end{corollary}

\begin{proof}
The first part of the corollary is immediate, the second follows from the remark after \Cref{def:clusterstructure} since the $\partial F_k$ are independent in homology when $\sigma$ is nonempty.
\end{proof}

\begin{figure}
\centering
\begin{tikzpicture}
\newcommand*{\off}{18};\newcommand*{\rad}{2.3};\newcommand*{\vrad}{1.0};
\node (l) [matrix] at (0,0) {
\coordinate (bl) at (180:\rad/2); \coordinate (br) at (0:\rad/2);
\coordinate (wt) at (90:\rad*.8); \coordinate (wb) at (-90:\rad*.8);
\coordinate (wr) at (0:\rad*.9); \coordinate (wl) at (180:\rad*.9);
\coordinate (ur) at ($(-\off:\rad)!\ifac!(-180+\off:\rad)$); \coordinate (lr) at ($(\off-180:\rad)!\ifac!(-\off:\rad)$);
\coordinate (ul) at ($(\off:\rad)!\ifac!(-180-\off:\rad)$); \coordinate (ll) at ($(-\off-180:\rad)!\ifac!(\off:\rad)$);
\coordinate (mr) at ($(ur)!.5!(lr)$); \coordinate (ml) at ($(ul)!.5!(ll)$);
\path[fill=\fcolor, name path=p1] (90+\off:\rad) -- (-90-\off:\rad) arc[start angle=-90-\off,end angle=-90+\off,radius=\rad] (-90+\off:\rad) -- (90-\off:\rad) arc[start angle=90-\off,end angle=90+\off,radius=\rad] (90+\off:\rad);
\path[fill=\fcolor, name path=p2] (180-\off:\rad)  to[out=0,in=-180] (lr) to (ur) [out=0,in=-180] to (\off:\rad) arc[start angle=\off,delta angle=-2*\off,radius=\rad] (-\off:\rad) [out=-180,in=0] to (ul) to (ll) to[out=-180,in=0] (180+\off:\rad) arc[start angle=180+\off,delta angle=-2*\off,radius=\rad] (180-\off:\rad);
\path[fill=white, name intersections={of=p1 and p2, name=i}] (i-1) -- (i-2) -- (i-4) -- (i-3) -- cycle;
\draw[dashed] (0,0) circle (\rad);
\draw[graphstyle] (wb) to (br) to (wt) to (bl) to (wb);
\draw[graphstyle] (br) to (wr) to (0:\rad);
\draw[graphstyle] (bl) -- (wl) -- (180:\rad);
\draw[graphstyle] (wt) to (90:\rad); \draw[graphstyle] (wb) to (-90:\rad);
\foreach \c in {bl,br,wt,wb,wr,wl} {\fill[black] (\c) circle (\bvertrad);}
\foreach \c in {wt,wb,wr,wl} {\fill[white] (\c) circle (\wvertrad);}
\draw[asdstyle,righthairs] (90+\off:\rad) -- (-90-\off:\rad);
\draw[asdstyle,righthairs] (-90+\off:\rad) -- (90-\off:\rad);
\draw[asdstyle,lefthairs] (mr) to (ur) [out=0,in=180] to (\off:\rad);
\draw[asdstyle,righthairs] (mr) to (lr) to[in=0,out=180] (180-\off:\rad);
\draw[asdstyle,righthairs] (ml) to (ul) [out=0,in=180] to (-\off:\rad);
\draw[asdstyle,lefthairs] (ml) to (ll) to[in=0,out=180] (180+\off:\rad);\\};

\node (m) [matrix] at (2.5*\rad,0) {
\path[fill=\fcolor] (0,0) to[out=-90,in=90]  (-90-\off:\rad) arc[start angle=-90-\off,delta angle=2*\off,radius=\rad] (-90+\off:\rad) to[out=90,in=-90] (0,0);
\path[fill=\fcolor] (0,0) to[out=0,in=180]  (-\off:\rad) arc[start angle=-\off,delta angle=2*\off,radius=\rad] (\off:\rad) to[out=180,in=0] (0,0);
\path[fill=\fcolor] (0,0) to[out=90,in=-90]  (90-\off:\rad) arc[start angle=90-\off,delta angle=2*\off,radius=\rad] (90+\off:\rad) to[out=-90,in=90] (0,0);
\path[fill=\fcolor] (0,0) to[out=180,in=0]  (180-\off:\rad) arc[start angle=180-\off,delta angle=2*\off,radius=\rad] (180+\off:\rad) to[out=0,in=180] (0,0);
\draw[dashed] (0,0) circle (\rad);
\draw[asdstyle,lefthairsnogap] (0,0) to[out=90,in=-90] (90+\off:\rad);
\draw[asdstyle,righthairsnogap] (0,0) to[out=-90,in=90]  (-90-\off:\rad);
\draw[asdstyle,righthairsnogap] (0,0) to[out=90,in=-90]  (90-\off:\rad);
\draw[asdstyle,lefthairsnogap] (0,0) to[out=-90,in=90]  (-90+\off:\rad);
\draw[asdstyle,lefthairsnogap] (0,0) to[out=0,in=180] (\off:\rad);
\draw[asdstyle,righthairsnogap] (0,0) to[out=180,in=0] (180-\off:\rad);
\draw[asdstyle,righthairsnogap] (0,0) to[out=0,in=180] (-\off:\rad);
\draw[asdstyle,lefthairsnogap] (0,0) to[out=180,in=0] (180+\off:\rad);\\};

\node (r) [matrix] at (5*\rad,0) {
\coordinate (bl) at (180+90:\rad/2); \coordinate (br) at (0+90:\rad/2);
\coordinate (wt) at (90+90:\rad*.8); \coordinate (wb) at (-90+90:\rad*.8);
\coordinate (wr) at (0+90:\rad*.9); \coordinate (wl) at (180+90:\rad*.9);
\coordinate (ur) at ($(-\off+90:\rad)!\ifac!(-90+\off:\rad)$); \coordinate (lr) at ($(\off-90:\rad)!\ifac!(90-\off:\rad)$);
\coordinate (ul) at ($(\off+90:\rad)!\ifac!(-90-\off:\rad)$); \coordinate (ll) at ($(-\off-90:\rad)!\ifac!(90+\off:\rad)$);
\coordinate (mr) at ($(ur)!.5!(lr)$); \coordinate (ml) at ($(ul)!.5!(ll)$);
\path[fill=\fcolor, name path=p1] (90+90+\off:\rad) -- (-90+90-\off:\rad) arc[start angle=-\off,end angle=\off,radius=\rad] -- (180-\off:\rad) arc[start angle=180-\off,end angle=180+\off,radius=\rad];
\path[fill=\fcolor, name path=p2] (180+90-\off:\rad)  to[out=90,in=-90] (lr) to (ur) [out=90,in=-90] to (\off+90:\rad) arc[start angle=90+\off,delta angle=-2*\off,radius=\rad] [out=-90,in=90] to (ul) to (ll) to[out=-90,in=90] (180+90+\off:\rad) arc[start angle=180+90+\off,delta angle=-2*\off,radius=\rad];
\path[fill=white, name intersections={of=p1 and p2, name=i}] (i-1) -- (i-2) -- (i-4) -- (i-3) -- cycle;
\draw[dashed] (0,0) circle (\rad);
\draw[graphstyle] (wb) to (br) to (wt) to (bl) to (wb);
\draw[graphstyle] (br) to (wr) to (90:\rad);
\draw[graphstyle] (bl) -- (wl) -- (180+90:\rad);
\draw[graphstyle] (wt) to (90+90:\rad); \draw[graphstyle] (wb) to (-90+90:\rad);
\foreach \c in {bl,br,wt,wb,wr,wl} {\fill[black] (\c) circle (\bvertrad);}
\foreach \c in {wt,wb,wr,wl} {\fill[white] (\c) circle (\wvertrad);}
\draw[asdstyle,righthairs] (90+90+\off:\rad) -- (-90+90-\off:\rad);
\draw[asdstyle,righthairs] (-90+90+\off:\rad) -- (90+90-\off:\rad);
\draw[asdstyle,lefthairs] (mr) to (ur) [out=0+90,in=180+90] to (\off+90:\rad);
\draw[asdstyle,righthairs] (mr) to (lr) to[in=0+90,out=180+90] (180+90-\off:\rad);
\draw[asdstyle,righthairs] (ml) to (ul) [out=0+90,in=180+90] to (-\off+90:\rad);
\draw[asdstyle,lefthairs] (ml) to (ll) to[in=0+90,out=180+90] (180+90+\off:\rad);\\};
\coordinate (rw) at ($(r.west)+(-1mm,0)$); \coordinate (me) at ($(m.east)+(1mm,0)$);
\coordinate (mw) at ($(m.west)+(-1mm,0)$); \coordinate (le) at ($(l.east)+(1mm,0)$);
\foreach \ca/\cb in {le/mw,me/rw} {
\draw[arrowstyle] (\ca) -- (\cb);
\draw[arrhdstyle] ($(\cb)+(-1.2mm,1.2mm)$) -- (\cb); \draw[arrhdstyle] ($(\cb)+(-1.2mm,-1.2mm)$) -- (\cb);}
\end{tikzpicture}
\caption{The square move as Lagrangian surgery.  The left and right frames show the front projections of alternating Legendrians $\Lambda$, $\Lambda'$ related by a square move of their associated bipartite graphs.  The shaded regions indicate the projections of a family of exact Lagrangian fillings, which become singular in the middle frame.}\label{fig:surgery}
\end{figure}
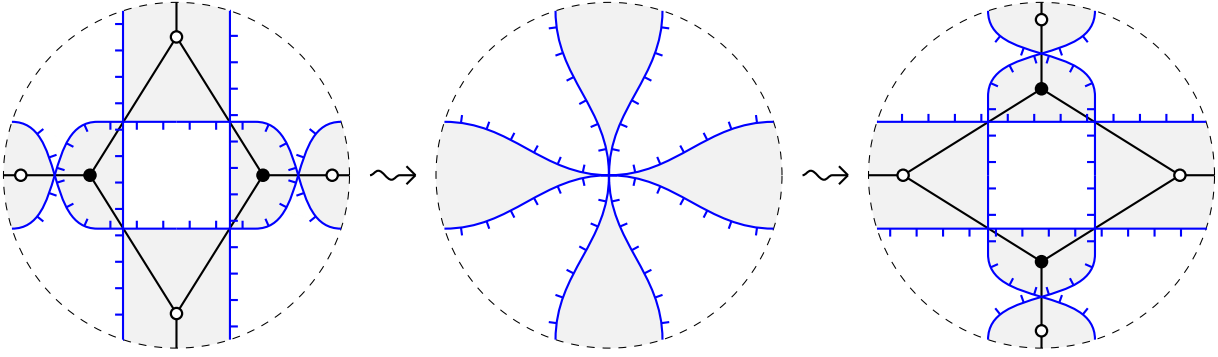

Given an exact Lagrangian surface $L$ with a nodal singularity, one can produce two exact Lagrangians $L_+$, $L_-$ which coincide with $L$ outside a neighborhood of the singularity.  These are said to differ by Lagrangian surgery \cite{LS,Pol}.  Both $L_+$ and $L_-$ have  degenerations to $L$ through smooth exact Lagrangians but are not themselves related by a Hamiltonian isotopy.  The degenerations to $L$ are accompanied by the collapse of a sphere, the vanishing cycle of the surgery.  Surgery of exact Lagrangians is directly related to wall-crossing phenomena in Floer cohomology \cite[Chap. 10]{FOOO}, and thus to the appearance of cluster transformations in symplectic geometry in the guise of wall-crossing transformations \cite{KoS,KoS2}. 

In the present setting, Lagrangian surgery on conjugate Lagrangians provides the symplectic interpretation of the square move on bipartite graphs.

\begin{proposition}
\label{prop:squaremovesurgery}
Let $L$, $L'$ be conjugate Lagrangians related by performing a square move at $\partial F_k \in H_1(L;\Z)$.  Then $L$ and $L'$ are related by a Lagrangian surgery whose vanishing cycle is $\partial F_k$.
\end{proposition}
\begin{proof}
Consider $T^*\R^2_x \cong \R^2_x \times \R^2_y$ with base coordinates $x = (x_1, x_2)$ and fiber coordinates $y = (y_1, y_2)$. 
Let $\{L_t \subset T^*\R^2_x\}_{t \in \R}$ denote the following family of Lagrangians, smooth except when $t=0$. For $t > 0$, $L_t \subset T^*\R^2_x$ is parametrized by $(0,1) \times S^1$ via
\begin{align*} (r, \theta) \mapsto & (x_1, x_2; y_1, y_2)\\ & = (tr^{-1/2} \cos \theta, \,t(1-r)^{-1/2} \sin \theta;\, t(1-r)^{-1/2} \cos \theta,\, -tr^{-1/2} \sin \theta). \end{align*}
In the projection to $\R^2_x$, $r$ parametrizes all ellipses passing through the four points $(\pm t, \pm t)$. One easily checks that $L_t$ is exact. It is the conjugate Lagrangian of the alternating Legendrian whose front projection is the union of the lines $x_1 = \pm t$, co-oriented towards the $x_1$-axis, and the lines $x_2 = \pm t$, co-oriented away from the $x_2$-axis.

We define $L_t$ for $t < 0$ similarly, but with the $x_1$- and $x_2$-axes reversed, and let $L_0$ denote the union of the conormal planes to the coordinate axes. Together this family interpolates between $L$,~$L'$ as pictured in \Cref{fig:surgery}, up to restricting to a sufficiently small disk and perturbing the middle frame of the figure so that near the origin its front projection coincides with the coordinate axes.

We recall the local model for surgery on the immersed Lagrangian $T^*_0 \R_x^2 \cup \R_x^2$ following \cite{Sei}. Let $C \subset T^*\R_{x_1}$ be the image of a smooth embedding of $\R$ such that $C$ coincides with $(\R_+ \times \{0\}) \cup (\{0\} \times \R_+)$ outside of a compact set and such that $C$ and $-C$ are disjoint. Then the surgery $T^*_0 \R_x^2 \# \R_x^2 \subset T^* \R_x^2$ is the orbit of $C$ under the Hamiltonian $S^1$ action lifting rotation in $\R_x^2$.

We now transform the family $L_t$ by rotation in $T^*\R_{x_2}$ so that for $t > 0$ it is parametrized via
$$ (r, \theta) \mapsto (tr^{-1/2} \cos \theta,\, tr^{-1/2} \sin \theta;\, t(1-r)^{-1/2} \cos \theta,\, t(1-r)^{-1/2} \sin \theta). $$
This takes $L_0$ to $T^*_0 \R_x^2 \cup \R_x^2$. For $t > 0$, $L_t$ is now $S^1$-invariant and meets $T^*\R_{x_1}$ along the closed curve
$$C_t = \{(x_1, y_1)\, |\, x_1^{-2} + y_1^{-2} = t^{-2}, x_1, y_1 > 1\},$$
which is asymptotic to the lines $x_1 = t$, $y_1 = t$.

Now let $H_t \in C^\infty(T^*\R_{x}^2)$ be an $S^1$-invariant function whose restriction to $T^*\R_{x_1}$ coincides with $\frac{x_1 + y_1}{|x_1 + y_1|}(x_1 - y_1)$ on the region $U$ where $|x_1 + y_1| > (\sqrt{2} - 1)t - \epsilon$. The Hamiltonian flow of $H_t$ preserves $T^*\R_{x_1}$ and its restriction to $U$ is the contracting flow $\frac{-(x_1 + y_1)}{|x_1 + y_1|} (\partial_{x_1} + \partial_{y_1})$ towards the line $x_1 + y_1 = 0$. After flowing for time $t$ the curve $C_t$ remains in the first quadrant but is now asymptotic to the coordinate axes. Since $H_t$ is $S^1$-invariant its Hamiltonian flow preserves the $S^1$-invariance of $L_t$. By a further $S^1$-invariant Hamiltonian perturbation supported near infinity we can make $L_t$ coincide with $T^*_0 \R_x^2 \cup \R_x^2$ outside a sufficiently large ball centered at the origin. Thus for $t>0$, $L_t$ is Hamiltonian isotopic to the surgery $T^*_0 \R_x^2 \# \R_x^2$.

On the other hand, by switching the roles of $x_1$ and $x_2$ we see by symmetry that for $t<0$ $L_t$ is Hamiltonian isotopic to the opposite surgery $\R_x^2 \# T^*_0 \R_x^2$. In either case it is evident that the vanishing cycle is as stated. 
\end{proof}

\subsection{The boundary measurement map} 
In \Cref{sec:Dylan} we saw that essentially all Legendrian braid satellites of cocircles have alternating Legendrian isotopy representatives. This leads to the question of describing the associated cluster charts in terms of natural coordinates on spaces of $\beta$-filtered local systems. We treat here the fundamental case of positroid strata, showing that our study of Legendrian isotopy recovers the boundary measurement map of Postnikov. This says in particular that, in terms of Pl\"ucker coordinates, the cluster charts produced by Hamiltonian isotopy of conjugate Lagrangians can be described as sums over perfect matchings on bipartite graphs.

Recall from \Cref{sec:Dylan} that a reduced plabic graph in $D^2$ is one whose alternating Legendrian satisfies \Cref{def:reduced}. These are exactly the Legendrians which, up to isotopy, arise from positroids:

\begin{proposition}
Given a reduced plabic graph $\Gamma$, there exists a unique cyclic rank matrix $r$ such that $\Lambda_\Gamma$ and $\Lambda_r$ are Legendrian isotopic.  
\end{proposition}
\begin{proof}
First note that the components of $\Lambda_r$ determine a bijection between incoming and outgoing intersections of its front projection with the boundary of $D^2$, and this bijection determines $r$.  Thus given $\Gamma$ we define the associated cyclic rank matrix as the one corresponding to the boundary matching of $\Lambda_\Gamma$ (in the terminology of \cite{Po} the restriction that all boundary-adjacent vertices of $\Gamma$ are white means we need only consider matchings or permutations rather than decorated permutations). On the other hand, $\Lambda_r$ is clearly reduced and by \Cref{prop:uniqueisotopies} any reduced Legendrians with the same boundary matching are Legendrian isotopic.
\end{proof}

Moreover, Proposition \ref{prop:uniqueisotopies} tells us that 
the Legendrian $\Lambda_\Gamma$ 
admits a \emph{contractible} space of isotopies, fixed at the boundary of the disk, 
to the positroid Legendrian $\Lambda_r$.
Thus by \Cref{prop:GKS}, there is a canonical isomorphism of framed moduli spaces
$\cM_1^{fr}(\Lambda_\Gamma) \cong \cM_1^{fr}(\Lambda_r)$.  On the other hand, in \Cref{thm:positroidasmoduli} we gave 
a canonical identification $\cM_1^{fr}(\Lambda_r) \cong \Pi_r$ of the positroid stratum with the framed moduli space of the positroid Legendrian. 
If $L$ is the conjugate Lagrangian of $\Lambda_\Gamma$, we can compose this identification with the family
Floer theoretic morphism $Loc^{fr}_{\:1}(L) \into \cM_1^{fr}(\Lambda_\Gamma)$ to 
obtain a toric chart on the positroid stratum.  Finally, recalling that $L$ retracts to the graph $\Gamma$,
we have the composition 
$$\im_\Gamma: Loc^{fr}_{\:1}(\Gamma) \cong Loc^{fr}_{\:1}(L) \into \cM_1^{fr}(\Lambda_\Gamma) \cong \cM_1^{fr}(\Lambda_r) \cong \Pi_r.$$
Here, the framing on $L$ or on $\Gamma$ is again a trivialization of each connected component
of the boundary. That is, $Loc^{fr}_{\:1}(\Gamma)$ is the algebraic torus $H^1(\Gamma,\partial \Gamma; \G_m)$, where $\partial \Gamma = \Gamma \cap \partial \cD$. We also implicitly use the standard trivialization of \Cref{def:coords} to define $\im_\Gamma$.

On the other hand, the motivation for considering reduced plabic graphs in \cite{Po} is that each gives rise to a \textbf{boundary measurement map}
\[
\Meas_{\Gamma}: Loc^{fr}_{\:1}(\Gamma) \into \Pi_r.
\]
We recall the definition of $\Meas_{\Gamma}$ as reformulated by the main result of \cite{Ta}.  Fix a cyclically-ordered labeling of the boundary vertices of $\Gamma$ by $\{1,\dotsc,n\}$.  Orient $\Gamma$ so that every white vertex has exactly one incoming edge and every black vertex has exactly one outgoing edge; following \cite{Po} this is called a perfect orientation of $\Gamma$.  Let $I \subset [1,n]$ be the subset of incoming boundary vertices of $\Gamma$, which necessarily has $k$ elements.  If $J \subset [1,n]$ is any other $k$-element subset, we say a flow from $I$ to $J$ is a collection of disjoint self-avoiding oriented cycles in $\Gamma$, relative to $\partial \Gamma$, such that each nonclosed cycle connects a boundary vertex in $I$ to a boundary vertex in $J$.  

Each flow $F$ gives rise to a function on $Loc^{fr}_{\:1}(\Gamma)$, which by a slight abuse we also denote by $F$.  Then $\Meas_{\Gamma}$ is defined by the condition that the pullback of the $J$th Pl{\"u}cker coordinate to $Loc^{fr}_{\:1}(\Gamma)$ is 
\begin{equation}\label{eq:bmm}
\Meas_{\Gamma}^*\Delta_J = \sum_{F:I \to J}F,
\end{equation}
where the sum is over all flows from $I$ to $J$.  This definition turns out to be independent of the choice of perfect orientation (of course, it is only the \emph{ratios} of Pl\"ucker coordinates that are meaningful, and changing the orientation may rescale all of them by a common factor).  Note also that the labels of the boundary vertices are used to determine the sign of $\Meas_{\Gamma}^*\Delta_J$.

We want to compare $\im_{\Gamma}$ and $\Meas_{\Gamma}$, but we can see already that we can only expect them to agree up to certain signs. For example, the definition of $\Meas_{\Gamma}$ implicitly depends on how the boundary vertices are labeled by $1,\dotsc, n$, since this ordering is needed to fix the signs of Pl\"ucker coordinates in \Cref{eq:bmm}. The definition of $\im_{\Gamma}$, on the other hand, is manifestly independent of the boundary labels. This is related to the fact that $\im_\Gamma$ transforms by a \emph{signed} cluster transformation under square moves (it is defined using standard face coordinates, so \Cref{thm:clustertrans} applies). The maps $\im_\Gamma$ do not naturally define a positive locus in $\Gr(k,n)$, and we cannot expect them to given that they are cyclically invariant: when $k$ is even the usual positive part of $\Gr(k,n)$ is itself not cyclically invariant. 


\begin{theorem}
\label{thm:positroid}
Let $\Gamma$ be a reduced plabic graph and $\Pi_r$ the associated positroid stratum.  The maps $\Meas_{\Gamma}$ and $\im_{\Gamma}$ coincide up to signs of Pl\"ucker coordinates.
\end{theorem}
\begin{proof}
The main idea is that as $\Gamma$ ranges over the set of all reduced plabic graphs of all positroid strata, $\Meas_{\Gamma}$ is determined by certain ``recursion relations'' with respect to direct sum and projections \cite[Section 4.4]{ABCGPT}.  Thus it suffices to show that $\im_{\Gamma}$ satisfies the same relations (up to signs), and to verify the theorem by hand in the trivial cases when $\Pi_r$ is the open stratum of $\Gr(1,3)$ or $\Gr(2,3)$.

For any pair of Grassmannians there is a direct sum map $\Gr(k_1,n_1) \times \Gr(k_2,n_2) \to \Gr(k_1+k_2,n_1+n_2)$ which on $\coeffs$-points acts by taking $(\coeffs^{k_1} \onto E_1, \coeffs^{k_2} \onto E_2)$ to $\coeffs^{k_1+k_2} \onto E_1 \oplus E_2$ (here $E_1$, $E_2$ are locally free modules of ranks $n_1$, $n_2$).  Let $\Gamma_1$ and $\Gamma_2$ be two reduced plabic graphs with associated positroid strata $\Pi_{r_1} \subset \Gr(k_1,n_1)$, $\Pi_{r_2} \subset \Gr(k_2,n_2)$.  Let $\Gamma_3$ denote the reduced plabic graph which is the disjoint union of $\Gamma_1$ and $\Gamma_2$, with boundary vertices labeled so that those from $\Gamma_1$ retain their original labels while those from $\Gamma_2$ have $n_1$ added to their labels.  Let $\Pi_{r_3} \subset \Gr(k_1+k_2,n_1+n_2)$ be the positroid stratum associated with $\Gamma_3$; it is the image of $\Pi_{r_1} \times \Pi_{r_2}$ under the direct sum map.  There is an obvious isomorphism $Loc^{fr}_{\:1}(\Gamma_1) \times Loc^{fr}_{\:1}(\Gamma_2) \cong Loc^{fr}_{\:1}(\Gamma_3)$.
The boundary measurement map $\Meas_{\Gamma_3}$ is determined by $\Meas_{\Gamma_1}$, $\Meas_{\Gamma_2}$ in the sense that the following diagram commutes (note that the indexing prescription on $\Gamma_3$ fixes the signs of Pl\"ucker coordinates in the bottom map):
\[
\begin{tikzpicture}
  \matrix (m) [matrix of math nodes,row sep=3em,column sep=4em,minimum width=2em]
  {
     Loc^{fr}_{\:1}(\Gamma_1) \times Loc^{fr}_{\:1}(\Gamma_2) & \Pi_{r_1} \times \Pi_{r_2} \\
     Loc^{fr}_{\:1}(\Gamma_3) & \Pi_{r_3} \\};
  \path[-stealth]
    (m-1-1) edge (m-2-1)
            edge node [above] {$\Meas_{\Gamma_1} \times \Meas_{\Gamma_2}$} (m-1-2)
    (m-2-1) edge node [below] {$\Meas_{\Gamma_3}$} (m-2-2)
    (m-1-2) edge (m-2-2);
\end{tikzpicture}
\]

On the other hand, if we replace the boundary measurement maps above by their counterparts $\im_{\Gamma_1}$, $\im_{\Gamma_2}$, $\im_{\Gamma_3}$, then the above diagram still commutes. It suffices to show that the direct sum map corresponds to the isotopy isomorphism $\cM^{fr}_{\:1}(\Lambda_{r_1}) \times \cM^{fr}_{\:1}(\Lambda_{r_2}) \congto \cM^{fr}_{\:1}(\Lambda_{r_3})$ under the identification of \Cref{thm:positroidasmoduli} (by \Cref{prop:uniqueisotopies} there is a unique such isomorphism, and this uniqueness forces the diagram to commute).  This follows from the appearance of direct sums in the Reidemeister-II move (see \Cref{fig:reid}): the isotopy relates the maximal-rank stalk of a sheaf in $\cM^{fr}_{\:1}(\Lambda_{r_3})$ to the maximal rank stalks of sheaves in $\cM^{fr}_{\:1}(\Lambda_{r_1})$, $\cM^{fr}_{\:1}(\Lambda_{r_2})$ by a sequence of Reidemeister-II's.  This determines the map of positroid strata since by the construction of \Cref{thm:positroidasmoduli} these stalks and the maps they receive from the boundary stalks determine the maps $\im_{\Gamma_1}$, $\im_{\Gamma_2}$, $\im_{\Gamma_3}$.  See \Cref{fig:directsum} for an example.

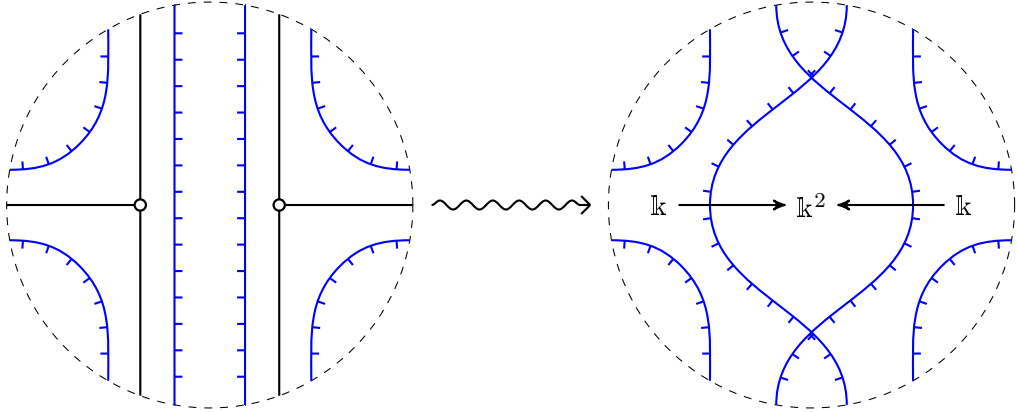
\begin{figure}
\centering
\begin{tikzpicture}
\newcommand*{\off}{10}; \newcommand*{\rad}{2.7}; \newcommand*{\ang}{70}; \newcommand*{\fac}{.3};
\node (l) [matrix] at (0,0) {
\draw[dashed] (0,0) circle (\rad);
\coordinate (tl) at (180-\ang:\rad); \coordinate (bl) at (180+\ang:\rad);
\coordinate (tr) at (\ang:\rad); \coordinate (br) at (-\ang:\rad);
\coordinate (lv) at ($(tl)!.5!(bl)$); \coordinate (rv) at ($(tr)!.5!(br)$);
\draw[graphstyle] (tl) to (bl); \draw[graphstyle] (tr) to (br); \draw[graphstyle] (180:\rad) to (lv); \draw[graphstyle] (0:\rad) to (rv);
\draw[asdstyle,lefthairs] (180-\off:\rad) to[in=180+45,out=0] ($(lv)+(-\fac*\rad,\fac*\rad)$) to[in=-90,out=45] (180-\ang+\off:\rad);
\draw[asdstyle,righthairs] (180+\off:\rad) to[out=0,in=180-45] ($(lv)+(-\fac*\rad,-\fac*\rad)$) to[out=-45,in=90] (180+\ang-\off:\rad);
\draw[asdstyle,lefthairs] (-\off:\rad) to[in=45,out=180] ($(rv)+(\fac*\rad,-\fac*\rad)$) to[in=90,out=180+45] (-\ang+\off:\rad);
\draw[asdstyle,righthairs] (\off:\rad) to[out=180,in=-45] ($(rv)+(\fac*\rad,\fac*\rad)$) to[out=180-45,in=-90] (\ang-\off:\rad);
\draw[asdstyle,lefthairs] ($(180+\ang+\off:\rad)!.5!(180-\ang-\off:\rad)$) to (180+\ang+\off:\rad);
\draw[asdstyle,righthairs] ($(180+\ang+\off:\rad)!.5!(180-\ang-\off:\rad)$) to (180-\ang-\off:\rad);
\draw[asdstyle,lefthairs] ($(\ang+\off:\rad)!.5!(-\ang-\off:\rad)$) to (\ang+\off:\rad);
\draw[asdstyle,righthairs] ($(\ang+\off:\rad)!.5!(-\ang-\off:\rad)$) to (-\ang-\off:\rad);
\foreach \c in {lv,rv} {\fill[black] (\c) circle (\bvertrad); \fill[white] (\c) circle (\wvertrad);}\\};
\node (r) [matrix] at (8,0) {
\draw[dashed] (0,0) circle (\rad);
\coordinate (tl) at (180-\ang:\rad); \coordinate (bl) at (180+\ang:\rad);
\coordinate (tr) at (\ang:\rad); \coordinate (br) at (-\ang:\rad);
\coordinate (lv) at ($(tl)!.5!(bl)$); \coordinate (rv) at ($(tr)!.5!(br)$);
\draw[asdstyle,lefthairs] (180-\off:\rad) to[in=180+45,out=0] ($(lv)+(-\fac*\rad,\fac*\rad)$) to[in=-90,out=45] (180-\ang+\off:\rad);
\draw[asdstyle,righthairs] (180+\off:\rad) to[out=0,in=180-45] ($(lv)+(-\fac*\rad,-\fac*\rad)$) to[out=-45,in=90] (180+\ang-\off:\rad);
\draw[asdstyle,lefthairs] (-\off:\rad) to[in=45,out=180] ($(rv)+(\fac*\rad,-\fac*\rad)$) to[in=90,out=180+45] (-\ang+\off:\rad);
\draw[asdstyle,righthairs] (\off:\rad) to[out=180,in=-45] ($(rv)+(\fac*\rad,\fac*\rad)$) to[out=180-45,in=-90] (\ang-\off:\rad);
\draw[asdstyle,lefthairs] (.5*\rad,0) to[in=90,out=-90] (180+\ang+\off:\rad);
\draw[asdstyle,righthairs] (.5*\rad,0) to[out=90,in=-90] (180-\ang-\off:\rad);
\draw[asdstyle,lefthairs] (-.5*\rad,0) to[in=-90,out=90] (\ang+\off:\rad);
\draw[asdstyle,righthairs] (-.5*\rad,0) to[out=-90,in=90] (-\ang-\off:\rad);
\node (ls) at (-.75*\rad,0) {$\coeffs$}; \node (rs) at (.75*\rad,0) {$\coeffs$}; \node (cs) at (0,0) {$\coeffs^2$};
\draw[genmapstyle] (ls) to (cs); \draw[genmapstyle] (rs) to (cs);\\};
\coordinate (rw) at ($(r.west)+(-1mm,0)$); \coordinate (le) at ($(l.east)+(1mm,0)$);
\draw[arrowstyle] (le) -- (rw);
\draw[arrhdstyle] ($(rw)+(-1.2mm,1.2mm)$) -- (rw); \draw[arrhdstyle] ($(rw)+(-1.2mm,-1.2mm)$) -- (rw);
\end{tikzpicture}
\caption{The isotopy corresponding to the direct sum map when $\Gamma_1$ and $\Gamma_2$ have a single trivalent white vertex, so $\Pi_{r_1} = \Pi_{r_2}$ is the big positroid stratum in $\Gr(1,3)$.  The framed moduli space of the left picture is manifestly isomorphic $\Pi_{r_1} \times \Pi_{r_2}$, applying the construction of \Cref{thm:positroidasmoduli} separately to its left and right halves.  The framed moduli space of the right picture is manifestly isomorphic to a positroid stratum $\Pi_{r_3}$ in $\Gr(2,6)$ where the first three (and last three) columns of any matrix representative are pairwise linearly dependent.  The crossing conditions on the right assert that the stalk of a sheaf in the middle region is canonically identified with the direct sum of a stalk from the left region and the right region.}
\label{fig:directsum}
\end{figure}

Now let $r$ be a cyclic rank matrix of type $(k,n)$ such that $r_{12} = 2$, and let $\Gamma$ be a reduced plabic graph for $r$.  Assume that the bicolored graph $\Gamma'$ obtained by gluing boundary vertices 1 and 2 together is again a reduced plabic graph, and let $r'$ be its associated cyclic rank matrix.  
We have a projection map $\Pi_r \to \Pi_{r'} \subset \Gr(k-1,n-2)$ which on $\coeffs$-points take $\coeffs^n \onto E$ to $\coeffs^{n-2} \onto E/\langle v_1-v_2\rangle$, where $v_1$, $v_2$ are the images of $1$ in the first two factors of $\coeffs^n$.   There is a natural map $Loc^{fr}_{\:1}(\Gamma) \onto Loc^{fr}_{\:1}(\Gamma')$, since the framings let us identify the stalks at boundary vertices 1 and 2 of a framed local system on $\Gamma$.  The boundary measurement map $\Meas_{\Gamma'}$ is determined by $\Meas_{\Gamma}$ in the sense that the following diagram commutes:
\[
\begin{tikzpicture}
  \matrix (m) [matrix of math nodes,row sep=3em,column sep=4em,minimum width=2em]
  {
     Loc^{fr}_{\:1}(\Gamma) & \Pi_r \\
     Loc^{fr}_{\:1}(\Gamma') & \Pi_{r'} \\};
  \path[-stealth]
    (m-1-1) edge (m-2-1)
            edge node [above] {$\Meas_{\Gamma}$} (m-1-2)
    (m-2-1) edge node [below] {$\Meas_{\Gamma'}$} (m-2-2)
    (m-1-2) edge (m-2-2);
\end{tikzpicture}
\]

As above, we claim the diagram still commutes up to signs of Pl\"ucker coordinates after replacing the boundary measurement maps by $\im_{\Gamma}$, $\im_{\Gamma'}$.  In terms of alternating Legendrians, gluing vertices 1 and 2 of $\Gamma$ together corresponds to ``capping off'' the front projection of $\Lambda_\Gamma$ with two strands going outside the disk, then pulling the cap back inside the disk.  Let $\widehat{\Lambda}_r$ be the Legendrian obtained from $\Lambda_r$ by capping off its front projection in the same way.  There is a natural map $\cM^{fr}_{\:1}(\Lambda_r) \to \cM^{fr}_{\:1}(\widehat{\Lambda}_r)$ constructed the same way as the gluing map $Loc^{fr}_{\:1}(\Gamma) \onto Loc^{fr}_{\:1}(\Gamma')$.    It suffices to show that the projection map corresponds to the composition of this with the isotopy isomorphism $\cM^{fr}_{\:1}(\widehat{\Lambda}_r) \to \cM^{fr}_{\:1}(\Lambda_{r'})$ under the identification of \Cref{thm:positroidasmoduli} (as in the direct sum case, by \Cref{prop:uniqueisotopies} there is a unique such isomorphism, and this uniqueness forces the diagram to commute).  

\begin{figure}
\centering
\begin{tikzpicture}
\newcommand*{\off}{10}; \newcommand*{\rad}{2.7}; \newcommand*{\ang}{70}; \newcommand*{\fac}{.3};
\node (l) [matrix] at (0,0) {
\draw[dashed] (0,0) circle (\rad);
\coordinate (tl) at (180-\ang:\rad); \coordinate (bl) at (180+\ang:\rad);
\coordinate (tr) at (\ang:\rad); \coordinate (br) at (-\ang:\rad);
\coordinate (lv) at ($(tl)!.5!(bl)$); \coordinate (rv) at ($(tr)!.5!(br)$);
\draw[asdstyle,righthairs] (180+\off:\rad) to[out=0,in=180-45] ($(lv)+(-\fac*\rad,-\fac*\rad)$) to[out=-45,in=90] (180+\ang-\off:\rad);
\draw[asdstyle,lefthairs] (-\off:\rad) to[in=45,out=180] ($(rv)+(\fac*\rad,-\fac*\rad)$) to[in=90,out=180+45] (-\ang+\off:\rad);
\draw[asdstyle,lefthairs] (0,1.3*\rad) to[in=90,out=0] (\ang-\off:\rad) to[in=180-45,out=-90] ($(rv)+(\fac*\rad,\fac*\rad)$) to[in=180,out=-45] (\off:\rad);
\draw[asdstyle,righthairs] (0,1.3*\rad) to[out=180,in=90] (180-\ang+\off:\rad)  to[out=-90,in=45] ($(lv)+(-\fac*\rad,\fac*\rad)$) to[out=180+45,in=0] (180-\off:\rad);
\draw[asdstyle,lefthairs] (0,1.1*\rad) to[out=180,in=90] (180-\ang-\off:\rad) to[in=90,out=-90] (.5*\rad,0) to[in=90,out=-90] (180+\ang+\off:\rad);
\draw[asdstyle,righthairs] (0,1.1*\rad) to[out=0,in=90] (\ang+\off:\rad) to[out=-90,in=90] (-.5*\rad,0) to[out=-90,in=90] (-\ang-\off:\rad);
\node (ls) at (-.75*\rad,0) {$\coeffs$}; \node (rs) at (.75*\rad,0) {$\coeffs$}; \node (cs) at (0,0) {$\coeffs^2$};
\draw[genmapstyle] (ls) to (cs); \draw[genmapstyle] (rs) to (cs);\\};
\node (r) [matrix] at (8.1,-\rad*.15) {
\draw[dashed] (0,0) circle (\rad);
\coordinate (tl) at (180-\ang:\rad); \coordinate (bl) at (180+\ang:\rad);
\coordinate (tr) at (\ang:\rad); \coordinate (br) at (-\ang:\rad);
\coordinate (lv) at ($(tl)!.5!(bl)$); \coordinate (rv) at ($(tr)!.5!(br)$);
\draw[asdstyle,righthairs] (180+\off:\rad) to[out=0,in=180-45] ($(lv)+(-\fac*\rad,-\fac*\rad)$) to[out=-45,in=90] (180+\ang-\off:\rad);
\draw[asdstyle,lefthairs] (-\off:\rad) to[in=45,out=180] ($(rv)+(\fac*\rad,-\fac*\rad)$) to[in=90,out=180+45] (-\ang+\off:\rad);
\draw[asdstyle,lefthairs] (0,.7*\rad) to[in=180,out=0] (\off:\rad);
\draw[asdstyle,righthairs] (0,.7*\rad) to[out=180,in=0] (180-\off:\rad);
\draw[asdstyle,lefthairs] (0,-.3*\rad) to[in=90,out=180] (180+\ang+\off:\rad);
\draw[asdstyle,righthairs] (0,-.3*\rad) to[out=0,in=90] (-\ang-\off:\rad);
\node (cs) at (0,.2*\rad) {$\coeffs$};\\};
\coordinate (rw) at ($(r.west)+(-1mm,0)$); \coordinate (le) at ($(l.east)+(1mm,-.15*\rad)$);
\draw[arrowstyle] (le) -- (rw);
\draw[arrhdstyle] ($(rw)+(-1.2mm,1.2mm)$) -- (rw); \draw[arrhdstyle] ($(rw)+(-1.2mm,-1.2mm)$) -- (rw);
\end{tikzpicture}
\caption{The isotopy corresponding to the projection map when $\Pi_r$ is the positroid stratum from the right picture of \Cref{fig:directsum}.  The starting point is to cap off $\Lambda_r$ to obtain the Legendrian $\widehat{\Lambda}_r$ whose front projection is on the left; there is a canonical map from $\cM^{fr}_{\:1}(\Lambda_r)$ to $\cM^{fr}_{\:1}(\widehat{\Lambda}_r)$.  When we isotope to the right hand side the rank two region in the middle is replaced by a rank one region where the boundary stalks from the left and right sides of the picture are identified up to a scalar.  Under the correspondence of \Cref{thm:positroidasmoduli} this is exactly the projection map from $\Pi_r$ to $\Pi_{r'}$}
\label{fig:capping}
\end{figure}
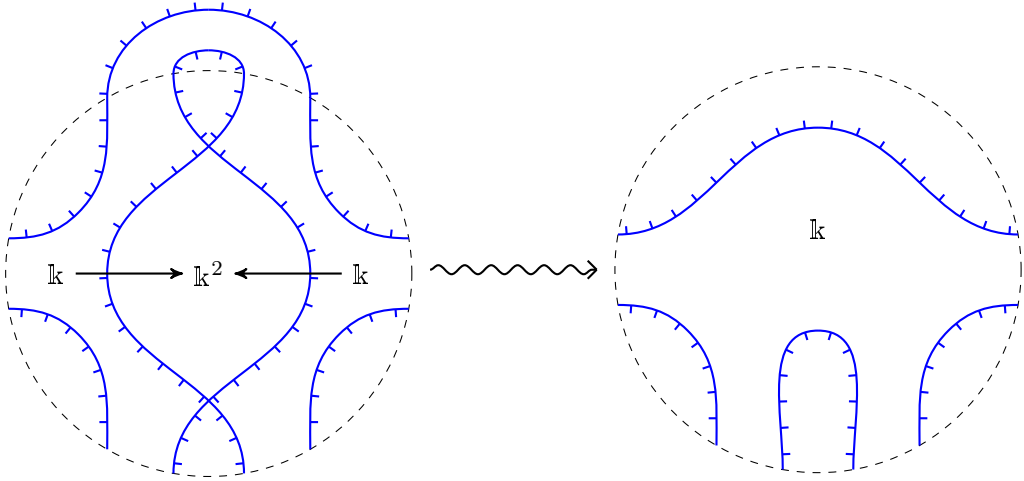

Adding the cap to $\Lambda_r$, however, exactly imposes the relation that the boundary stalks of a sheaf at vertices 1 and 2 are identified under generization maps into the disk.  Since $\Gamma'$ is reduced and boundary stalks $F_1$, $F_2$ have distinct images in $F_x$, the innermost strands of $\Lambda_r$ that are glued together in $\widehat{\Lambda}_r$ are distinct and cross twice.  The isotopy between $\widehat{\Lambda}_r$ and $\Lambda_{r'}$ pulls the inner part of the cap through the picture, a series of Reidemeister-III moves, and then pulls it apart by a Reidemeister-II.  The resulting isotopy isomorphism acts on maximal-rank stalks exactly by the projection map; see \Cref{fig:capping} for a simple example.

Finally, we consider the base cases of $\Gr(1,3)$ and $\Gr(2,3)$.  For the former there is essentially nothing to show, so we only explicitly discuss the latter.  There is only one reduced plabic graph, a trivalent black vertex connected to three white vertices numbered $1$, $2$, and $3$ along the boundary.  Let $X_{12}$, $X_{23}$, $X_{31}$ denote the face holonomies on
$Loc^{fr}_{\:1}(\Gamma)$, i.e. $X_{12}$ is the parallel transport from the trivialized stalk at vertex 1 to the one at vertex 2.  With the perfect orientation such that 1 is a sink, 2 and 3 sources, we have 
\[
\Meas_{\Gamma}^*\Delta_{12} = X_{31}, \quad \Meas_{\Gamma}^*\Delta_{13} = X_{12}^{-1}, \quad \Meas_{\Gamma}^*\Delta_{23} = 1.
\]
On the other hand, $\im_{\Gamma}$ is determined by the invariance of microlocalization under Legendrian isotopy, i.e. \Cref{lem:R3}.  In standard face coordinates we compute that
\[
\im_{\Gamma}^*\Delta_{12} = X_{31}, \quad \im_{\Gamma}^*\Delta_{13} = -X_{12}^{-1}, \quad \im_{\Gamma}^*\Delta_{23} = 1.
\]
This agrees with $\Meas_{\Gamma}$ up to signs, completing the proof.
\end{proof}
\section{Distinguishing fillings} 
\label{sec:rainspir}

\newcommand\Tstrut{\rule{0pt}{2.6ex}}         
\newcommand\Bstrut{\rule[-1.2ex]{0pt}{0pt}}   

A fundamental problem in symplectic geometry is classifying Lagrangians up to Hamiltonian isotopy. Because the sheaf category is invariant under Hamiltonian isotopy it can be used to approach this problem \cite{GKS,Gui,Tam,N}. 
In this section we observe that our results thus far allow us to package information about the classification of exact Lagrangian fillings of Legendrian knots into structures of cluster algebra.  
We also explain how results about alternating Legendrians in $T^\infty \R^2$ (which necessarily have nontrivial winding number around the fibers) lead to results about Legendrians in $\R^3$. This provides, for example, new combinatorial constructions of inequivalent exact fillings of many Legendrian links in $\R^3$, as well as information about how these fillings are related by surgery.

\vspace{2mm}

Fix some Legendrian $\Lambda$, and let $\Lambda_\alpha$ be a collection of
alternating Legendrians equipped with Legendrian isotopies to $\Lambda$.  
Let $L_\alpha$ be the exact filling of $\Lambda$ obtained by Hamiltonian isotopy from the conjugate Lagrangian filling of $\Lambda_\alpha$. 
Assume in addition that the 
various $\Lambda_\alpha$ can be isotoped to each other via square moves. 
Then it follows from our results that the comparison of charts amongst the
$Loc_1(L_\alpha)$ are governed by cluster transformation rules computable 
from the dual quiver to the bicolored graph determining any one of the 
$\Lambda_\alpha$.  These rules, and in particular the question of whether
two such charts are the same, have received extensive study in the combinatorial
literature. We have the following consequence of the quantization results of \cite{GKS} and \cite[\S 3.19]{JinTreumann}.


\begin{proposition}\label{prop:distinguish}
In the above setting, if $L_\alpha$ is Hamiltonian isotopic (fixing the boundary) to
$L_\beta$, then the induced rational morphism $Loc_1(L_\alpha) \dashrightarrow 
Loc_1(L_\beta)$ is a regular isomorphism. 
\end{proposition}

We can apply the above notion to any class of links which have alternating representatives. 
For example, let $\bigcirc \subset T^\infty \R^2$ be a co-circle, $\beta$ a positive braid, and
$\Delta$ the half-twist.
\Cref{con:one} asserts that every word for $\beta$
gives rise to an alternating representative of the Legendrian satellite $\beta\Delta^2
\looparrowright \bigcirc$, but 
there are generally more.  If
$\beta = T_{k,n}$ is the $(k, n)$ torus braid, then $T_{k,n} \Delta^2 = 
T_{k, n+k}$ and 
$T_{k, n+k} \looparrowright \bigcirc$ 
is the braid corresponding to the big positroid stratum of $Gr(k, n+k)$. The enumeration of inequivalent reduced plabic graphs for a fixed positroid was studied in \cite{OPS}. In the case of the big stratum of $Gr(k, n+k)$, they are in bijection with maximal collections of pairwise weakly separated $k$-element subsets of $[1,n]$.  One says two $k$-element subsets $I, J \subset [1,n]$ are weakly separated if they can be cycically shifted so that every element of $I \smallsetminus (I \cap J)$ is less than every element of $J \smallsetminus (I \cap J)$.

\begin{proposition}\label{prop:numforpos}  
The link  $T_{k, n+k} \looparrowright \bigcirc$  admits a collection of exact Lagrangian fillings
into $T^* \R^2$ labeled by maximal pairwise weakly separated $k$-element subsets of $[1,k+n]$.  
No two are Hamiltonian isotopic. 
In particular, if $k = 2$ the number of distinct exact Lagrangian fillings is at least the Catalan number $C_{n}$.
\end{proposition}

\begin{proof}
That distinct maximal weakly separated collections correspond to reduced plabic graphs whose boundary measurement maps have distinct images follows from the results of \cite{MS}.  Loc. cited shows the image of the boundary measurement map is defined (up to a fixed global automorphism, the so-called twist \cite{MaSc}) by the nonvanishing of a collection of Pl\"ucker coordinates associated to the graph as in \cite{Sco}.  Distinct weakly separated collections correspond to distinct collections of Pl\"ucker coordinates, hence their nonvanishing loci are distinct.  The main statement then follows from \Cref{prop:distinguish}, and the Catalan numbers of the $k=2$ case appear since in this case reduced plabic graphs are in correspondence with triangulations of an $n$-gon \cite{FZ}.
\end{proof}

There is analogous notion of weakly separated collection in a more general positroid, and using this \Cref{prop:numforpos} generalizes to any positive annular braid arising from a positroid stratum. We refer to \cite{OPS} for the relevant definitions and results. Except for the open positroid stratum of $\Gr(2,n)$, we do not know of a closed formula for the number of maximal weakly separated collections.

A Legendrian of the form $\beta \looparrowright \bigcirc$ lives  
in $T^\infty \R^2$; the above statement concerns its fillings in $T^* \R^2$.   
However, it is more common to consider Legendrians in the standard contact 
$\R^3$ and their fillings in its symplectization $\R^4$.  
Of course, we can view $\R^3 = J^1(\R)$ as $T^{\infty,-}\R^2$, half the co-circle bundle of
$\R^2$, and correspondingly view $\R^4$ as $T^{-} \R^2 \subset T^* \R^2$.  However, now the 
front projections of the knots will have cusps, a phenomenon we have 
avoided throughout this paper.

Nonetheless, our techniques have implications for this setting. Given a $\Lambda \subset J^1(S^1)$, we write $\Lambda \looparrowright \bigcirc \subset T^\infty \R^2$ and $\Lambda \looparrowright \cunknot \subset T^{\infty,-}\R^2$ for the resulting Legendrian satellites of the co-circle and standard unknot, respectively. From the above contactomorphism one has the following correspondence \cite{Ng-sat,Etn}.

\begin{proposition} \label{prop:trans}
The exact Lagrangian fillings of the satellite 
$\Lambda \looparrowright \bigcirc$ inside $T^* \R^2$ 
are in bijective correspondence with those of $\Lambda \looparrowright \cunknot$
inside $T^{-}\R^2$, and this bijection respects Hamiltonian isotopy. 
\end{proposition}



Proposition \ref{prop:trans} allows us to translate our results to statements about Legendrians 
in the usual contact $\R^3$.  To make this explicit, let us describe
more explicitly the links of the form $\Lambda \looparrowright \cunknot$.  

For a positive braid $\beta$, we write $\beta^\succ$ for the Legendrian with the following 
front diagram:  
\begin{center}  
\includegraphics[scale = .3]{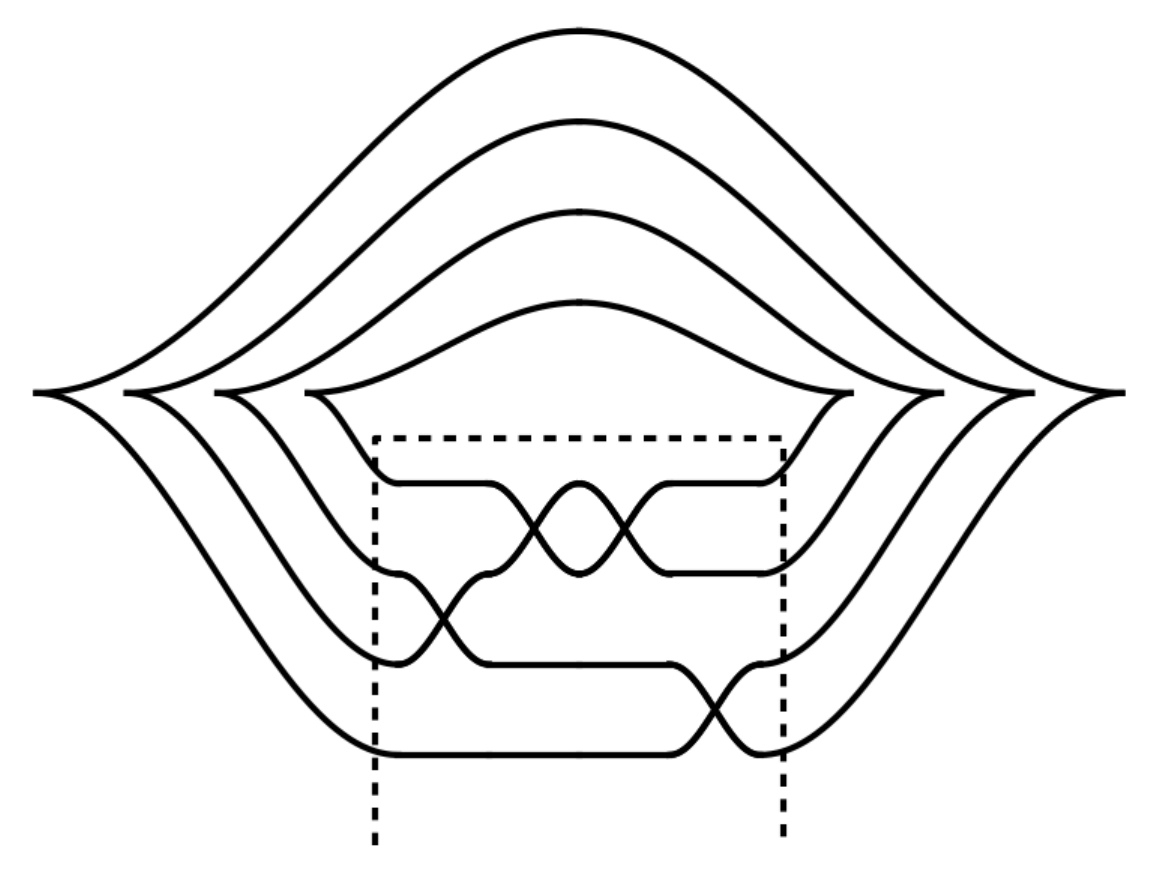} 
\end{center}
Here the braid is $\beta = s_2 s_3^2 s_1$.  In general, we can place any positive braid in the the interior of the dashed region.
We called this
the ``rainbow closure'' in \cite[Sec. 6.2.]{STZ}.  This Legendrian
is a maximal Thurston-Benniquin representative of the braid closure of $\beta$ \cite{Tan},
related to the above satellite construction as follows: 

\begin{proposition} \label{prop:rainsat}
Let $\beta$ be a positive braid.  
Then $\Delta \beta \Delta \looparrowright \cunknot$ and $\beta^\succ$ are Legendrian 
isotopic. 
\end{proposition}
\begin{proof}
According to  \cite{Ng-sat}, the isotopy of Figure \ref{fig:rainsat} relates 
$\beta^\succ$.
to $\Delta \beta \Delta \looparrowright \cunknot$. 
\end{proof}
\begin{figure}[H]
\includegraphics[scale=0.4]{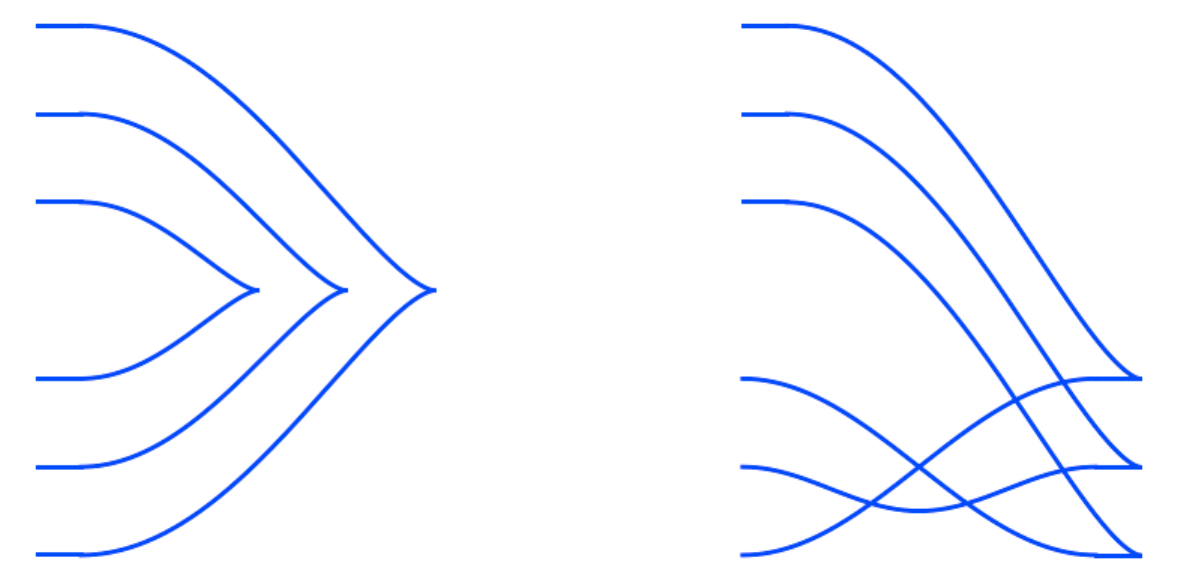}
\caption{\label{fig:rainsat} Equivalence of rainbow closure (left) and satellite  with half twist at cusps (right).}
\end{figure}

By \cite{EH},  
$T_{k,n}^\succ$ is the unique 
Legendrian $(k, n)$ torus knot of maximal Thurston-Bennequin number.

\begin{corollary}\label{cor:count}
The Legendrian $(k, n)$ torus link of maximal Thurston-Bennequin number 
admits a collection of exact Lagrangian fillings 
 labeled by maximal pairwise weakly separated $k$-element subsets of $[1,k+n]$.  
No two are Hamiltonian isotopic. 
In particular, if $k = 2$ the number of distinct exact Lagrangian fillings is at least the Catalan number $C_{n}$.
\end{corollary}

We leave it to the reader to formulate the analogous statement related to more general positive Legendrian braid closures using the notion of weakly separated collections in a positroid \cite{OPS}. We note that in the case of the $(2,n)$ torus link the fillings constructed above are identified in \cite[Sec. 2.3]{TZ} with those constructed in \cite{EHK}. Togeter with the results of this section this proves in particular that these fillings are pairwise not Hamiltonian isotopic, which was left as an open question in \cite{EHK}.

\end{document}